\documentclass[12pt,a4paper]{amsart}
\usepackage{amssymb,amsxtra,eucal}
\usepackage[all,cmtip]{xy}

\calclayout

\newcommand{\+}{\nobreakdash-}
\renewcommand{\:}{\colon}

\newcommand{\rarrow}{\longrightarrow}
\newcommand{\ot}{\otimes}
\newcommand{\ocn}{\odot}
\newcommand{\tim}{\rightthreetimes}

\newcommand{\bu}{{\text{\smaller\smaller$\scriptstyle\bullet$}}}
\newcommand{\lrarrow}{\mskip.5\thinmuskip\relbar\joinrel\relbar\joinrel
 \rightarrow\mskip.5\thinmuskip\relax}

\DeclareMathOperator{\Hom}{Hom}
\DeclareMathOperator{\Ext}{Ext}

\DeclareMathOperator{\Ctrtor}{Ctrtor}
\DeclareMathOperator{\Ann}{Ann}

\DeclareMathOperator{\Id}{Id}
\DeclareMathOperator{\id}{id}
\DeclareMathOperator{\im}{im}
\DeclareMathOperator{\coker}{coker}

\DeclareMathOperator{\CT}{CT}

\newcommand{\Sets}{\mathsf{Sets}}

\newcommand{\rop}{{\mathrm{op}}}

\newcommand{\modl}{{\operatorname{\mathsf{--mod}}}}
\newcommand{\modr}{{\operatorname{\mathsf{mod--}}}}

\newcommand{\contra}{{\operatorname{\mathsf{--contra}}}}
\newcommand{\discr}{{\operatorname{\mathsf{discr--}}}}

\newcommand{\ctra}{{\operatorname{\mathsf{-ctra}}}}

\newcommand{\A}{\mathfrak A}
\newcommand{\B}{\mathfrak B}
\newcommand{\C}{\mathfrak C}
\newcommand{\D}{\mathfrak D}

\newcommand{\F}{\mathfrak F}
\newcommand{\G}{\mathfrak G}
\newcommand{\I}{\mathfrak I}
\newcommand{\J}{\mathfrak J}
\newcommand{\fK}{\mathfrak K}
\newcommand{\fL}{\mathfrak L}
\newcommand{\fM}{\mathfrak M}
\newcommand{\fN}{\mathfrak N}
\newcommand{\fH}{\mathfrak H}
\renewcommand{\P}{\mathfrak P}
\newcommand{\Q}{\mathfrak Q}
\newcommand{\R}{\mathfrak R}
\renewcommand{\S}{\mathfrak S}
\newcommand{\T}{\mathfrak T}
\newcommand{\U}{\mathfrak U}

\newcommand{\m}{\mathfrak m}

\newcommand{\cC}{\mathcal C}
\newcommand{\cD}{\mathcal D}
\newcommand{\cL}{\mathcal L}
\newcommand{\M}{\mathcal M}
\newcommand{\N}{\mathcal N}

\newcommand{\sA}{\mathsf A}
\newcommand{\sB}{\mathsf B}

\newcommand{\boT}{\mathbb T}
\newcommand{\boZ}{\mathbb Z}
\newcommand{\boQ}{\mathbb Q}

\newcommand{\Section}[1]{\bigskip\section{#1}\medskip}
\theoremstyle{plain}
\newtheorem{thm}{Theorem}[section]

\newtheorem{lem}[thm]{Lemma}
\newtheorem{prop}[thm]{Proposition}
\newtheorem{cor}[thm]{Corollary}
\theoremstyle{definition}
\newtheorem{ex}[thm]{Example}

\newtheorem{rem}[thm]{Remark}

\begin{document}

\title{Contramodules over pro-perfect \\ topological rings}

\author{Leonid Positselski}

\address{Institute of Mathematics of the Czech Academy of Sciences \\
\v Zitn\'a~25, 115~67 Praha~1 (Czech Republic); and
\newline\indent
Laboratory of Algebra and Number Theory \\
Institute for Information Transmission Problems \\
Moscow 127051 (Russia)}

\email{positselski@yandex.ru}

\begin{abstract}
 For four wide classes of topological rings $\R$, we show that all flat
left $\R$\+contramodules have projective covers if and only if all flat
left $\R$\+contramodules are projective if and only if all left
$\R$\+contramodules have projective covers if and only if all
descending chains of cyclic discrete right $\R$\+modules terminate if
and only if all the discrete quotient rings of $\R$ are left perfect.
 Three classes of topological rings for which this holds are
the  complete, separated topological associative rings with a base of
neighborhoods of zero formed by open two-sided ideals such that either
the ring is commutative, or it has a countable base of neighborhoods of
zero, or it has only a finite number of semisimple discrete
quotient rings.
 The fourth class consists of all the topological rings with a base of
neighborhoods of zero formed by open right ideals which have
a closed two-sided ideal with certain properties such that the quotient
ring is a topological product of rings from the previous three classes.
 The key technique on which the proofs are based is the contramodule
Nakayama lemma for topologically T\+nilpotent ideals.
\end{abstract}

\maketitle

\tableofcontents

\section{Introduction}
\medskip

\subsection{{}}
 In a classical paper of Bass~\cite{Bas} (based on his Ph.~D.
dissertation), it was shown that every left module over an associative
ring $R$ has a projective cover if and only if every flat left
$R$\+module is projective.
 Such rings were called \emph{left perfect}, and a number of further
equivalent characterizations of them were provided in the paper
(in particular, one of the conditions equivalent to left perfectness
is that all descending chains of principal right ideals in
$R$ terminate).
 Many years later, Enochs conjectured~\cite{Eno}, and subsequently
Bican, El~Bashir, and Enochs proved~\cite{BBE} (see also~\cite{Bash})
that, over any associative ring, all modules have flat covers.
 This assertion became known as the ``flat cover conjecture/theorem''.

 In our recent paper~\cite{PR}, it is shown that, over any complete,
separated topological associative ring with a countable base of
neighborhoods of zero formed by open right ideals, all left
contramodules have flat covers.
 This provides a contramodule analogue of the result of Bican,
El~Bashir, and Enochs.
 In fact, there are two proofs of the existence of flat covers in
their paper~\cite{BBE}, one following the approach of Bican
and El~Bashir, and the other one based on the results of
Eklof and Trlifaj~\cite{ET}.
 Both of these lines of argumentation are extended to contramodules
over topological rings with a countable base of neighborhoods of
zero in the paper~\cite{PR}.

\subsection{{}}
 The aim of the present paper is to extend the results of Bass'
paper~\cite{Bas} to the contramodule realm.
 One reason why this is interesting is because perfect rings are
relatively rare, while pro-perfect topological rings (for many of
which we prove that all contramodules have projective covers and
all flat contramodules are projective) are more numerous.
 In particular, our results apply to the contramodules over all
commutative pro-perfect topological rings.
 This class of topological rings includes complete Noetherian local
commutative rings and the $S$\+completions of $S$\+almost perfect
commutative rings (as defined in our previous paper~\cite{BP}).

 Here a \emph{pro-perfect} topological ring is a complete, separated
topological ring with a base of neighborhoods of zero formed by open
two-sided ideals such that all its discrete quotient rings are perfect.
 In fact, it is only because of possible problems with non-well-behaved
uncountable projective limits that we do not claim applicability of our
results to \emph{all} pro-perfect topological rings.
 Such problems do not exist for topological rings with a countable base
of neighborhoods of zero, and we find them manageable for
commutative topological rings and topological rings having only
a finite number of semisimple discrete quotient rings.

 Thus three classes of topological rings for which our main results hold
are the complete, separated topological associative rings with a base
of neighborhoods of zero formed by open two-sided ideals such that
either (a)~the ring is commutative, or (b)~it has a countable base of
neighborhoods of zero, or (c)~it has only a finite number of semisimple
discrete quotient rings.

\subsection{{}} \label{introd-ideal-H}
 The assumption of existence of a base of neighborhoods of zero formed
by open two-sided ideals is somewhat restrictive, given that
the definition of a left contramodule over a topological ring only
requires a base of neighborhoods of zero consisting of open
right ideals.
 Some, though not all, of our main results are applicable to topological
rings with a base of neighborhoods of zero formed by open right ideals.
 In particular, we prove the existence of projective covers and
projectivity of flat contramodules over many topological rings $\R$
having a topologically T\+nilpotent closed two-sided ideal $\fH$
such that the quotient ring $\R/\fH$ is a product of simple Artinian
discrete rings (endowed with the product topology).

 Here a closed ideal $\fH$ in a topological ring $\R$ is said to be
\emph{topologically left T\+nilpotent} if, for any sequence of
elements $a_1$, $a_2$, $a_3$,~\dots\ in $\fH$, the sequence of products
$a_1$, $a_1a_2$, $a_1a_2a_3$,~\dots\ converges to zero in~$\R$.
 Among our main technical tools, we present two versions of
the Nakayama lemma for topologically T\+nilpotent ideals.
 Firstly, a closed left ideal $\fH\subset\R$ is topologically left
T\+nilpotent if and only if any nonzero discrete right $\R$\+module
has a nonzero submodule annihilated by~$\fH$.
 Secondly, a closed left ideal $\fH$ is topologically left T\+nilpotent
if and only if for any nonzero left $\R$\+contramodule $\C$
the contraaction map $\fH[[\C]]\rarrow\C$ is not surjective.

\subsection{{}}
 Moreover, there is a wider fourth class of topological rings $\R$,
containing the three classes~(a), (b), and~(c), for which we show that
all flat left $\R$\+contramodules have projective covers if and only if
all flat left $\R$\+contramodules are projective if and only if all
left $\R$\+contramodules have projective covers if and only if all
descending chains of cyclic discrete right $\R$\+modules terminate if
and only if all the discrete quotient rings of $\R$ are left perfect.
 This class~(d) consists of complete, separated topological rings with
a base of neighborhoods of zero formed by open right ideals having
a topologically left T\+nilpotent closed ideal $\fK\subset\R$ such
that the quotient ring $\R/\fK$ is a topological product of topological
rings satisfying (a), (b), or~(c).
 Topological rings satisfying the condition~(d) do not need to have
a base of neighborhoods of zero formed by open two-sided ideals.

 We also discuss the notion of the \emph{topological Jacobson radical}
of a topological ring $\R$ with a base of neighborhoods of zero formed
by open right ideals (cf.~\cite{IMR}).
 The topological Jacobson radical of $\R$ is defined as
the intersection of all the open maximal right ideals in~$\R$.
 Generally speaking, the topological Jacobson radical of $\R$ is
a closed two-sided ideal containing the usual Jacobson
radical of the ring $\R$ viewed as an abstract ring.
 For a topological ring $\R$ with a two-sided ideal $\fH$ of the kind
mentioned above in Section~\ref{introd-ideal-H}, both the topological
and the abstract nontopological Jacobson radicals of $\R$ coincide
with~$\fH$.

\subsection{{}}
 This paper was originally motivated by the idea to apply contramodules
in the study of covers and direct limits in module categories, and
more generally, in additive and abelian categories.
 Such applications to covers and direct limits in the categorical
tilting context and beyond, and in particular, for injective homological
ring epimorphisms, are discussed in a companion paper~\cite{BP3}. 
 Furthermore, the main results of this paper provide supporting evidence
for the discussion of conjectures about topologically perfect
topological rings in another companion paper~\cite{PS3}.

 Finally, a general proof of the assertion that a direct limit of
projective contramodules is projective if it has a projective cover
is given in the paper~\cite{BPS}.
 Based on this result, we prove in~\cite{PS3} that, for a complete
separated topological ring $\R$ with a base of neighborhoods of zero
formed by open right ideals, all flat left $\R$\+contramodules are
projective if an only if all flat left $\R$\+contramodules have
projective covers if and only if all left $\R$\+contramodules have
projective covers.

\medskip\noindent
\textbf{Acknowledgment.}
 I~owe a debt of gratitude to Silvana Bazzoni for her invaluable
participation in the early stages of my work on this paper,
which was started when I~was visiting her in Padova in
January--February~2018.
 I~would like to thank Jan Trlifaj\- and Jan \v St\!'ov\'\i\v cek
for helpful discussions.
 Finally, I~wish to thank an anonymous referee for reading
the manuscript carefully and making several helpful suggestions.
 The author was supported by the Israel Science Foundation
grant~\#\,446/15, research plan RVO:~67985840, and
the GA\v CR project 20-13778S.

\Section{Preliminaries on Topological Rings}
\label{preliminaries-secn}

 Throughout this paper, by ``direct limits'' in a category we mean
inductive limits indexed by directed posets.
 Otherwise, these are known as the directed or filtered colimits.

 The material in Sections~\ref{prelim-linear-topol}\+-%
\ref{prelim-discrete-modules} below is well-known (some relevant
references are~\cite[Chapter~VI]{St}, the books~\cite{RD,AGM},
and the paper~\cite{Beil}).
 So is the example in Section~\ref{prelim-endomorphism-ring}
(where the book~\cite[Section~107]{Fuc} can be used as a reference).
 Sections~\ref{prelim-linear-comb}\+-\ref{prelim-reductions} go
back to~\cite[Remark~A.3]{Psemi} and~\cite[Section~1.2]{Pweak};
for later expositions, see~\cite[Sections~2.1 and~2.3]{Prev},
\cite[Sections~1.2 and~5]{PR}, \cite[Sections~6\+-7]{PS},
and~\cite[Section~2]{Pcoun}.
 The counterexample in Section~\ref{prelim-strongly-closed-subgroups}
comes from~\cite{RD} or~\cite{AGM}, while the rest of the material of
Sections~\ref{prelim-strongly-closed-subgroups}\+-%
\ref{prelim-strongly-closed-ideals} may be somewhat new
(it is implicit in~\cite[Sections~1.3 and~B.4]{Pweak}).

 For a basic background discussion of topological abelian groups and
vector spaces with linear topology, going into many details and
correcting some errors in~\cite{Beil}, see the very recent
preprint~\cite{Pextop}.

\subsection{Abelian groups with linear topology}
\label{prelim-linear-topol}
 A ``topological abelian group'' in this paper will always mean
a topological abelian group with a base of neighborhoods of zero
formed by open subgroups.
 Any nonempty collection of subgroups $B$ in an abelian group $A$
such that for any two subgroups $U'\in B$ and $U''\in B$ there exists
a subgroup $U\in B$, \ $U\subset U'\cap U''$, forms a base of
neighborhoods of zero in a (uniquely defined) topology on $A$
compatible with the abelian group structure.

 The \emph{completion} $\A=A_B\sphat\,$ of a topological
abelian group $A$ is the abelian group $\A=\varprojlim_{U\in B}A/U$
endowed with the \emph{projective limit topology}, in which a base
of neighborhoods of zero $\B$ in $\A$ is formed by the kernels
$\U=U\sphat\,$ of the projection maps $\A\rarrow A/U$.
 The topological group $\A=A_B\sphat$ does not depend on the choice
of a particular base of neigborhoods of zero $B$ in a topological
group~$A$.

 A topological abelian group $A$ is said to be \emph{separated} if
the completion map $A\rarrow\A$ is injective, and \emph{complete}
if it is surjective.
 The projection maps $\A\rarrow A/U$ induce isomorphisms
$\A/\U\rarrow A/U$, so the completion $\A$ of a topological abelian
group $A$ is always separated and complete (in the projective limit
topology of~$\A$).
 The completion map $A\rarrow\A$ is an isomorphism of topological
abelian groups (i.~e., an additive homeomorphism) whenever $A$ is
separated and complete.

\subsection{Subgroup and quotient group topologies}
\label{prelim-quotient-topologies}
 Let $A$ be a topological abelian group and $A'\subset A$ be a subgroup.
 Then the group $A'$ can be endowed with the topology induced from~$A$.
 If $B$ is a collection of subgroups forming a base of neighborhoods of
zero in $A$, then the collection of subgroups
$B'=\{A'\cap U\mid U\in B\}$ forms a base of neighborhoods of zero
in~$A'$.

\begin{lem} \label{subgroup-topology-lem}
\textup{(a)} If a topological abelian group $A$ is separated, then
any subgroup $A'\subset A$ is separated in the induced topology. \par
\textup{(b)} Let\/ $\A$ be a complete, separated topological
abelian group.
 Then a subgroup\/ $\A'\subset\A$ is complete in the induced topology
if and only if it is closed in~$\A$. \par
\textup{(c)} More generally, if\/ $\A$ is separated and complete, then
for any subgroup $A'\subset\A$ the induced map between the completions
$A'{}\sphat\,\rarrow\A\sphat=\A$ provides a topological group
isomorphism of the completion $A'{}\sphat\,$ of the group  $A'$\
with the closure\/ $\A'\subset\A$ of the subgroup $A'\subset\A$
(where the topology on $A'$ and\/ $\A'$ is induced from\/~$\A$).
\end{lem}

\begin{proof}
 See, e.~g., \cite[Lemma~1.1]{Pextop}.
\end{proof}

 Let $A$ be a topological abelian group, $A'\subset A$ be a subgroup,
and $A''=A/A'$ be the quotient group.
 Then the group $A''$ can be endowed with the quotient topology, in
which a subset of $A''$ is open if and only if its preimage in $A$ is.
 If $B$ is a collection of subgroups forming a base of neighborhoods of
zero in $A$, then the collection of subgroups $B''=\{(A'+U)/A'\mid
U\in B\}$ forms a base of neighborhoods of zero in~$A''$.

\begin{lem} \label{quotient-group-topology-lem}
\textup{(a)} If a topological abelian group $A$ is separated, then
the quotient group $A/A'$ is separated in the quotient topology
if and only if the subgroup $A'$ is closed in~$A$. \par
\textup{(b)} Let\/ $\A$ be a complete, separated topological abelian
group with a countable base of neighborhoods of zero formed by
open subgroups, and let\/ $\A'\subset\A$ be a closed subgroup.
 Then the quotient group\/ $\A''=\A/\A'$ is complete in
the quotient topology.
\end{lem}

\begin{proof}
 We will only prove part~(b).
 Let $\B$ be a countable base of neighborhoods of zero consisting of
open subgroups in~$\A$.
 Then for any $\U\in\B$ we have a short exact sequence of abelian groups
$$
 0\lrarrow \A'/\U'\lrarrow\A/\U\lrarrow\A''/\U''\lrarrow0,
$$
where $\U'=\A'\cap\U$ and $\U''=(\A'+\U)/\A'$.
 Now $(\A'/\U')_{\U\in\B}$ is a projective system of abelian groups and
surjective morphisms between them, indexed by a countable directed
poset~$\B$.
 Hence the sequence remains exact after the passage to the projective
limits,
$$
 0\lrarrow \A'{}_{\B'}\sphat\,\lrarrow\A_\B\sphat\,\lrarrow
 \A''{}_{\B''}\sphat\,\lrarrow0,
$$
where $\B'=\{\U'\mid\U\in\B\}$ and $\B''=\{\U''\mid\U\in\B\}$.
 By assumption, we have $\A=\A_\B\sphat\,$; by
Lemma~\ref{subgroup-topology-lem}(b), $\A'=\A'{}_{\B'}\sphat\,$.
 Thus the completion map $\A''\rarrow\A''{}_{\B''}\sphat\,$ is
bijective, that is, the topological abelian group $\A''$ is
(separated and) complete.

 For a generalization, see~\cite[Proposition~1.4]{Pextop}.
\end{proof}

 Without the countability assumption, the assertion of
Lemma~\ref{quotient-group-topology-lem}(b) is \emph{not} true.
 See Section~\ref{prelim-strongly-closed-subgroups} for further
discussion and counterexamples.

\subsection{Topological rings}  \label{prelim-topol-rings}
 In this paper, the word ``ring'' means ``an associative ring with
unit'' by default.
 Unless otherwise mentioned, all ring homomorphisms are supposed
to preserve units, and all modules are presumed to be unital.
 When considering rings without unit or subrings without unit, as
we will at some point in Section~\ref{t-nilpotent-secn}, we will always
explicitly refer to them as being ``without unit''.

 Given an (associative and unital) ring $R$, we denote the abelian
category of (arbitrary unital) left $R$\+modules by $R\modl$ and
the abelian category of right $R$\+modules by $\modr R$.
 The ring with the opposite multiplication to a ring $R$ is denoted by
$R^\rop$.

 In this paper we are interested in topological associative rings $R$
such that open right ideals $I\subset R$ form a base of neighborhoods
of zero in~$R$.
 A nonempty collection of right ideals $B$ in a ring $R$ forms a base of
neighborhoods of zero in a topology compatible with the ring structure
on $R$ if and only if it satisfies the following conditions
(cf.~\cite[Section~VI.4]{St} and~\cite[Remark~1.1(ii), Claim~1.4,
and Lemma~1.4]{Beil}):
\begin{enumerate}
\renewcommand{\theenumi}{\roman{enumi}}
\item for any two right ideals $I'$ and $I''\in B$, there exists
a right ideal $J\in B$ such that $J\subset I'\cap I''$; and
\item for any right ideal $I\in B$ and any element $r\in R$,
there exists a right ideal $J\in B$ such that $rJ\subset I$.
\end{enumerate}

 The \emph{completion} $\R=R_B\sphat\,$ of a topological ring $R$
is the abelian group $\R=\varprojlim_{I\in B}R/I$ endowed with
the projective limit topology (in which a base of neighborhoods
of zero $\B$ in $\R$ is formed by the kernels $\I=I\sphat\,\subset\R$
of the projection maps $\R\rarrow R/I$).
 One readily checks, using the conditions (i) and~(ii), that there
exists a unique associative ring structure on $\R$ that is continuous
with respect to the projective limit topology and such that
the natural map $R\rarrow\R$ is a ring homomorphism.
 Given two elements $r'=(r'_I)_{I\in B}$ and $r''=(r''_I)_{I\in B}\in\R$,
in order to compute the $I$\+component $r_I$ of the product
$r=(r_I)_{I\in B}\in\R$, one has to find an open right ideal $J\in B$
such that $J\subset I$ and $r'_IJ\subset I$; then one can set
$r_I=r'_Ir''_J+I$.

 The open subsets $\I=I\sphat\,$ are right ideals in~$\R$, so
the topological ring $\R$ has a base of neighborhoods of zero
consisting of open right ideals.
 When the base of neighborhoods of zero $B$ in $R$ consists of open
two-sided ideals, the topological ring $\R$ can be simply defined as
the projective limit of (discrete) rings~$R/I$.
 Then the open subsets $\I=I\sphat\,\subset\R$ are two-sided ideals.

 A topological ring $R$ is said to be \emph{separated} (resp.,
\emph{complete}) if it is separated (resp., complete) as a topological
abelian group.

\subsection{Discrete modules}  \label{prelim-discrete-modules}
 Let $R$ be a topological ring with a base of neighborhoods of zero
formed by open right ideals.
 A right $R$\+module $\N$ is said to be \emph{discrete} if for every
element $b\in\N$ the annihilator $\Ann_R(b)=\{r\in R\mid br=0\}$
is an open right ideal in~$\R$.
 The annihilator of an element in a right $R$\+module is always a right
ideal in $R$, so topological rings with a base of neighborhoods of zero
formed by open right ideals are a natural setting for considering
discrete right modules.

 The full subcategory of discrete right $R$\+modules
$\discr R\subset\modr R$ is a hereditary pretorsion class in the abelian
category of right $R$\+modules $\modr R$, and all hereditary pretorsion
classes in $\modr R$ appear in this way~\cite[Lemma~VI.4.1 and
Proposition~VI.4.2]{St}.
 Viewed as an abstract category, the category $\discr R$ is
a Grothendieck abelian category.
 So, in particular, the abelian category $\discr R$ is complete and
cocomplete, has exact direct limits, and an injective cogenerator.

 One readily checks that any discrete right $R$\+module has a unique
discrete right $\R$\+module structure compatible with its $R$\+module
structure (where $\R=R\sphat\,$ denotes the completion of
the topological ring~$R$).
 Thus the abelian categories of discrete right $R$\+modules and
discrete right $\R$\+modules are naturally equivalent (in fact,
isomorphic), $\discr R\cong\discr\R$.

\subsection{Convergent formal linear combinations}
\label{prelim-linear-comb}
 Given an abelian group $A$ and a set $X$, we denote by
$A[X]=A^{(X)}$ the direct sum of $X$ copies of the abelian group $A$,
viewed as the group of all finite formal linear combinations
$\sum_{x\in X}a_xx$ of elements of $X$ with the coefficients in~$A$.
 A formal linear combination $\sum_{x\in X} a_xx$ belongs to $A[X]$
if and only if the set of all indices $x\in X$ for which $a_x\ne0$
is finite.

 Given a separated and complete topological abelian group $\A$ with
a base of neighborhoods of zero $\B$ consisting of open subgroups,
and a set $X$, we denote by $\A[[X]]$ the abelian group
$\varprojlim_{\U\in\B} (\A/\U)[X]$.
 Clearly, the group $\A[[X]]$ does not depend on the choice of
a particular base of neighborhoods of zero $\B$ in~$\A$. 
 We interpret $\A[[X]]$ as the group of all infinite formal linear
combinations $\sum_{x\in X}a_xx$ of elements of $X$ with 
the coefficients in $\A$ forming an \emph{$X$\+indexed family
of elements in $\A$ converging to zero in the topology of\/~$\A$}.
 This means that the subgroup $\A[[X]]\subset\A^X$ consists of
all the infinite formal linear combinations $\sum_{x\in X}a_xx$ such
that, for every open subgroup $\U\subset\A$, one has $a_x\in\U$
for \emph{all but a finite subset of indices} $x\in X$.

 The map assigning to a set $X$ the abelian group $\A[[X]]$ extends
naturally to a covariant functor from the category of sets to
the category of abelian groups.
 Given a map of sets $f\:X\rarrow Y$, one defines the induced map
$\A[[f]]\:\A[[X]]\rarrow\A[[Y]]$ by the rule
$\sum_{x\in X}a_xx\longmapsto\sum_{y\in Y}
\bigl(\sum_{f(x)=y}a_x\bigr)y$,
where the sum of elements~$a_x$ in the parentheses is understood
as the limit of finite partial sums in the topology of~$\A$.
 Such a limit is unique and exists because the topological abelian
group $\A$ is separated and complete, while the family of elements
$(a_x)_{x\in X}$, and consequently its subfamily indexed by all
$x\in X$ with $f(x)=y$ for a fixed $y\in Y$, converges to zero in~$\A$.

\subsection{The monad structure}
 Let $\R$ be a complete, separated topological ring with a base of
neighborhoods of zero formed by open right ideals.
 Let us consider the functor $X\longmapsto\R[[X]]$ as taking values
in the category of sets; so it becomes an endofunctor
$\boT_\R=\R[[{-}]]\:\Sets\rarrow\Sets$.
 The key observation is that the functor $\boT_\R$ has a natural
structure of a \emph{monad} on the category of
sets~\cite[Remark~A.3]{Psemi}, \cite[Section~1.2]{Pweak},
\cite[Section~2.1]{Prev}, \cite[Sections~1.1\+-1.2 and~5]{PR},
\cite[Section~1]{Pper}, \cite[Section~6]{PS}.

 This means that the functor $\boT_\R$ is endowed with natural
transformations $\epsilon\:\Id\rarrow\boT_\R$ and
$\phi\:\boT_\R\circ\boT_\R\rarrow\boT_\R$ satisfying the monad
equations (of unitality and associativity).
 The monad unit $\epsilon_X\:X\rarrow\R[[X]]$ is the ``point measure''
map, assigning to an element $x_0\in X$ the (finite) formal linear
combination $\sum_{x\in X}r_xx\in\R[[X]]$, where $r_{x_0}=1$
and $r_x=0$ for all $x\ne x_0$.
 The monad multiplication $\phi_X\:\R[[\R[[X]]]]\rarrow\R[[X]]$ is
the ``opening of parentheses'' map, assigning a formal linear
combination to a formal linear combination of formal
linear combinations.

 Given a set $X$ and an element $\mathbf r\in\R[[\R[[X]]]]$, computing
the element $\phi_X(\mathbf r)\in\R[[X]]$ involves opening
the parentheses, computing the products of pairs of elements in
the ring $\R$, and then computing the infinite sums.
 The coefficient~$t_x$ of $\phi(\mathbf r)=\sum_{x\in X}t_xx$ at
an element $x\in X$ is an infinite sum of products of pairs of elements
in $\R$, understood as the limit of finite partial sums in the topology
of~$\R$.
 Thus it is crucial for the definition of~$\phi_X$ that $\R$ is
separated and complete, and that all the infinite sums involved
converge.
 The latter is guaranteed by the assumption that open right ideals form
a base of neighborhoods of zero in~$\R$.

 Indeed, let $Y$ denote the set $\R[[X]]$; then we have
$\mathbf r=\sum_{y\in Y}r_yy$ for some $r_y\in\R$,
and $y=\sum_{x\in X}s_{y,x}x$ for all $y\in Y$ and some
$s_{y,x}\in\R$.
 For any $x\in X$, the coefficient $t_x$ is to be computed as
$t_x=\sum_{y\in Y}r_ys_{y,x}\in\R$.
 In order to show that this sum converges in the topology of $\R$, we
have to check that, for every open right ideal $\I\subset\R$,
the product $r_ys_{y,x}$ belongs to $\I$ for all but a finite set of
indices $y\in Y$.
 Now, one has $r_ys_{y,x}\in\I$ whenever $r_y\in\I$; and there is
only a finite set of indices~$y$ with $r_y\notin\I$, because
$\mathbf r\in\R[[Y]]$.
 So the coefficient $t_x\in\R$ is well-defined for every $x\in X$.
 In order to check that $\sum_{x\in X}t_xx\in\R[[X]]$, that is
$t_x\in\I$ for all but a finite set of indices $x\in X$, one has to
use the condition~(ii) from Section~\ref{prelim-topol-rings}.

\subsection{Contramodules}  \label{prelim-contramodules}
 Let $\R$ be a complete, separated topological ring with a base of
neighborhoods of zero formed by open right ideals.
 By the definition, a left $\R$\+contramodule is an algebra or module
(depending on the terminology) over the monad $\boT_\R\:X
\longmapsto\R[[X]]$ on the category of sets.

 This means that a left $\R$\+contramodule $\C$ is a set endowed
with a map of sets $\pi_\C\:\R[[\C]]\rarrow\C$, called 
the \emph{left contraaction map}.
 The map~$\pi_\C$ must satisfy the equations of \emph{contraunitality},
telling that the composition $\pi_\C\epsilon_\C$ with the monad unit
map~$\epsilon_\C$ is the identity map~$\id_\C$,
$$
 \C\rarrow\R[[\C]]\rarrow\C,
$$
and \emph{contraassociativity}, asserting that the two maps
$\R[[\R[[\C]]]]\rightrightarrows\R[[\C]]$, one of which is the monad
multiplication map~$\phi_\C$ and the other one is the map
$\R[[\pi_\C]]$ induced by~$\pi_\C$, should have equal compositions
with the contraaction map~$\pi_\C$,
$$
 \R[[\R[[\C]]]]\rightrightarrows\R[[\C]]\rarrow\C.
$$
 We denote the category of left $\R$\+contramodules by
$\R\contra$.

 In particular, any associative ring $R$ can be considered as
a topological ring with the discrete topology.
 In this case, we have $R[[X]]=R[X]$, and a left $R$\+contramodule
is the same thing as a left $R$\+module~\cite[Section~6.1]{PS}.
 So the above definition of an $\R$\+contramodule, restricted to
the particular case when the topological ring $\R=R$ is discrete,
provides a fancy way to define the familiar notion of a module over
an associative ring.

 Now, for any complete, separated topological associative ring $\R$
with a base of neighborhoods of zero formed by open right ideals,
and for any left $\R$\+contramodule $\C$, one can compose
the contraaction map $\pi_\C\:\R[[\C]]\rarrow\C$ with the identity
embedding $\R[\C]\rarrow\R[[\C]]$ of the set of all finite formal linear
combinations into the set of all convergent infinite ones.
 This defines a natural structure of an algebra/module over the monad
$X\longmapsto\R[X]$ on the set~$\C$, which means
a left $\R$\+module structure.
 Thus all left $\R$\+contramodules have underlying structures of left
modules over the ring $\R$, viewed as an abstract (nontopological) ring.
 We have constructed the forgetful functor $\R\contra\rarrow\R\modl$.
 In particular, it means that all left $\R$\+contramodules, which were
originally defined as only sets endowed with a contraaction map, are 
actually abelian groups.

 The category $\R\contra$ is abelian~\cite[Lemma~1.1]{Pper}, and
the forgetful functor $\R\contra\rarrow\R\modl$ is exact.
 The category $\R\contra$ is also complete and cocomplete, with
the forgetful functor $\R\contra\rarrow\R\modl$ preserving
infinite products (but not coproducts).
 Consequently, infinite products (but, generally speaking, not
coproducts) are exact in $\R\contra$.
 Given two left $\R$\+contramodules $\C$ and $\D$, we denote by
$\Hom^\R(\C,\D)$ the abelian group of morphisms $\C\rarrow\D$
in $\R\contra$.

 For any set $X$, the map $\pi_{\R[[X]]}=\phi_X$ endows
the set/abelian group $\R[[X]]$ with the structure of a left
$\R$\+contramodule.
 It is called the \emph{free} left $\R$\+contramodule generated by
a set~$X$.
 For any left $\R$\+contramodule $\C$, morphisms $\R[[X]]\rarrow
\C$ in the category $\R\contra$ correspond bijectively to maps of
sets $X\rarrow\C$,
$$
 \Hom^\R(\R[[X]],\C)\cong\Hom_\Sets(X,\C).
$$
 Hence free left $\R$\+contramodules are projective objects of
the category $\R\contra$.
 There are also enough of them: for any left $\R$\+contramodule $\C$,
the contraaction map $\pi_\C\:\R[[\C]]\rarrow\C$ is
an $\R$\+contramodule morphism presenting $\C$ as a quotient
contramodule of the free left $\R$\+contramodule $\R[[\C]]$.
 So the abelian category $\R\contra$ has enough projectives, and
a left $\R$\+contramodule is projective if and only if it is
a direct summand of a free one.

\subsection{Contratensor product}  \label{prelim-contratensor}
 As above, we denote by $\R$ a complete, separated topological ring
with a base of neighborhoods of zero formed by open right ideals.
 Some of the simplest examples of non-free left $\R$\+contramodules
are obtained by dualizing discrete right $\R$\+modules.

 Let $A$ be an associative ring, and let $\N$ be an $A$\+$\R$\+bimodule
whose right $\R$\+module structure is that of a discrete right
$\R$\+module.
 Let $V$ be a left $A$\+module.
 Then the induced left $\R$\+module structure of the abelian group
$\Hom_A(\N,V)$ extends naturally to a left $\R$\+contramodule structure.
 Indeed, to construct a left $\R$\+contraaction map for the set
$\D=\Hom_A(\N,V)$, consider an element $\mathbf r=
\sum_{d\in\D}r_dd\in\R[[\D]]$.
 Set $f=\pi_\D(\mathbf r)$ to be the left $A$\+module map
$f\:\N\rarrow V$ taking any element $b\in\N$ to the element
$\sum_{d\in\D}d(br_d)\in V$,
$$
 \pi_\D\left(\sum\nolimits_{d\in\D}r_dd\right)(b) =
 \sum\nolimits_{d\in\D} d(br_d),
$$
where the sum in the right-hand side is finite, because one has
$br_d=0$ for all but a finite number of elements $d\in\D$.
 Indeed, the annihilator ideal $\Ann_\R(b)\subset\R$ is open by
assumption, and consequently, it has to contain the coefficient~$r_d$
for all but a finite number of elements $d\in\D$.

 Let $\N$ be a discrete right $\R$\+module and $\C$ be a left
$\R$\+contramodule.
 The \emph{contratensor product} $\N\ocn_\R\C$ is an abelian group
defined as the cokernel of (the difference of) a natural pair of abelian
group homomorphisms
$$
 \N\ot_\boZ\R[[\C]]\rightrightarrows\N\ot_\boZ\C.
$$
 Here one of the maps $\N\ot_\boZ\R[[\C]]\rarrow\N\ot_\boZ\C$ is
just the map $\N\ot\pi_\C$ induced by the contraaction map
$\pi_\C\:\R[[\C]]\rarrow\C$, while the other map is the composition
$\N\ot_\boZ\R[[\C]]\rarrow\N[\C]\rarrow\N\ot_\boZ\C$, where
(following our general notation system) $\N[\C]$ denotes the group
of all finite formal linear combinations of elements of $\C$ with
the coefficients in~$\N$.
 The map $\N\ot_\boZ\R[[\C]]\rarrow\N[\C]$, induced by
the right action map $\N\ot_\boZ\R\rarrow\N$, is well-defined due to
the assumption that $\N$ is a discrete right $\R$\+module.
 The map $\N[\C]\rarrow\N\ot_\boZ\C$ is just the obvious one,
taking a finite formal linear combination $\sum_{c\in\C} b_cc$,
\ $b_c\in\N$, to the tensor $\sum_{c\in\C}b_c\ot c$.

 Essentially by the definition, the contratensor product is
a quotient group of the tensor product: for any discrete right
$\R$\+module $\N$ and any left $\R$\+contramodule $\C$, there is
a natural surjective map of abelian groups
$$
 \N\ot_\R\C\twoheadrightarrow\N\ocn_\R\C.
$$

 For any $A$\+$\R$\+bimodule $\N$ whose right $\R$\+module structure
is that of a discrete right $\R$\+module, any left $\R$\+contramodule
$\C$, and any abelian group $V$, there is a natural adjunction
isomorphism of abelian groups~\cite[Section~5]{PR}
$$
 \Hom_A(\N\ocn_\R\C,\>V)\cong\Hom^\R(\C,\Hom_A(\N,V)).
$$
 The functor of contratensor product
$\ocn_\R\:\discr\R\times\R\contra\rarrow\boZ\modl$
preserves colimits (i.~e., is right exact and preserves coproducts) in
both its arguments.
 For any discrete right $R$\+module $\N$ and any set $X$, there is
a natural isomorphism of abelian groups
$$
 \N\ocn_\R\R[[X]]\cong\N[X].
$$

\subsection{Change of scalars} \label{prelim-change-of-scalars}
 Let us start with a discussion of change of scalars for homomorphisms
of abstract associative rings (without any topology) before passing
to contramodules and discrete modules over topological rings.

 Let $g\:R\rarrow S$ be a homomorphism of associative rings.
 Then any $S$\+module can be endowed with an $R$\+module structure
via~$g$, so we have the functor of restriction of scalars
$g_*\:S\modl\rarrow R\modl$.
 The functor~$g_*$ is exact and faithful, and it preserves both
the infinite coproducts and products.

 The functor $g_*\:S\modl\rarrow R\modl$ has adjoints on both sides.
 The functor of extension of scalars $g^*\:R\modl\rarrow S\modl$
given by the rule $g^*(M)=S\ot_RM$ is left adjoint to~$g_*$, while
the functor of coextension of scalars $g^!\:R\modl\rarrow S\modl$
given by the formula $g^!(M)=\Hom_R(S,M)$ is right adjoint to~$g_*$.

 When one passes to contramodules and discrete modules over topological
rings, this picture of three functors splits in two halves.
 Let $f\:\R\rarrow\S$ be a continuous homomorphism of
complete, separated topological rings, each of them having a base of
neighborhoods of zero formed by open right ideals.
 Then for any set $X$ there is the induced map of sets/abelian groups
$f[[X]]\:\R[[X]]\rarrow\S[[X]]$.

 Let $\C$ be a left $\S$\+contramodule.
 Composing the map $f[[\C]]\:\R[[\C]]\rarrow\S[[\C]]$ with
the contraaction map $\pi_\C\:\S[[\C]]\rarrow\C$, we obtain a map
$\R[[\C]]\rarrow\C$ defining a left $\R$\+contramodule structure
on the set~$\C$.
 We have constructed an exact, faithful functor of
\emph{contrarestriction of scalars} $f_\sharp\:\S\contra\rarrow
\R\contra$ forming a commutative square diagram with the forgetful
functors $\R\contra\rarrow\R\modl$, \ $\S\contra\rarrow\S\modl$ and
the restriction-of-scalars functor $\S\modl\rarrow\R\modl$.
 The functor~$f_\sharp$ also preserves infinite products.

 The functor~$f_\sharp$ has a left adjoint functor of
\emph{contraextension of scalars} $f^\sharp\:\R\contra\allowbreak
\rarrow\S\contra$.
 To construct the functor~$f^\sharp$, one can first define it on free
left $\R$\+con\-tramodules by the rule $f^\sharp(\R[[X]])=\S[[X]]$ for
all sets~$X$, and then extend to a right exact functor on the whole
category $\R\contra$.
 As any left adjoint functor, the functor~$f^\sharp$ preserves
coproducts.

 Notice that the forgetful functor $\S\contra\rarrow\S\modl$ does
\emph{not} preserve coproducts, generally speaking; and accordingly
the functor of contrarestriction of scalars $f_\sharp\:\S\contra
\rarrow\R\contra$ does not preserve coproducts.
 (For a counterexample, take $\R=\S$ with the given topology on
$\S$ and the discrete topology on $\R$, and let $f\:\R\rarrow\S$
be the identity map; or alternatively, take $\R=\boZ$ with
the discrete topology and let $f\:\R\rarrow\S$ be the unique ring
homomorphism.)
 Hence the functor~$f_\sharp$ does \emph{not} have a right adjoint,
in general.

 Similarly, the map~$f$ endows any discrete right $\S$\+module with
a discrete right $\R$\+module structure.
 In other words, the conventional functor of restriction of scalars
$\modr\S\rarrow\modr\R$ takes discrete right $\S$\+modules to
discrete right $\R$\+modules.
 So we have an exact, faithful functor of restriction of scalars
$f_\diamond\:\discr\S\rarrow\discr\R$.
 The functor~$f_\diamond$ also preserves infinite coproducts.

 As any colimit-preserving functor between Grothendieck abelian
categories, the functor~$f_\diamond$ has a right adjoint functor of
\emph{coextension of scalars} $f^\diamond\:\discr\R\rarrow\discr\S$.
 The functor~$f^\diamond$ is left exact and preserves products.

 Notice that the inclusion functor $\discr\S\rarrow\modr\S$ does
\emph{not} preserve products, generally speaking.
 In fact, the product of a family of objects $\M_\alpha$ in
$\discr\S$ can be constructed as the maximal discrete $\S$\+submodule
of the product of the $\S$\+modules $\M_\alpha$ taken in $\modr\S$.
 Accordingly, the functor of restriction of scalars
$f_\diamond\:\discr\S\rarrow\discr\R$ does not preserve products,
in general.
 (Once again, for a counterexample it suffices to take $\R=\S$ with
the given topology on $\S$ and the discrete topology on~$\R$.)
 Hence the functor~$f_\diamond$ usually does \emph{not} have
a left adjoint.

 For any discrete right $\S$\+module $\M$ and any left
$\R$\+contramodule $\C$ there is a natural isomorphism of abelian groups
$$
 f_\diamond(\M)\ocn_\R\C\cong\M\ocn_\S f^\sharp(\C).
$$

\subsection{Reductions modulo ideals}
\label{prelim-reductions}
 Let $\R$ be a complete, separated topological ring with a base of
neighborhoods of zero formed by open right ideals.

 Let $\N$ be a discrete right $\R$\+module, and $A\subset\R$ be
an additive subgroup in~$\R$.
 Then we denote by $\N_A\subset\N$ the additive subgroup in $\N$
consisting of all the elements $b\in\N$ such that $br=0$
for all $r\in A$.
 If $J$ is a left ideal in $\R$, then the subgroup $\N_J$
is an $\R$\+submodule in~$\N$.

 Let $R$ be an associative ring, $C$ be a left $R$\+module, and
$A\subset R$ be an additive subgroup in~$R$.
 As usually, we will denote by $AC\subset C$ the subgroup in $C$
spanned by all the elements $ac$ with $a\in A$ and $c\in C$.
 Clearly, one has $AC=A(RC)=(AR)C$, where $AR\subset R$ is
the right ideal generated by~$A$.
 When $J\subset R$ is a left ideal, the subgroup $JC$ is
an $R$\+submodule in~$C$.

 Let $\C$ be a left $\R$\+contramodule, and $\A\subset\R$ be
a closed additive subgroup in~$\R$.
 Then $\A$ is a complete, separated topological abelian group in
the topology induced from $\R$, and for any set $X$ the group $\A[[X]]$
is a subgroup in $\R[[X]]$.
 Following~\cite[Remark~A.3]{Psemi}, \cite[Section~1.3]{Pweak},
\cite[Section~D.1]{Pcosh}, \cite[Section~5]{PR}, we will denote by
$\A\tim\C\subset\C$ the image of the map $\A[[\C]]\rarrow\C$
obtained by restricting the contraaction map $\R[[\C]]\rarrow\C$
to the subgroup $\A[[\C]]\subset\R[[\C]]$.
 Clearly, one has $\A\C\subset\A\tim\C$.

 Let $\J\subset\R$ be a closed left ideal.
 Then the composition of the identity embedding $\R[[\J[[X]]]]\rarrow
\R[[\R[[X]]]]$ with the map $\phi_X\:\R[[\R[[X]]]]\rarrow\R[[X]]$
takes values inside the subset $\J[[X]]\subset\R[[X]]$, so $\J[[X]]$
is a left $\R$\+subcontramodule of $\R[[X]]$.
 It follows that, for any left $\R$\+contramodule $\C$, the subgroup
$\J\tim\C\subset\C$ is the image of the composition of left
$\R$\+contramodule morphisms $\J[[\C]]\rarrow\R[[\C]]\rarrow\C$.
 Hence $\J\tim\C$ is an $\R$\+subcontramodule in~$\C$.

 For any closed right ideal $\J\subset\R$ and any set $X$, one has
$$
 \J\tim(\R[[X]])=\J[[X]]\subset\R[[X]].
$$

 Let $\I\subset\R$ be an open right ideal.
 Then the right $\R$\+module $\R/\I$ is discrete, and for any
left $\R$\+contramodule $\C$ there is a natural isomorphism of
abelian groups
\begin{equation} \label{reduction-as-contratensor}
 (\R/\I)\ocn_\R\C\cong\C/(\I\tim\C).
\end{equation}
 The isomorphism~\eqref{reduction-as-contratensor} is a part of
the commutative square diagram
\begin{equation} \label{reduction-contratensor-diagram}
\begin{gathered}
 \xymatrix{
 (\R/\I)\ot_\R\C \ar@{=}[r] \ar@{->>}[d] & \C/\I\C \ar@{->>}[d] \\
 (\R/\I)\ocn_\R\C \ar@{=}[r] & \C/(\I\tim\C)
 }
\end{gathered}
\end{equation}
with the familiar natural isomorphism in the upper horizontal line and
the natural surjective maps of abelian groups shown by the vertical
arrows.
 To convince oneself of the existence of
the isomorphism~\eqref{reduction-as-contratensor}, one observes that
the two kernels of the surjective maps in the vertical lines
of~\eqref{reduction-contratensor-diagram} are identified with each
other by the isomorphism in the upper horizontal line.
 Essentially, this holds because one has
$\R[[\C]]=\R[\C]+\J[[\C]]$ for any open right ideal $\J\subset\R$,
and specifically, for the open right ideal $\J$ provided by
the condition~(ii) from Section~\ref{prelim-topol-rings}.

 In particular, for any set $X$ one has $\R[[X]]/(\I\tim\R[[X]])
\cong(\R/\I)[X]$ and therefore $\R[[X]]\cong\varprojlim_\I\R[[X]]/
(\I\tim\R[[X]])$, where the projective limit is taken over all
the open right ideals $\I\subset\R$.
 It follows that the natural map
$$
 \P\lrarrow\varprojlim\nolimits_\I\P/(\I\tim\P)
$$
is an isomorphism for every projective left $\R$\+contramodule~$\P$.

\subsection{Strongly closed subgroups}
\label{prelim-strongly-closed-subgroups}
 Let $\A$ be a complete, separated topological abelian group (with
a base of neighborhoods of zero formed by open subgroups).
 Let $\fH\subset\A$ be a closed subgroup.
 Then the quotient group $Q=\A/\fH$ is separated by
Lemma~\ref{quotient-group-topology-lem}(a), but it does not have
to be complete.
 Let $\Q=Q\sphat\,$ be the completion of the topological group~$Q$.
 Then the natural morphism $p\:\A\rarrow\Q$ is the cokernel of
the morphism $\fH\rarrow\A$ in the category of complete, separated 
topological abelian groups.
 The group $Q$~is a dense subgroup in $\Q$, and $\fH$ is the kernel
of~$p$; but $p$~does not need to be surjective.

 A counterexample, going back to the book~\cite[Proposition~11.1]{RD},
can be also found in the book~\cite[Theorem~4.1.48]{AGM}.
 It is explained in these books how to construct, for any separated
topological abelian group $Q$, a complete, separated topological group
structure on the direct sum $\A=\bigoplus_{i=1}^\infty Q$ of
a countable set of copies of $Q$ in such a way that the original
topology on $Q$ is the quotient topology of the topology on $\A$ with
respect to the summation map $\A\rarrow Q$.
 So $\A$ is complete while $Q$ may be not.
 (See also~\cite[Problem~20D]{KN}.
 In a different context of topological vector spaces over the field of
real numbers, in its real topology, a counterexample to completeness
of quotients was suggested in~\cite[Exercise~IV.4.10(b)]{Bour};
see also~\cite[Proposition~11.2]{Pal}.)
 We refer to~\cite[Section~2]{Pextop} for a detailed discussion of
this construction of counterexamples.

 Furthermore, given a set $X$, we have the induced map of sets/abelian
groups $p[[X]]\:\A[[X]]\rarrow\Q[[X]]$.
 The subgroup $\fH[[X]]\subset\A[[X]]$ is the kernel of~$p[[X]]$.
 But even if the map~$p$ is surjective (i.~e., the topological group $Q$
is complete), the map $p[[X]]$ does not have to be surjective.
 Essentially, the problem consists in the following: given
an $X$\+indexed family of elements in the group $\Q$ converging
to zero in the topology of $\Q$, how to lift it to an $X$\+indexed
family of elements in the group $\A$ converging to zero in
the topology of~$\A$?

 Once again, the same construction from the books~\cite{RD,AGM}
provides a counterexample.
 One can check that, for any separated topological abelian group $Q$,
the topology on the group $\A=\bigoplus_{i=1}^\infty Q$ constructed
in~~\cite[Proposition~11.1]{RD} and~\cite[Theorem~4.1.48]{AGM}, has
the property that \emph{no infinite family of nonzero elements
in\/ $\A$ converges to zero in\/~$\A$}.
 This is essentially mentioned in the formulation
of~\cite[Proposition~11.1]{RD}.
 So, generally speaking, the above-described lifting problem may be
unsolvable.
 See~\cite[Example~11.2]{Pextop} for the details.

 We will say that a closed subgroup $\fH$ in a complete, separated
topological abelian group $\A$ is \emph{strongly closed} if
the quotient group $\A/\fH$ is complete and, for every set $X$,
the induced map $\A[[X]]\rarrow(\A/\fH)[[X]]$ is surjective.
 Clearly, any open subgroup in a complete, separated topological
abelian group is strongly closed.

\begin{lem} \label{countable-base-strongly-closed}
 Let\/ $\A$ be a complete, separated topological abelian group with
a countable base of neighborhoods of zero consisting of open subgroups.
 Then any closed subgroup in\/ $\A$ is strongly closed.
\end{lem}

\begin{proof}
 This is a straightforward extension of
Lemma~\ref{quotient-group-topology-lem}(b).
 For a generalization (and some details),
see~\cite[Proposition~11.6]{Pextop}.
\end{proof}

 Notice that countability of a base of neighborhoods of zero in
the topological quotient group $Q=\A/\fH$ is \emph{not} sufficient for
the validity of Lemma~\ref{countable-base-strongly-closed},
as the very same counterexample shows.

\begin{lem} \label{strongly-closed-in-quotients}
 Let\/ $\fK\subset\fH\subset\A$ be two embedded closed subgroups
in a complete, separated topological abelian group\/~$\A$.
 In this situation, \par
\textup{(a)} if\/ $\fK$ is strongly closed in\/ $\A$, then\/ $\fK$ is
strongly closed in\/~$\fH$; \par
\textup{(b)} if\/ $\fK$ is strongly closed in\/ $\A$ and\/ $\fH/\fK$
is strongly closed in\/ $\A/\fK$, then $\fH$ is strongly
closed in\/~$\A$; \par
\textup{(c)} if\/ $\fH$ is strongly closed in\/ $\A$ and\/ $\A/\fK$ is
complete, then $\fH/\fK$ is strongly closed in\/~$\A/\fK$; \par
\textup{(d)} if\/ $\fH$ is strongly closed in\/ $\A$ and\/ $\fK$ is
strongly closed in\/ $\fH$, then\/ $\fK$ is strongly closed in\/~$\A$.
\end{lem} 

\begin{proof}
 Part~(a): one observes that for any closed subgroups $\fK\subset\fH
\subset\A$, the subgroup $\fH/\fK$ is closed in $\A/\fK$.
 Hence $\fH/\fK$ is complete whenever $\A/\fK$ is (see
Lemma~\ref{subgroup-topology-lem}(b)).
 Furthermore, the square diagram of abelian groups
$\fH[[X]]\rarrow(\fH/\fK)[[X]] \rarrow(\A/\fK)[[X]]$, \
$\fH[[X]]\rarrow\A[[X]]\rarrow(\A/\fK)[[X]]$
is Cartesian for any set~$X$.
 Hence the morphism $\fH[[X]]\rarrow(\fH/\fK)[[X]]$ is surjective
whenever the morphism $\A[[X]]\rarrow(\A/\fK)[[X]]$ is.

 In parts~(b) and~(c) one uses the isomorphism of topological groups
$\A/\fH\cong(\A/\fK)/(\fH/\fK)$ and commutativity of the triangle
diagram $\A[[X]]\rarrow(\A/\fK)[[X]]\rarrow(\A/\fH)[[X]]$.

 Part~(d): to prove that the topological group $\A/\fK$ is complete,
consider the commutative diagram of abelian groups
$$
 \xymatrix{
  0 \ar[r] & \fH/\fK \ar[r] \ar[d]^\cong & \A/\fK \ar[r] \ar[d]
  & \A/\fH \ar[r] \ar[d]^\cong & 0 \\
  0 \ar[r] & \varprojlim_\U \fH/\fH\cap(\fK+\U) \ar[r]
  & \varprojlim_\U \A/(\fK+\U) \ar[r]
  & \varprojlim_\U \A/(\fH+\U)
 }
$$
 Here the projective limits in the lower row are taken over all
the open subgroups $\U$ in~$\A$.
 The upper row is a short exact sequence.
 The lower row is left exact as the projective limit of short exact
sequences $0\rarrow\fH/\fH\cap(\fK+\U)\rarrow\A/(\fK+\U)\rarrow
\A/(\fH+\U)\rarrow0$.
 The leftmost and rightmost vertical arrows are isomorphisms, since
the topological groups $\fH/\fK$ and $\A/\fH$ are complete by
assumption.
 It follows that the lower row is also a short exact sequence and
the middle vertical arrow is an isomorphism.

 To prove that the map $\A[[X]]\rarrow(\A/\fK)[[X]]$ is surjective
for any set $X$, consider the commutative diagram of abelian groups
$$
 \xymatrix{
  0 \ar[r] & \fH[[X]] \ar[r] \ar@{->>}[d] & \A[[X]] \ar[r] \ar[d]
  & (\A/\fH)[[X]] \ar[r] \ar@{=}[d] & 0 \\
  0 \ar[r] & (\fH/\fK)[[X]] \ar[r] & (\A/\fK)[[X]] \ar[r]
  & (\A/\fH)[[X]]
 }
$$
 The upper row is obviously left exact; and in fact it is exact since
the map $\A[[X]]\rarrow(\A/\fH)[[X]]$ is surjective by assumption.
 The lower row is also obviously left exact.
 The leftmost vertical arrow is surjective by assumption.
 It follows that the lower row is also a short exact sequence and
the middle vertical arrow is surjective.

 Alternatively, all the assertions of lemma follow from the existence
of the \emph{strong exact structure} on the additive category of
complete, separated topological abelian groups;
see~\cite[Theorem~11.5]{Pextop}.
\end{proof}

\begin{ex} \label{pro-finite-dimensional}
 A topological vector space over a field~$k$ is said to be
\emph{pro-finite-dimensional} (or \emph{pseudo-compact},
or~\emph{linearly compact}) if it is isomorphic to the projective limit
of a directed diagram of discrete finite-dimensional $k$\+vector spaces,
endowed with the projective limit topology.
 All pro-finite-dimensional topological vector spaces are complete
and separated; a complete, separated topological vector space $W$ is
pro-finite-dimensional if and only if open vector subspaces of finite
codimension form a base of neighborhoods of zero in~$W$.
 Any closed vector subspace of a pro-finite-dimensional vector space
is pro-finite-dimensional in the induced topology, and the related
quotient space is pro-finite-dimensional in the quotient topology.

 The category of pro-finite-dimensional topological vector spaces
(and continuous linear maps between them) is abelian; in fact,
it is anti-equivalent to the category of discrete $k$\+vector spaces.
 The anti-equivalence assigns to a discrete vector space $V$
the pro-finite-dimensional vector space $V^*=\Hom_k(V,k)$ with
the topology where the annihilators of finite-dimensional vector
subspaces of $V$ form a base of neighborhoods of zero.
 Conversely, to a pro-finite-dimensional vector space $W$,
the discrete vector space of all continuous linear maps $W\rarrow k$
is assigned.

 It follows that any short exact sequence of pro-finite-dimensional
vector spaces splits.
 In other words, any closed vector subspace in a pro-finite-dimensional
vector space is a direct summand (in the category of topological
vector spaces).
 Consequently, all closed vector subspaces in a pro-finite-dimensional
topological vector space are strongly closed.
 Moreover, in any topological vector space with linear topology
(i.~e., with a base of neighborhoods of zero consisting of vector 
subspaces), any vector subspace which is pro-finite-dimensional in
the induced topology is a direct summand in the category of topological
vector spaces, and consequently, is strongly closed (see~\cite[end
of Section~3]{Pextop}).
\end{ex}

\subsection{Strongly closed two-sided ideals}
\label{prelim-strongly-closed-ideals}
 Let $\R$ be a complete, separated topological ring with a base of
neighborhoods of zero formed by open right ideals.
 Let $\fH\subset\R$ be a closed two-sided ideal.
 Then the quotient ring $S=\R/\fH$ in its quotient topology is
a separated topological ring with a base of neighborhoods of zero
formed by open right ideals.
 Hence the completion $\S=S\sphat\,$ is a complete, separated
topological ring with a base of neighborhoods of zero formed by
open right ideals.

 The natural morphism $p\:\R\rarrow\S$ is a continuous homomorphism
of complete, separated topological rings with the kernel $\fH$, and
the universal one with this property; but it needs not be surjective.
 If $\R$ has a base of neighborhoods of zero consisting of open
two-sided ideals, then so do $S$ and~$\S$.

 The abelian category of discrete right $S$\+modules or, equivalently,
discrete right $\S$\+modules is a full subcategory of the abelian
category of discrete right $\R$\+modules.
 In other words, the exact functor $p_\diamond\:\discr\S\rarrow\discr\R$
is fully faithful.
 The full subcategory $\discr\S\subset\discr\R$ is closed under
arbitrary subobjects, quotient objects, and coproducts.

 For any discrete right $\R$\+module $\N$, the discrete right
$\R$\+module structure on the submodule $\N_\fH\subset\N$ comes from
a discrete right $\S$\+module structure.
 In other words, the discrete right $\R$\+module $\N_\fH$
belongs to the essential image of the functor $p_\diamond\:\discr\S
\rarrow\discr\R$.
 This is the maximal $\R$\+submodule in $\N$ with this property.
 The functor $\N\longmapsto\N_\fH$ is the coextension-of-scalars
functor with respect to the morphism $p\:\R\rarrow\S$, that is
$$
 p^\diamond(\N)\cong\N_\fH\qquad\text{for all \,$\N\in\discr\R$}.
$$

 Now let us assume that $\fH\subset\R$ is a strongly closed two-sided
ideal.
 Then surjectivity of the maps $p[[X]]\:\R[[X]]\rarrow\S[[X]]$ for
all sets $X$ implies that the exact functor of contrarestriction
of scalars $p_\sharp\:\S\contra\rarrow\R\contra$ is fully faithful.
 So the abelian category $\S\contra$ is a full subcategory in
the abelian category $\R\contra$.
 One easily observes that $\S\contra$ is closed under arbitrary
subobjects, quotient objects, and products in $\R\contra$.

 For any left $\R$\+contramodule $\C$, the left $\R$\+contramodule
structure of the quotient contramodule $\C/(\fH\tim\C)$ comes from
a left $\S$\+contramodule structure.
 In other words, the left $\R$\+contramodule $\C/(\fH\tim\C)$ belongs
to the essential image of the functor $p_\sharp\:\S\contra
\rarrow\R\contra$.
 This is the maximal quotient $\R$\+contramodule of $\C$ with
this property.
 The functor $\C\longmapsto\C/(\fH\tim\C)$ is
the contraextension-of-scalars functor with respect to the morphism
$p\:\R\rarrow\S$, that is
$$
 p^\sharp(\C)\cong\C/(\fH\tim\C)\qquad\text{for all \,$\C\in\R\contra$}.
$$

 When open two-sided ideals form a base of neighborhoods of zero
in~$\R$, one can compute the contratensor product $\N\ocn_\R\C$ as
$$
 \N\ocn_\R\C\cong\varinjlim\nolimits_\I \N_\I\ocn_\R\C\cong
 \varinjlim\nolimits_\I \N_\I\ocn_{\R/\I} (\C/\I\tim\C)
$$
for any discrete right $\R$\+module $\N$ and left
$\R$\+contramodule $\C$, where the inductive limits are taken over all
the open two-sided ideals $\I\subset\R$ (cf.~\cite[Section~D.2]{Pcosh}).

\subsection{Example} \label{prelim-endomorphism-ring}
 Let $A$ be an associative ring and $M$ be a left $A$\+module.
 Consider the associative ring $\R=\Hom_A(M,M)^\rop$ opposite to
the ring of endomorphisms of the $A$\+module~$M$.
 Then the ring $A$ acts in $M$ on the left and the ring $\R$ acts in $M$
on the right; so $M$ is an $A$\+$\R$\+bimodule.
 The following topology on the ring $\R$ is known as the \emph{finite
topology} in the literature~\cite[Section~IX.6]{Jac},
\cite[Section~107]{Fuc} (it is implicit in Jacobson's density
theorem~\cite{Jac0}, \cite[Section~IX.13]{Jac}).

 For every finitely generated $A$\+submodule $E\subset M$, consider
the subgroup $\Ann(E)=\Hom_A(M/E,M)\subset\Hom_A(M,M)$ consisting of
all the endomorphisms of the $A$\+module $M$ which annihilate
the submodule~$E$.
 Then $\Ann(E)$ is a left ideal in the ring $\Hom_A(M,M)$ and a right
ideal in the ring~$\R$.
 Let $\B$ denote the set of all right ideals in $\R$ of the form
$\Ann(E)$, where $E$ ranges over all the finitely generated submodules
in~$M$.
 Then $\B$ is a base of a complete, separated topology compatible with
the associative ring structure on~$\R$ \cite[Theorem~107.1]{Fuc}, 
\cite[Theorem~7.1]{PS}.
 The right action of $\R$ in $M$ makes $M$ a discrete right
$\R$\+module \cite[Proposition~7.3]{PS}.

 This example plays a key role in the categorical tilting
theory~\cite{PS,PS2}, and it is also our intended example of
a topological ring in the companion paper~\cite{BP3}.
 Besides the categories of left modules over an associative rings,
there are also other/wider classes of additive categories $\sA$ such
that for any object $M\in\sA$ there is a natural structure of
a complete, separated topological ring with a base of neighborhoods of
zero formed by open right ideals on the ring $\R=\Hom_\sA(M,M)^\rop$.
 A detailed discussion of these can be found
in~\cite[Sections~9\+-10]{PS} and~\cite[Section~3]{PS3}.

\Section{Flat Contramodules} \label{flat-contramodules-secn}

 The interactions of flatness with adic completion were studied by
Yekutieli for ideals in Noetherian commutative rings~\cite{Yek} and
in the greater generality of weakly proregular finitely generated ideals
in commutative rings~\cite{Yek2}.
 In the work of the present author, the theory of flat contramodules
was developed for ideals in Noetherian commutative
rings~\cite[Sections~B.8--B.9]{Pweak},
for centrally generated ideals in noncommutative Noetherian
rings~\cite[Section~C.5]{Pcosh}, for topological associative rings
with a countable base of neighborhoods of zero formed by open
two-sided ideals~\cite[Section~D.1]{Pcosh}, and for topological
rings with a countable base of neighborhoods of zero formed by
open right ideals~\cite[Sections~5\+-7]{PR}.

 In this section, we obtain some very partial results for topological
rings with an uncountable base of neighborhoods of zero.

 Let $\R$ be a complete, separated topological ring with a base of
neighborhoods of zero formed by open right ideals.
 A left $\R$\+contramodule $\F$ is called
\emph{flat}~\cite[Section~5]{PR} if the functor of contratensor product
with~$\F$
$$
 {-}\ocn_\R\F\:\discr\R\lrarrow\boZ\modl
$$
is exact as a functor from the abelian category of discrete right
$\R$\+modules to the category of abelian groups.
 The class of flat left $\R$\+contramodules is closed under coproducts
and direct limits in the category $\R\contra$ \cite[Lemma~5.6]{PR}.
 All projective left $\R$\+contramodules are flat.

 If $\I\subset\R$ is an open two-sided ideal, then the left
$\R/\I$\+module $\F/\I\tim\F$ is flat for any flat left
$\R$\+contramodule~$\F$.
 Indeed, the functor ${-}\ot_{\R/\I}(\F/\I\tim\F)\:\modr\R/\I\rarrow
\boZ\modl$ is exact, because there is a natural isomorphism
$$
 N\ot_{\R/\I}(\F/\I\tim\F)\cong N\ocn_\R\F \qquad
 \text{for all right $\R/\I$\+modules~$N$}.
$$
 If open two-sided ideals form a base of neighborhoods of zero in $\R$,
then the converse assertion also holds: a left $\R$\+contramodule $\F$
is flat if and only if the left $\R/\I$\+module $\F/\I\tim\F$ is flat
for every open two-sided ideal $\I\subset\R$.
 (Cf.\ Sections~\ref{prelim-reductions}
and~\ref{prelim-strongly-closed-ideals}.)

 The left derived functor
$$
 \Ctrtor^\R_i\:\discr\R\times\R\contra\lrarrow\boZ\modl
$$
is constructed using projective resolutions of the second (contramodule)
argument.
 So, if $\dotsb\rarrow\P_2\rarrow\P_1\rarrow\P_0\rarrow\C\rarrow0$ is
an exact complex in the abelian category $\R\contra$ and $\P_i$ are
projective left $\R$\+contramodules for all $i\ge0$, then
$$
 \Ctrtor^\R_i(\N,\C)=H_i(\N\ocn_\R\P_\bu)
 \qquad\text{for all \,$\N\in\discr\R$ and $i\ge0$}.
$$
 As always with derived functors of one argument, for any short exact
sequence of left $\R$\+contramodules $0\rarrow\A\rarrow\B\rarrow\C
\rarrow0$ and any discrete right $\R$\+module $\N$ there is a natural
long exact sequence of abelian groups
\begin{multline} \label{ctrtor-contramod-arg-long-sequence}
 \dotsb\lrarrow\Ctrtor_{i+1}^\R(\N,\C)\lrarrow\Ctrtor_i^\R(\N,\A) \\
 \lrarrow\Ctrtor_i^\R(\N,\B)\lrarrow\Ctrtor_i^\R(\N,\C)\lrarrow\dotsb
\end{multline}
 Since the functor of contratensor product $\N\ocn_\R{-}\:\R\contra
\rarrow\boZ\modl$ is right exact on the abelian category $\R\contra$
for every $\N\in\discr\R$, one has
$$
 \Ctrtor^\R_0(\N,\C)=\N\ocn_\R\C.
$$
 Furthermore, for any short exact sequence of discrete right
$\R$\+modules $0\rarrow\cL\rarrow\M\rarrow\N\rarrow0$ and any complex
of projective left $\R$\+contramodules $\P_\bu$, the short sequence
of complexes of abelian groups $0\rarrow\cL\ocn_\R\P_\bu\rarrow
\M\ocn_\R\P_\bu\rarrow\N\ocn_\R\P_\bu\rarrow0$ is exact (because
projective left $\R$\+contramodules are flat).
 Therefore, for any left $\R$\+contramodule $\C$ there is a long
exact sequence of abelian groups
\begin{multline} \label{ctrtor-discrete-arg-long-sequence}
 \dotsb\lrarrow\Ctrtor^\R_{i+1}(\N,\C)\lrarrow\Ctrtor^\R_i(\cL,\C) \\
 \rarrow\Ctrtor^\R_i(\M,\C)\lrarrow\Ctrtor^\R_i(\N,\C)\lrarrow\dotsb
\end{multline}

 We will say that a left $\R$\+contramodule $\F$ is
\emph{$1$\+strictly flat} if $\Ctrtor^\R_1(\N,\F)=0$ for all discrete
right $\R$\+modules~$\N$.
 More generally, a left $\R$\+contramodule $\F$ is
\emph{$n$\+strictly flat} if $\Ctrtor^\R_i(\N,\F)=0$ for all discrete
right $\R$\+modules $\N$ and all $1\le i\le n$.
 A left $\R$\+contramodule $\F$ is \emph{$\infty$\+strictly flat} if
it is $n$\+strictly flat for all $n>0$.

 Clearly, all projective left $\R$\+contramodules are
$\infty$\+strictly flat.
 It follows from the exact
sequence~\eqref{ctrtor-contramod-arg-long-sequence} that the class of
all $n$\+strictly flat left $\R$\+contramodules is closed under
extensions in $\R\contra$ for every $n\ge1$, and that the class
of all $\infty$\+strictly flat left $\R$\+contramodules is closed under
(extensions and) the passage to the kernels of surjective morphisms.
 Given some $n\ge1$, the class of all $n$\+strictly flat left
$\R$\+contramodules is closed under the kernels of surjective morphisms
if and only if it coincides with the class of all $\infty$\+strictly
flat left $\R$\+contramodules.

 From the exact sequence~\eqref{ctrtor-discrete-arg-long-sequence} one
can conclude that every $1$\+strictly flat left $\R$\+contramodule
is flat.
 According to~\cite[proof of Lemma~6.10, Remark~6.11 and
Corollary~6.15]{PR}, when the topological ring $\R$ has a countable base
of neighborhoods of zero (formed by open right ideals), the classes of
flat, $1$\+strictly flat, and $\infty$\+strictly flat left
$\R$\+contramodules coincide.

 Let us say that a short exact sequence of left $\R$\+contramodules
$0\rarrow\A\rarrow\B\rarrow\C\rarrow0$ is \emph{contratensor pure} if
the induced sequence $0\rarrow\N\ocn_\R\A\rarrow\N\ocn_\R\B\rarrow
\N\ocn_\R\C\rarrow0$ is exact (i.~e., the map $\N\ocn_\R\A\rarrow
\N\ocn_\R\B$ is injective) for every discrete right $\R$\+module~$\N$.
 If the left $\R$\+contramodule $\B$ is $1$\+strictly flat, then
the sequence $0\rarrow\A\rarrow\B\rarrow\C\rarrow0$ is contratensor pure
if and only if the left $\R$\+contramodule $\C$ is $1$\+strictly flat.

\begin{lem} \label{1-strictly-flat-class-coproduct-closed}
 The class of all\/ $1$\+strictly flat left\/ $\R$\+contramodules is
closed under infinite coproducts in\/ $\R\contra$.
\end{lem}

\begin{proof}
 Let $(\F_\alpha)_\alpha$ be a family of $1$\+strictly flat
left $\R$\+contramodules.
 Choose short exact sequences of left $\R$\+contramodules
$0\rarrow\fK_\alpha\rarrow\P_\alpha\rarrow\F_\alpha\rarrow0$, where
$\P_\alpha$ are projective left $\R$\+contramodules.
 Then the sequence $\coprod_\alpha\fK_\alpha\rarrow\coprod_\alpha
\P_\alpha\rarrow\coprod_\alpha\F_\alpha\rarrow0$ (where the coproducts
are taken in $\R\contra$) is exact, as the functors of coproduct
are right exact in any abelian category.
 Therefore, there is a natural surjective $\R$\+contramodule morphism
from $\coprod_\alpha\fK_\alpha$ onto the kernel $\fK$ of the morphism
$\coprod_\alpha\P_\alpha\rarrow\coprod_\alpha\F_\alpha$.
 The contratensor product functor $\ocn_\R$ preserves colimits, hence
for any discrete right $\R$\+module $\N$ the morphism $\N\ocn_\R
\coprod_\alpha\fK_\alpha\rarrow\N\ocn_\R\coprod_\alpha\P_\alpha$
is injective (being isomorphic to the morphism $\coprod_\alpha
\N\ocn_\R\fK_\alpha\rarrow\coprod_\alpha\N\ocn_\R\P_\alpha$, where
the coproducts are taken in the category of abelian groups).
 At the same time, the morphism $\N\ocn_\R\coprod_\alpha\fK_\alpha
\rarrow\N\ocn_\R\fK$ is surjective.
 It follows that the morphism $\N\ocn_\R\coprod_\alpha\fK_\alpha\rarrow
\N\ocn_\R\fK$ is an isomorphism and the morphism $\N\ocn_\R\fK\rarrow
\N\ocn_\R\coprod_\alpha\P_\alpha$ is injective, that is, the short exact
sequence $0\rarrow\fK\rarrow\coprod_\alpha\P_\alpha\rarrow\coprod_\alpha
\F_\alpha\rarrow0$ is contratensor pure.
 Since the left $\R$\+contramodule $\coprod_\alpha\P_\alpha$ is
projective, it follows that the left $\R$\+contramodule
$\coprod_\alpha\F_\alpha$ is $1$\+strictly flat.
\end{proof}

\begin{lem} \label{1-strictly-flat-class-countable-colimit-closed}
 The class of all\/ $1$\+strictly flat left\/ $\R$\+contramodules is
closed under \emph{countable} direct limits in\/ $\R\contra$.
\end{lem}

\begin{proof}
 Let $\F_1\rarrow\F_2\rarrow\F_3\rarrow\dotsb$ be a sequence
of left $\R$\+contramodules and $\R$\+contramodule morphisms
between them.
 Then the direct limit $\varinjlim_n\F_n$ is the cokernel of the morphism
$\id-\mathit{shift}\:\coprod_{n=1}^\infty\F_n\rarrow\coprod_{n=1}^\infty
\F_n$.
 Denote the image of this morphism by~$\fL$.
 Arguing as in the proof of
Lemma~\ref{1-strictly-flat-class-coproduct-closed}, we have
a surjective morphism $\coprod_{n=1}^\infty\F_n\rarrow\fL$ and
an exact sequence $0\rarrow\fL\rarrow\coprod_{n=1}^\infty\F_n\rarrow
\varinjlim_n\F_n\rarrow0$.
 The morphism $\N\ocn_\R\coprod_n\F_n\rarrow\N\ocn_\R\coprod_n\F_n$
is injective (being isomorphic to the morphism
$\coprod_n\N\ocn_\R\nobreak\F_n\rarrow\coprod_n\N\ocn_\R\F_n$) for every
discrete right $\R$\+module~$\N$.
 At the same time, the morphism $\N\ocn_\R\coprod_n\F_n\rarrow
\N\ocn_\R\fL$ is surjective.
 It follows that the morphism $\N\ocn_\R\coprod_n\F_n\rarrow\N\ocn_\R\fL$
is an isomorphism and the morphism $\N\ocn_\R\fL\rarrow\N\ocn_\R
\coprod_n\F_n$ is injective, that is, the short exact sequence
$0\rarrow\fL\rarrow\coprod_n\F_n\rarrow\varinjlim_n\F_n\rarrow0$ is
contratensor pure.
 Now, if the left $\R$\+contramodules $\F_n$ are $1$\+strictly flat
for all $n\ge1$, then the left $\R$\+contramodule $\coprod_n\F_n$ is
$1$\+strictly flat by
Lemma~\ref{1-strictly-flat-class-coproduct-closed}, and it follows that
the left $\R$\+contramodule $\varinjlim_n\F_n$ is also
$1$\+strictly flat.
\end{proof}

 In fact, the class of all $1$\+strictly flat left $\R$\+contramodules
is closed under \emph{all} direct limits in $\R\contra$.
 This is the result of~\cite[Corollary~7.1]{BPS}.

 A left $\R$\+contramodule $\C$ is said to be \emph{separated} if
the intersection of its subgroups $\I\tim\C$, with $\I\subset\R$
ranging over all the open right ideals in $\R$, vanishes:
$\bigcap_{\I\subset\R}\I\tim\C=0$.
 Following the discussion at the end of
Section~\ref{prelim-reductions}, all projective
left $\R$\+contramodules are separated.
 Over a topological ring with a countable base of neighborhoods of
zero, all flat contramodules are separated~\cite[Corollary~6.15]{PR}. 

\begin{lem} \label{separated-telescope-exact}
 Let\/ $\F_1\rarrow\F_2\rarrow\F_3\rarrow\dotsb$ be a sequence of
left\/ $\R$\+contramodules and contramodule morphisms between them.
 Assume that the left\/ $\R$\+contramodule\/ $\coprod_{n=1}^\infty
\F_n$ is separated.
 Then the natural short sequence of left\/ $\R$\+contramodules\/
\begin{equation} \label{telescope}
 0\lrarrow\coprod\nolimits_{n=1}^\infty\F_n\lrarrow
 \coprod\nolimits_{n=1}^\infty\F_n\lrarrow
 \varinjlim\nolimits_{n\ge1}\F_n\lrarrow0
\end{equation}
is exact.
\end{lem}

\begin{proof}
 In any additive category with countable coproducts,
the short sequence~\eqref{telescope} is always right exact
(cf.\ the previous proof).
 The nontrivial assertion is that the left $\R$\+contramodule
morphism $\id-\mathit{shift}\:\F\rarrow\F$, where we put
$\F=\coprod_{n=1}^\infty\F_n$, is injective in the assumptions of
the lemma.
 Indeed, let $\N$ be a discrete right $\R$\+module.
 Applying to~\eqref{telescope} the contratensor product
functor $\N\ocn_\R{-}$, we obtain the short sequence of abelian groups
$$
 0\lrarrow\coprod\nolimits_{n=1}^\infty\N\ocn_\R\F_n\lrarrow
 \coprod\nolimits_{n=1}^\infty\N\ocn_\R\F_n\lrarrow
 \varinjlim\nolimits_{n\ge1}\N\ocn_\R\F_n\lrarrow0,
$$
which is exact since direct limits are exact functors in
the category of abelian groups.
 So the abelian group homomorphism $\N\ocn_\R(\id-\mathit{shift})$
is injective for any discrete right $\R$\+module~$\N$.
 In particular, let $\I\subset\R$ be an open right ideal.
 Then the morphism $\R/\I\ocn_\R(\id-\mathit{shift})\:
\F/\I\tim\F\rarrow\F/\I\tim\F$ is injective.
 It follows that the kernel of the morphism $\id-\mathit{shift}$
is contained in the subgroup (in fact, always a left
$\R$\+subcontramodule) \,$\bigcap_{\I\subset\R}\I\tim\F\subset\F$.
\end{proof}

 Before formulating the next corollary, we notice that, if a left
$\R$\+contramodule $\F$ has projective dimension not exceeding~$n$
(as an object of the abelian category $\R\contra$) and $\F$ is
$n$\+strictly flat, then $\F$ is also $\infty$\+strictly flat.

\begin{cor} \label{countable-colimits-of-projective-contramodules}
 Any countable direct limit of projective left\/ $\R$\+contramodules
has projective dimension not exceeding\/~$1$ in\/ $\R\contra$.
 In particular, any such\/ $\R$\+con\-tramodule is\/
$\infty$\+strictly flat.
\end{cor}

\begin{proof}
 The first assertion is a corollary of
Lemma~\ref{separated-telescope-exact}.
  Indeed, a coproduct of projective contramodules is projective, and
any projective contramodule is separated; so
the exact sequence~\eqref{telescope} is a desired projective resolution.
 The second assertion follows immediately from the first one together
with Lemma~\ref{1-strictly-flat-class-countable-colimit-closed}.
\end{proof}

\Section{Projective Covers of Flat Contramodules}
\label{projective-covers-of-flats}

 Let $\sB$ be an abelian category with enough projective objects.
 An epimorphism $p\:P\rarrow C$ in $\sB$ is called a \emph{projective
cover} (of the object $C$) if the object $P$ is projective and, for any
endomorphism $e\:P\rarrow P$, the equation $pe=p$ implies that $e$~is
an automorphism of $P$ (i.~e., $e$~is invertible).

 A subobject $K$ of an object $Q\in\sB$ is said to be \emph{superfluous}
if, for any other subobject $G\subset Q$, the equation $K+G=Q$ implies
that $G=Q$.
 If a subobject $K\subset Q$ is superfluous then, for any subobject
$E\subset Q$, the quotient $(K+E)/E$ is a superfluous subobject
of the quotient~$Q/E$.

\begin{lem} \label{superfluous-kernel}
 Let $P\in\sB$ be a projective object.
 Then an epimorphism $p\:P\rarrow C$ in\/ $\sB$ is a projective cover
if and only if its kernel $K$ is a superfluous subobject in\/~$P$.
\end{lem}

\begin{proof}
 Let $p\:P\rarrow C$ be a projective cover with the kernel $K$, and
let $G\subset P$ be a subobject such that $K+G=P$.
 Then the restriction of~$p$ onto $G$ is an epimorphism $s\:G\rarrow C$.
 Since $P$ is projective, there exists a morphism $f\:P\rarrow G$
making the triangle diagram $P\rarrow G\rarrow C$ commutative.
 Let $e\:P\rarrow P$ be the composition of the morphism~$f$ with
the embedding $G\rarrow P$.
 Then $pe=p$, and by assumption it follows that $e$~is invertible.
 Hence $G=P$.

 Conversely, let $p\:P\rarrow C$ be an epimorphism with a superfluous
kernel $K\subset P$, and let $e\:P\rarrow P$ be an endomorphism
satisfying $pe=p$.
 Let $G\subset P$ be the image of~$e$; then $K+G=P$.
 By assumption, it follows that $G=P$, so $e$~is an epimorphism.
 Then, since $P$ is projective, the kernel $L$ of~$e$ must be
a direct summand of~$P$.
 Denote by $E\subset P$ a complementary direct summand.
 The equation $pe=p$ implies that $L\subset K$, hence $K+E=P$.
 Again by assumption, it follows that $E=P$, so $L=0$ and $e$~is
an automorphism of~$P$.
\end{proof}

\begin{prop} \label{projective-covers-of-flat-modules-prop}
 Let $R$ be an associative ring.
 Then any flat $R$\+module that has a projective cover is projective.
\end{prop}

\begin{proof}
 A proof of this result can be found in~\cite[Section~36.3]{Wis}.
\end{proof}

%
%
%
%

 We refer to Sections~\ref{prelim-reductions}
and~\ref{prelim-strongly-closed-ideals} for the discussion of
reductions of contramodules modulo strongly closed ideals.

\begin{lem} \label{cover-reduction-lem}
 Let\/ $\R$ be a complete, separated topological ring with a base of
neighborhoods of zero formed by open right ideals, $\J$ be a strongly
closed two-sided ideal in\/ $\R$, and\/ $\T=\R/\J$ be
the topological quotient ring.
 Assume that a left\/ $\R$\+contramodule\/ $\C$ has a projective cover
$p\:\P\rarrow\C$ in\/ $\R\contra$.
 Then the induced map\/ $\bar p\:\P/\J\tim\P\rarrow\C/\J\tim\C$ is
a projective cover of the left\/ $\T$\+contramodule\/ $\C/\J\tim\C$.
\end{lem}

\begin{proof}
 Set $\fK=\ker(p)$.
 Then $0\rarrow\fK\rarrow\P\rarrow\C\rarrow0$ is a short exact
sequence in $\R\contra$ and $0\rarrow\fK/(\J\tim\P)\cap\fK\rarrow
\P/\J\tim\P\rarrow\C/\J\tim\C\rarrow0$ is a short exact sequence
in $\T\contra$.
 The left $\T$\+contramodule $\P/\J\tim\P$ is projective, since
the left $\R$\+contramodule $\P$ is; and the $\T$\+subcontramodule
$\fK/(\J\tim\P)\cap\fK\subset\P/\J\tim\P$ is superfluous, since
the $\R$\+subcontramodule $\fK\subset\P$ is.
\end{proof}

%

\begin{prop} \label{projective-covers-of-flat-contramodules-prop}
 Let\/ $\R$ be a complete, separated topological ring with a base of
neighborhoods of zero formed by open two-sided ideals, and let\/ $\F$
be a\/ $1$\+strictly flat left\/ $\R$\+contramodule.
 Assume that\/ $\F$ has a projective cover $p\:\P\rarrow\F$ in\/
$\R\contra$.
 Then the map $p$~is an isomorphism and the\/ $\R$\+contramodule\/
$\F\cong\P$ is projective.
\end{prop}

\begin{proof}
 Let $\fL\subset\F$ be the kernel of the epimorphism~$p$.
 For any open two-sided ideal $\I\subset\R$, we have a short sequence of
left $\R/\I$\+modules $0\rarrow\fL/\I\tim\fL\rarrow\P/\I\tim\P
\rarrow\F/\I\tim\F\rarrow0$, which is exact since
$\Ctrtor^\R_1(\R/\I,\F)=0$.
 The left $\R/\I$\+module $\F/\I\tim\F$ is flat, since
the left $\R$\+contramodule $\F$ is.
 Furthermore, the morphism $\P/\I\tim\P\rarrow\F/\I\tim\F$ is
a projective cover in the category of left $\R/\I$\+modules by
Lemma~\ref{cover-reduction-lem} (applied in the particular case of
an open two-sided ideal $\I$ and a discrete quotient ring $\T=\R/\I$).
 Using Proposition~\ref{projective-covers-of-flat-modules-prop}, we
conclude that the $\R/\I$\+module $\F/\I\tim\F$ is projective and
$\fL/\I\tim\fL=0$ for every open two-sided ideal $\I\subset\R$.
 Since $\P$ is a projective left $\R$\+contramodule, one has
$\P=\varprojlim_\I\P/\I\tim\P$ (see Section~\ref{prelim-reductions}),
hence $\fL=\bigcap_\I \I\tim\fL\,\subset\,\bigcap_\I \I\tim\P=0$.
\end{proof}

 With the additional assumption that the projective dimension of
the left $\R$\+contra\-module $\F$ does not exceed~$1$, the result of
Proposition~\ref{projective-covers-of-flat-contramodules-prop} can
be extended to complete, separated topological rings with a base of
neighborhoods of zero formed by open right ideals.
 This is the assertion of~\cite[Theorem~3.1]{BPS}.

%
%


\Section{Bass Flat Contramodules}

 Let $a_1$, $a_2$, $a_3$,~\dots\ be a sequence of
elements in the topological ring~$\R$.
 For every $n\ge1$, the multiplication by~$a_n$ on the right is
a left $\R$\+contramodule morphism $\R\rarrow\R$ (where $\R$ is
viewed as a free left $\R$\+contramodule with one generator).
 The direct limit
$$
 \B=\varinjlim\,(\R\overset{*a_1}\lrarrow\R\overset{*a_2}\lrarrow\R
 \overset{*a_3}\lrarrow\dotsb)
$$
is called the \emph{Bass flat left\/ $\R$\+contramodule} associated
with the sequence of elements $(a_n\in\R)_{n\ge1}$.
 According to
Corollary~\ref{countable-colimits-of-projective-contramodules},
the Bass flat left $\R$\+contramodules are $\infty$\+strictly flat
and have projective dimension not exceeding~$1$.

 In particular, when the topological ring $\R=R$ is discrete,
the above construction specializes to the classical definition of
a \emph{Bass flat left $R$\+module}.

\begin{lem} \label{bass-modules-projective-implies-perfect}
 If all Bass flat left modules over an associative ring $R$ are
projective, then all flat left $R$\+modules are projective (i.~e.,
the ring $R$ is left perfect).
\end{lem}

\begin{proof}
 Clear from the proof of the implication~(5)$\,\Longrightarrow\,$(6)
in~\cite[Theorem~P]{Bas}, which only uses projectivity of the Bass
flat modules.
 (For a more recent exposition,
see~\cite[proof of Theorem~28.4\,(d)$\,\Longrightarrow\,$(e)]{AF}.)
\end{proof}

%
%

\begin{lem} \label{bass-contramodule-vanishes}
 Let\/ $\R$ be a complete, separated topological ring with a base of
neighborhoods of zero formed by open right ideals.
 Let $a_1$, $a_2$, $a_3$,~\dots\ be a sequence of elements in\/ $\R$,
and let\/ $\B$ be the related Bass flat left\/ $\R$\+contramodule.
 Then one has\/ $\B=0$ if and only if for every $m\ge 1$ the sequence
of elements~$a_m$, $a_ma_{m+1}$,~\dots, $a_ma_{m+1}\dotsm a_n$,~\dots\
converges to zero in the topology of\/ $\R$ as $n\to\infty$.
\end{lem}

\begin{proof}
 The left $\R$\+contramodule $\B$ can be constructed as the cokernel
of the left $\R$\+contramodule morphism $\id-\mathit{shift}\:
\coprod_{n=1}^\infty\R\rarrow\coprod_{n=1}^\infty\R$.
 In other words, $\B$ is the cokernel of the morphism of free left
$\R$\+contramodules
$$
 f=f_{a_1,a_2,a_3,\dotsc}\:\R[[y_1,y_2,y_3,\dotsc]]
 \lrarrow\R[[x_1,x_2,x_3,\dotsc]]
$$
defined on the generators by the rule $f(y_n)=x_n-a_nx_{n+1}$.
 According to 
Lemma~\ref{separated-telescope-exact},
the morphism~$f$ is injective (but we will not need to use this fact).

 One has $\B=0$ if and only if the image of~$x_m$ in $\B$ vanishes
for every $m\ge1$.
 We will show that the image of $x_m$ in $\B$ vanishes if and only if
the sequence of elements~$a_m$, $a_ma_{m+1}$,~\dots, $a_ma_{m+1}\dotsm
a_n$,~\dots\ converges to zero in~$\R$.
 More generally, given an element $r\in\R$, the image of the element
$rx_m$ under the map $\R[[x_1,x_2,x_3,\dotsc]]\rarrow\B$ vanishes
if and only if the sequence of elements~$ra_m$, $ra_ma_{m+1}$,~\dots,
$ra_ma_{m+1}\dotsm a_n$,~\dots\ converges to zero in $\R$ as
$n\to\infty$ (cf.~\cite[Lemma~1.2]{Bas} and~\cite[Lemma~28.1]{AF}).

 ``If'': assuming that the sequence~$ra_m$, $ra_ma_{m+1}$,~\dots\
converges to zero, we have to show that the element~$rx_m$ belongs to
the image of the morphism~$f$.
 Indeed, one has $\sum_{n=m}^\infty ra_m\dotsm a_{n-1}y_n=ry_m+
ra_my_{m+1}+ra_ma_{m+1}y_{m+2}+\dotsb\in\R[[y_1,y_2,y_3,\dotsc]]$ and
\begin{multline*}
 f\left(\sum\nolimits_{n=m}^\infty ra_m\dotsm a_{n-1}y_n\right)=
 \sum\nolimits_{n=m}^\infty ra_m\dotsm a_{n-1}(x_n-a_nx_{n+1}) \\ =
 \sum\nolimits_{n=m}^\infty ra_m\dotsm a_{n-1}x_n-
 \sum\nolimits_{n=m}^\infty ra_m\dotsm a_nx_{n+1} = rx_m.
\end{multline*}

 ``Only if'': assume that $rx_m\in\im f$; so there exists an element
$z=\sum_{n=1}^\infty b_ny_n\in\R[[y_1,y_2,y_3,\dotsc]]$ such that
$f(z)=rx_m$.
 Since $z\in\R[[y_1,y_2,y_3,\dotsc]]$, the sequence of elements
$b_n\in\R$ converges to zero as $n\to\infty$.
 On the other hand, the equation $f(z)=rx_m$ means that
$b_1=\delta_{m,1}r$ and $b_{n+1}-b_na_n=\delta_{m,n+1}r$ for $n\ge1$,
where $\delta_{i,j}$~is the Kronecker delta symbol.
 Hence $b_1=\dotsb=b_{m-1}=0$ and $b_m=r$, $b_{m+1}=ra_m$,
$b_{m+2}=ra_ma_{m+1}$,~\dots, $b_n=ra_m\dotsm a_{n-1}$ for $n\ge m$.
\end{proof}

 The following proposition, which is the main result of this section,
is the contramodule version of~\cite[Lemma~1.3]{Bas}
(see also~\cite[Lemma~28.2]{AF}).

\begin{prop} \label{bass-projective-implies-coperfect}
 Let\/ $\R$ be a complete, separated topological ring with a base of
neighborhoods of zero formed by open right ideals, and
let $a_1$, $a_2$, $a_3$,~\dots\ be a sequence of elements in\/~$\R$.
 Assume that the related Bass flat left\/ $\R$\+contramodule\/ $\B$
is a projective left\/ $\R$\+contramodule.
 Then for any open right ideal\/ $\I\subset\R$ the chain of right
ideals\/ $\R\supset a_1\R+\I\supset a_1a_2\R+\I\supset
a_1a_2a_3\R+\I\supset\dotsb$ terminates.
\end{prop}

\begin{proof}
 In the notation of the previous proof, set
$\P=\R[[x_1,x_2,x_2,\dotsc]]$.
 Then we have a surjective left $\R$\+contramodule morphism
$\rho\:\P\rarrow\B$.
 Assuming that $\B$ is a projective left $\R$\+contramodule,
$\rho$~has a section, i.~e., there exists a left $\R$\+contramodule
morphism $\sigma\:\B\rarrow\P$ such that $\rho\sigma=\id_\B$.

 Following~\cite[Proposition~4.2(b)]{Pcoun}, we assign to every
left $\R$\+contramodule $\C$ the functor of contratensor product
$\CT(\C)={-}\ocn_\R\C$ acting from the category of cyclic discrete
right $\R$\+modules to the category of abelian groups.
 So we have $\CT(\C)(\R/\I)=\R/\I\ocn_\R\C=\C/\I\tim\C$ for any
open right ideal $\I\subset\R$.
 In particular, $\CT(\P)(\R/\I)=(\R/\I)[x_1,x_2,x_3,\dots]$ is
naturally the direct sum of a countable set of copies of
the abelian group~$\R/\I$.
 We will denote formally the elements of the group $\CT(\P)(\R/\I)$
by expressions like $\sum_{m=1}^n r_mx_m$, where $r_m\in\R/\I$.
 Here the sum is a direct sum, of course.

 Similarly, the abelian group $\CT(\B)(\R/\I)$ can be computed
as the direct limit
$$
 \R/\I\ocn_\R\B=
 \varinjlim\,(\R/\I\overset{*a_1}\lrarrow\R/\I\overset{*a_2}\lrarrow
 \R/\I\overset{*a_3}\lrarrow\dotsb)
$$
of the sequence of the right $\R$\+modules $\R/\I$ and the maps of
right multiplication by the elements~$a_n$ acting between them.
 We will denote the elements of the groups $\R/\I$ in this sequence
by $rz_n$, where $r\in\R/\I$ and $z_n$~is a formal symbol; so
the transition map takes $rz_n$ to $ra_nz_{n+1}$.
 Let $r\tilde z_n\in\CT(\B)(\R/\I)$ be the notation for the images of
the elements~$rz_n$ in the direct limit.
 Then the equations $r\tilde z_n=ra_n\tilde z_{n+1}$ hold in
the abelian group $\CT(\B)(\R/\I)$ for all $n\ge1$ and $r\in\R/\I$.

 Applying the functor $\CT$ to the morphisms $\rho$ and~$\sigma$,
we obtain the abelian group homomorphism $\rho_\I\:\CT(\P)(\R/\I)
\rarrow\CT(\B)(\R/\I)$ taking $r_nx_n$ to $r_n\tilde z_n$, and its
section $\sigma_\I\:\CT(\B)(\R/\I)\rarrow\CT(\P)(\R/\I)$ (so
$\rho_\I\sigma_\I=\id$).
 Furthermore, for any two open right ideals $\I$ and $\J\subset\R$
and an element $s\in\R$ such that $s\J\subset\I$, there is
a morphism of discrete right $\R$\+modules
$\R/\J\overset{s*}\rarrow\R/\I$ taking a coset $t+\J\in\R/\J$,
\,$t\in\R$ to the coset $st+\I\in\R/\I$.
 Applying the functor $\CT(\P)$ to the morphism~$s*$ produces
the abelian group homomorphism $\CT(\P)(s*)\:\CT(\P)(\R/\J)
\rarrow\CT(\P)(\R/\I)$ taking an element $\sum_{m=1}^nt_mx_m$
to the element $\sum_{m=1}^n st_mx_m$, where $t_m\in\R/\J$ and
$st_m\in\R/\I$.
 Similarly, applying the functor $\CT(\B)$ to the morphism~$s*$ we
get the abelian group homomorphism $\CT(\B)(s*)\:\CT(\B)(\R/\J)
\rarrow\CT(\B)(\R/\I)$ taking $t\tilde z_n$ to $st\tilde z_n$
for all $t\in\R/\J$.
 Since $\CT(\rho)$ and $\CT(\sigma)$ are morphisms of functors,
the maps $\CT(\P)(s*)$ and $\CT(\B)(s*)$ form a commutative
square diagram with the maps $\rho_\I$ and~$\rho_\J$, and another
commutative square diagram with the maps $\sigma_\I$ and~$\sigma_\J$.

 Consider the element $\tilde z_1\in\CT(\B)(\R/\I)$,
where $z_1=1z_1\in\R/\I$ is the coset of the element $1\in\R$ and
$\tilde z_1\in\CT(\B)(\R/\I)$ is the image of~$z_1$ in the direct limit.
 Then the element $\sigma_\I(\tilde z_1)\in\CT(\P)(\R/\I)$ can be
presented in the form $\sigma_\I(\tilde z_1)=\sum_{i=1}^k r_ix_i$
with some integer $k\ge1$ and elements $r_i\in\R/\I$.
 So we have
\begin{multline*}
 a_1\dotsm a_{k-1}\tilde z_k = \tilde z_1=\rho_\I\sigma_\I(\tilde z_1)=
 r_1\tilde z_1+\dotsb+r_k\tilde z_k = \\
 (r_1a_1\dotsm a_{k-1}+r_2a_2\dotsm a_{k-1}+\dotsb+r_{k-1}a_{k-1}
 + r_k)\tilde z_k\,\in\,\CT(\B)(\R/\I).
\end{multline*}
 It follows that there exists an integer $m\ge k$ such that
the equation
$$
 a_1\dotsm a_{m-1}z_m = (r_1a_1\dotsm a_{m-1} + r_2a_2\dotsm a_{m-1}
 + \dotsb + r_ka_k\dotsm a_{m-1})z_m
$$
holds in the group $\R/\I=(\R/\I)z_m$. 
 Hence for every $n\ge m$ we have
$$
 a_1\dotsm a_{n-1} \equiv r_1a_1\dotsm a_{n-1} + r_2a_2\dotsm a_{n-1}
 + \dotsb + r_ka_k\dotsm a_{n-1} \mod \I.
$$

 Set $s=a_1\dotsm a_n\in\R$, and choose an open right ideal
$\J\subset\R$ such that $s\J\subset\I$.
 Consider the element $\tilde z_{n+1}\in\CT(\B)(\R/\J)$,
where $z_{n+1}=1z_{n+1}\in\R/\J$ is the coset of the element $1\in\R$.
 Then the element $\sigma_\J(\tilde z_{n+1})\in\CT(\P)(\R/\J)$ can be
presented in the form $\sigma_\J(\tilde z_{n+1})=\sum_{i=1}^l t_ix_i$
with some integer $l\ge k$ and elements $t_i\in\R/\J$.
 By commutativity of the square diagram formed by the maps
$\CT(\B)(s*)$ and $\CT(\P)(s*)$ together with the maps $\sigma_\J$
and~$\sigma_\I$, we have
$$
 \sigma_\I(s\tilde z_{n+1})=\sum\nolimits_{i=1}^l st_ix_i
 \,\in\,\CT(\P)(\R/\I).
$$
 On the other hand, 
$s\tilde z_{n+1}=a_1\dotsm a_n\tilde z_{n+1} = \tilde z_1
\,\in\,\CT(\B)(\R/\I)$, hence
$$
 \sigma_\I(s\tilde z_{n+1})=\sigma_\I(\tilde z_1)=
 \sum\nolimits_{i=1}^k r_ix_i \,\in\,\CT(\P)(\R/\I).
$$
 Since the summation signs in these expressions stand for a direct sum,
we can conclude that $r_i\equiv st_i\bmod\I$
for every $1\le i\le k$.

 Finally, we have
\begin{align*}
 a_1\dotsm a_{n-1} &\equiv r_1a_1\dotsm a_{n-1} + r_2a_2\dotsm a_{n-1}
 + \dotsb + r_ka_k\dotsm a_{n-1} \\
 &\equiv st_1a_1\dotsm a_{n-1} + st_2a_2\dotsm a_{n-1}
 + \dotsb + st_ka_k\dotsm a_{n-1} \\
 &\in\, s\R=a_1\dotsm a_n\R \mod\I.
\end{align*}
 Thus $a_1\dotsm a_{n-1}\in a_1\dotsm a_n\R+\I$, and it follows that
$a_1\dotsm a_{n-1}\R+\I=a_1\dotsm a_n\R+\I$ for all $n\ge m$,
as desired.
\end{proof}

\begin{lem} \label{bass-projective-reduction}
 Let\/ $\R$ be a complete, separated topological ring with a base of
neighborhoods of zero formed by open right ideals, $\J$ be a strongly
closed two-sided ideal in\/ $\R$, and $\T=\R/\J$ be the topological
quotient ring.
 Assume that all Bass flat left\/ $\R$\+contramodules are projective.
 Then all Bass flat left\/ $\T$\+contramodules are projective, too.
\end{lem}

\begin{proof}
 Let $\bar a_1$, $\bar a_2$, $\bar a_3$,~\dots\ be a sequence of
elements in $\T$ and $\overline\B=\varinjlim\,(\T\overset{*\bar a_1}
\rarrow\T\overset{*\bar a_2}\rarrow\T\overset{*\bar a_3}
\rarrow\dotsb)$ be the related Bass flat left $\T$\+contramodule.
 Lift the elements $\bar a_n\in\T$ to some elements $a_n\in\R$,
and consider the related Bass flat left $\R$\+contramodule~$\B$.
 Then we have $\B/\J\tim\B\cong\overline\B$, since the reduction
functors preserve colimits (see
Sections~\ref{prelim-change-of-scalars}\+-%
\ref{prelim-strongly-closed-ideals}).
 The left $\R$\+contramodule $\B$ is projective by assumption, hence
the left $\T$\+contramodule $\B/\J\tim\B$ is projective, too.
\end{proof}

\begin{cor} \label{discrete-quotients-left-perfect-if-bass-projective}
 Let\/ $\R$ be a complete, separated topological ring with a base of
neighborhoods of zero formed by open right ideals.
 Assume that all Bass flat left\/ $\R$\+contramodules are projective.
 Then the discrete ring\/ $R=\R/\I$ is left perfect for every open
two-sided ideal\/ $\I\subset\R$.
\end{cor}

\begin{proof}
 In view of Lemma~\ref{bass-modules-projective-implies-perfect},
it suffices to show that all Bass flat left $R$\+modules are projective.
 In our assumptions, this is a particular case of
Lemma~\ref{bass-projective-reduction}.
\end{proof}

\begin{lem} \label{bass-covers-reduction}
 Let\/ $\R$ be a complete, separated topological ring with a base of
neighborhoods of zero formed by open right ideals, $\J$ be a strongly
closed two-sided ideal in\/ $\R$, and $\T=\R/\J$ be the topological
quotient ring.
 Assume that all Bass flat left\/ $\R$\+contramodules have projective
covers in\/ $\R\contra$.
 Then all Bass flat left\/ $\T$\+contramodules have projective covers
in\/ $\T\contra$.
\end{lem}

\begin{proof}
 As in the proof of Lemma~\ref{bass-projective-reduction},
let $\bar a_1$, $\bar a_2$, $\bar a_3$,~\dots\ be a sequence of
elements in $\T$ and $\overline\B$ be the related Bass flat
left $\T$\+contramodule.
 Lift the elements $\bar a_n\in\T$ to some elements $a_n\in\R$,
and consider the related Bass flat left $\R$\+contramodule~$\B$.
 Then we have $\B/\J\tim\B\cong\overline\B$.
 By assumption, we know that the left $\R$\+contramodule $\B$ has
a projective cover in $\R\contra$; and by
Lemma~\ref{cover-reduction-lem} it follows that the left
$\T$\+contramodule $\overline\B$ has a projective cover in $\T\contra$.
\end{proof}


\begin{cor} \label{discrete-quotients-left-perfect-if-bass-covers}
 Let\/ $\R$ be a complete, separated topological ring with a base of
neighborhoods of zero formed by open right ideals.
 Assume that all Bass flat left\/ $\R$\+contramodules have projective
covers.
 Then the discrete ring\/ $R=\R/\I$ is left perfect for every open
two-sided ideal\/ $\I\subset\R$.
\end{cor}

\begin{proof}
 In view of Lemma~\ref{bass-modules-projective-implies-perfect} and
Proposition~\ref{projective-covers-of-flat-modules-prop}, it suffices
to show that all Bass flat left $R$\+modules have projective covers.
 In our assumptions, this is a particular case of
Lemma~\ref{bass-covers-reduction}.
 (In the more special case of a topological ring $\R$ with a base of
neighborhoods of zero consisting of open two-sided ideals, one could,
alternatively, use
Corollary~\ref{countable-colimits-of-projective-contramodules},
Proposition~\ref{projective-covers-of-flat-contramodules-prop}, and
Corollary~\ref{discrete-quotients-left-perfect-if-bass-projective}.)
\end{proof}

\Section{Topologically T-Nilpotent Ideals}
\label{t-nilpotent-secn}

 Let $H$ be a separated topological ring without unit.
 We will say that $H$ is \emph{topologically nil} if for any element
$a\in H$ the sequence of elements $a$, $a^2$, $a^3$,~\dots\ converges
to zero in the topology of~$H$.
 Furthermore, we will say that $H$ is \emph{topologically left
T\+nilpotent} if for any sequence of elements $a_1$, $a_2$,
$a_3$,~\dots\ in $H$ the sequence of elements $a_1$, $a_1a_2$,
$a_1a_2a_3$,~\dots, $a_1a_2\dotsm a_n$,~\dots\ converges to zero in
the topology of~$H$.

 For examples of topologically left T\+nilpotent two-sided ideals in
topological rings we refer to Section~\ref{examples-secn} and
Examples~\ref{adeles-example}\+-\ref{backwards-morphisms-example}.

 The following discrete module version of Nakayama lemma is also
a topological version of the condition~(7) and the implication
(7)$\,\Longrightarrow\,$(1) in~\cite[Theorem~P]{Bas}.

\begin{lem} \label{discrete-module-Nakayama}
 Let\/ $\R$ be a complete, separated topological ring with a base of
neighborhoods of zero formed by open right ideals, and let\/
$H\subset\R$ be a left ideal.
 Then\/ $H$ is topologically left T\+nilpotent (as a topological ring
without unit in the topology induced from\/~$\R$) if and only if
for any nonzero discrete right\/ $\R$\+module\/ $\N$ the submodule\/
$\N_H\subset\N$ of all the elements annihilated by\/ $H$ in\/ $\N$
is also nonzero.
\end{lem}

\begin{proof}
 ``Only if'': suppose that $H$ is topologically left T\+nilpotent,
and let $\N$ be a nonzero discrete right $\R$\+module.
 We have to show that $\N$ contains a nonzero element annihilated
by~$H$.
 This implication does not depend on the assumption that $H$ is a left
ideal in $\R$; so $H\subset\R$ can be an arbitrary subring without unit.
 Choose an arbitrary nonzero element $x\in\N$.
 Suppose $x$ is not annihilated by $H$; then there exists an element
$a_1\in H$ such that $xa_1\ne0$ in~$\N$.
 Suppose $xa_1$ is not annihilated by $H$; then there exists
an element $a_2\in H$ such that $xa_1a_2\ne0$ in $\N$, etc.
 Assuming that $\N_H=0$, we can proceed indefinitely in this way
and construct a sequence of elements $(a_n\in H)_{n\ge1}$ such that
$xa_1\dotsm a_n\ne0$ in $\N$ for all $n\ge1$.

 Let $\I\subset\R$ be the annihilator of~$x$; then $\I$ is an open
right ideal in $\R$, so $\I\cap H$ is a neighborhood of zero in~$H$.
 Since $H$ is topologically left T\+nilpotent, there exists $n\ge1$
such that $a_1a_2\dotsm a_n\in\I\cap H$.
 Hence $xa_1\dotsm a_n=0$ in~$\N$.
 The contradiction proves that $\N_H\ne0$.

 ``If'': suppose that $\N_H\ne0$ for every nonzero discrete
right $\R$\+module~$\N$.
 Since $H\subset\R$ is a left ideal, $\N_H$ is an $\R$\+submodule
in~$\N$.
 Whenever $\N_H\ne\N$, one then also has $(\N/\N_H)_H\ne0$.
 Proceeding in a transfinite induction, one constructs a filtration
$0=F_0\N\subset F_1\N\subset F_2\N\subset\dotsb\subset F_\alpha\N=\N$
of the $\R$\+module $\N$ by its $\R$\+submodules, indexed by some
ordinal~$\alpha$, such that $F_{i+1}\N/F_i\N=(\N/F_i\N)_H\ne0$ for
all ordinals $i<\alpha$ and $F_j\N=\bigcup_{i<j}F_i\N$ for all
limit ordinals $j\le\alpha$.

 Now let $(a_n\in H)_{n\ge1}$ be a sequence of elements.
 In order to prove that $H$ is topologically left T\+nilpotent, we have
to show that, for every open right ideal $\I\subset\R$, there exists
$n\ge1$ such that $a_1\dotsm a_n\in\I\cap H$.
 Consider the discrete right $\R$\+module $\N=\R/\I$ and its filtration
$(F_i\N)_{i=0}^\alpha$, as constructed above.
 Let $x\in\N$ denote the image of the element $1\in\R$.

 We follow the argument in the proof of (7)$\,\Longrightarrow\,$(1)
in~\cite[Theorem~P]{Bas}.
 Let $i_0\le\alpha$ be the minimal ordinal such that $x\in F_{i_0}\N$.
 Then $i_0$~cannot be a limit ordinal; so either $i_0=0$, or
$i_0=i'_0+1$ for some ordinal~$i'_0$.
 Since the $\R$\+module $F_{i_0}\N/F_{i'_0}\N$ is annihilated by~$H$,
we have $xa_1\in F_{i'_0}\N$.
 Let $i_1$ be the minimal ordinal such that $xa_1\in F_{i_1}\N$;
then $i_1<i_0$.
 Once again, $i_1$~cannot be a limit ordinal; so either $i_1=0$, or
$i_1=i'_1+1$ for some ordinal~$i'_1$, and then $xa_1a_2\in F_{i'_1}\N$.
 Proceeding in this way, we construct a decreasing chain of
ordinals $i_0>i_1>i_2>\dotsb$, which must terminate.
 Thus there exists $n\ge1$ such that $xa_1\dotsm a_n\in F_0\N=0$,
hence $a_1\dotsm a_n\in\I$.
\end{proof}

 The following version of contramodule Nakayama lemma is
a generalization of~\cite[Lemma~1.3.1]{Pweak} (which is, in turn,
a generalization of~\cite[Lemma~A.2.1]{Psemi}).
 For other versions of contramodule Nakayama lemma,
see~\cite[Lemma~D.1.2]{Pcosh} and~\cite[Lemma~6.14]{PR}.
 For a module version,
see~\cite[Lemma~28.3\,(a)$\,\Longleftrightarrow\,$(b)]{AF}.

\begin{lem} \label{contramodule-Nakayama}
 Let\/ $\R$ be a complete, separated topological ring with a base of
neighborhoods of zero formed by open right ideals, and let\/
$\fH\subset\R$ be a closed left ideal.
 Then\/ $\fH$ is topologically left T\+nilpotent (in the topology
induced from\/~$\R$) if and only if for any nonzero
left\/ $\R$\+contramodule\/ $\C$ one has\/ $\fH\tim\C\ne\C$.
\end{lem}

\begin{proof}
 ``Only if'': the argument follows the proof
of~\cite[Lemma~1.3.1]{Pweak} with an additional consideration
based on the K\H onig lemma (in the spirit of a paragraph
from~\cite[proof of Theorem~2.1]{Bas}).
 We will suppose that $\fH\tim\C=\C$ and prove that $\C=0$ in this case.
 This implication does not depend on the assumption that $\fH$ is
a left ideal in~$\R$; so $\fH\subset\R$ can be an arbitrary closed
subring without unit.

 Indeed, let $b\in\C$ be an element.
 By assumption, the contraaction map $\pi\:\fH[[\C]]\rarrow\C$ is
surjective.
 Let $h\:\C\rarrow\fH[[\C]]$ be a section of the map~$\pi$ (so
$\pi\circ h=\id_\C$).
 Introduce the notation $h(d)=\sum_{c\in\C}h_{d,c}c\in\fH[[\C]]$
for all $d\in\C$, where $h_{d,c}\in\fH$ and the $\C$\+indexed
family of elements $c\longmapsto h_{d,c}$ converges to zero in
the topology of $\fH$ for every $d\in\C$.

 For any set $X$, define inductively $\fH^{(0)}[[X]]=X$ and
$\fH^{(n)}[[X]]=\fH[[\fH^{(n-1)}[[X]]]]$ for $n\ge1$.
 Let $\phi^{(n)}_X\:\fH^{(n)}[[X]]\rarrow\fH[[X]]$ denote the iterated
monad multiplication (``opening of parentheses'') map.
 Set $b_1=h(b)\in\fH[[\C]]$, and define inductively
$b_n=\fH^{(n-1)}[[h]](b_{n-1})\in\fH^{(n)}[[\C]]$ for each $n\ge2$,
where $\fH^{(n-1)}[[h]]\:\fH^{(n-1)}[[\C]]\rarrow\fH^{(n)}[[\C]]$
is the map induced by~$h$.
 Put $a_n=\phi_\C^{(n)}(b_n)\in\fH[[\C]]$ for all $n\ge1$.

 Furthermore, set $q_n=\phi_{\fH[[\C]]}^{(n-1)}(b_n)=\fH[[h]](a_{n-1})
\in\fH[[\fH[[\C]]]]$ for all $n\ge2$.
 Then $\fH[[\pi]](q_n)=a_{n-1}$ and $\phi_\C(q_n)=a_n$.

 The abelian group $\fH[[X]]$ is separated and complete in its natural
topology with a base of neighborhoods of zero formed by the subgroups
$(\I\cap\fH)[[X]]\subset\fH[[X]]$, where $\I\subset\R$ are
open right ideals.
 For any map of sets $f\:X\rarrow Y$, the map $\fH[[f]]$ is continuous
with respect to such topologies on $\fH[[X]]$ and $\fH[[Y]]$.
 Besides, the map $\phi_X\:\fH[[\fH[[X]]]]\rarrow\fH[[X]]$ is
continuous, too, with respect to the above-described topology on
$\fH[[X]]$ and the similar topology of $\fH[[\fH[[X]]]]=\fH[[Y]]$,
where $Y=\fH[[X]]$ is viewed as an abstract set.

 The key observation is that the sequence of elements $a_n$ converges
to zero in the topology of $\fH[[\C]]$ as $n\to\infty$.
 In order to prove this convergence, we will represent the sequence of
elements $b_n\in\fH^{(n)}[[\C]]$ by an infinite rooted tree~$B$
in the following way.
 The root vertex (that is, the only vertex of depth~$0$) is marked by
the element $b\in\C$.
 Its children (i.~e., the vertices of depth~$1$) are marked by all
the elements $c\in\C$, one such child for every element~$c$.
 The edge leading from the root vertex~$b$ to its child~$c$ is marked
by the coefficient $h_{b,c}\in\fH$ in the formal linear combination
$b_1=h(b)=\sum_{c\in\C}h_{b,c}c\in\fH[[\C]]$.

 The element $b_2=\fH[[h]](b_1)\in\fH[[\fH[[\C]]]]$ has the form
$b_2=\sum_{c\in\C}h_{b,c}h(c)$, where $h(c)=\sum_{c_2\in\C}h_{c,c_2}c_2$
for every $c\in\C$.
 The children of a vertex of depth~$1$ marked by~$c$ in the tree $B$
are marked by all the elements $c_2\in\C$; and the edge leading
from~$c$ to~$c_2$ is marked by the element $h_{c,c_2}\in\fH$.

 Generally, the children of any vertex in $B$ are marked by all
the elements of~$\C$; we will write that the children of any fixed
vertex of depth~$n-1$ are marked by all the elements $c_n\in\C$, one
such child for every element $c_n\in\C$.
 Thus, a vertex~$v$ of depth~$n$ in $B$ is characterized by its root
path, which passes from the root vertex~$b$ through vertices marked by
$c_1=c$, $c_2$,~\dots, and comes to the vertex~$v$ marked by $c(v)=c_n$.
 So the set of all vertices of depth~$n$ in $B$ is bijective to
$\C^n=\{(c_1,\dotsc,c_n)\mid c_i\in\C\}$.
 The edge going from a vertex marked by~$c_{n-1}$ to its child marked
by~$c_n$ is marked by the element $h_{c_{n-1},c_n}\in\fH$.

 In addition to marking all the vertices and edges of $B$, let us
also mark all the root paths.
 A root path going from the root vertex~$b$ to a vertex marked
by~$c_1$, to a vertex marked by~$c_2$, etc., and coming to a vertex~$v$
marked by~$c_n$, goes along the edges marked by the elements
$h_{b,c_1}$, $h_{c_1,c_2}$,~\dots, $h_{c_{n-1},c_n}\in\fH$.
 We mark such a root path by the product $r(v)=h_{b,c_1}\dotsm
h_{c_{n-1},c_n}\in\fH$ of the elements marking its edges.

 The purpose of this construction is to observe that the element
$a_n\in\fH[[\C]]$ can be expressed as the infinite sum
$a_n=\sum_{v\in B_n}r(v)c(v)$ over the set $B_n$ of all
vertices of depth~$n$ in the tree~$B$.
 This sum converges in the topology of $\fH[[\C]]$.

 In order to show that the sequence of elements $a_n$ converges to
zero in $\fH[[\C]]$, choose a proper open right ideal $\I\subset\R$.
 Denote by $B^\I$ the subtree of $B$ formed by all the vertices
$v\in B$ with $r(v)\notin\I$.
 The root vertex belongs to $B^\I$, since $1\notin\I$; and whenever
$r(v)\in\I$ for some $v\in B$, one also has $r(w)\in B$ for all
the descendants $w\in B$ of the vertex~$v$; so $B^\I$ is indeed a tree.

 Furthermore, the tree $B^\I$ is locally finite, because for every
vertex $v\in B$ with $c(v)=c_{n-1}$ there exists an open right ideal
$\J\subset\R$ such that $r(v)\J\subset\I$, and the marking element
$h_{c_{n-1},c_n}$ of all but a finite subset of the edges going down
from~$v$ belongs to $\J\cap\fH$ (as $h(c_{n-1})=
\sum_{c_n\in\C}h_{c_{n-1},c_n}c_n$ is an element of $\fH[[\C]]$).
 So, denoting by $vc_n\in B$ the child of~$v$ marked by~$c_n$, we have
$r(vc_n)=r(v)h_{c_{n-1},c_n}\in\I$ for all but a finite subset
of $c_n\in\C$.

 Finally, the tree $B^\I$ has no infinite branches, since the ring
$\fH$ is topologically left T\+nilpotent.
 Indeed, the sequence of the marking elements $r(v_n)$ of the root
paths of the vertices~$v_n$ along any infinite branch in $B$ converges
to zero in~$\fH$, hence $r(v_n)\in\I\cap\fH$ for $n\gg0$.
 By the K\H onig lemma, it follows that the tree $B^\I$ is finite,
so it has a finite depth~$m$.
 Thus $a_n\in(\I\cap\fH)[[\C]]$ for all $n>m$.

 Now we can finish the proof of the ``only if'' assertion of the lemma.
 Since the sequence $a_n$ converges to zero in $\fH[[\C]]$ as
$n\to\infty$, the sequence $q_n=\fH[[h]](a_{n-1})\in\fH[[\fH[[\C]]]]$
converges to zero in the topology of $\fH[[Y]]$, where $Y=\fH[[\C]]$.
 So the sum $\sum_{n=2}^\infty q_n\in\fH[[\fH[[\C]]]]$ is well-defined
as the limit of finite partial sums.
 Furthermore, we have $\fH[[\pi]](q_{n+1})=a_n=\phi_\C(q_n)$ for all
$n\ge2$ and $\fH[[\pi]](q_2)=a_1=b_1$.
 Hence
$$
 \fH[[\pi]]\left(\sum\nolimits_{n=2}^\infty q_n\right)-
 \phi_\C\left(\sum\nolimits_{n=2}^\infty q_n\right) = b_1
$$
(we recall that the maps $\fH[[\pi]]$ and $\phi_\C$ are continuous, as
mentioned in the above discussion).
 Therefore, $b=\pi(b_1)=0$ by the contraassociativity equation,
$$
 \pi\circ(\fH[[\pi]]-\phi_\C)=0.
$$

 ``If'': we suppose that $\fH\tim\C=\C$ implies $\C=0$ for any left
$\R$\+contramodule $\C$, and prove that $\fH$ is topologically
left T\+nilpotent.
 Given a sequence of elements $(a_n\in\fH)_{n\ge1}$, consider
the related Bass flat left $\R$\+contramodule~$\B$.
 Denote by $z_n\in\B$ the images of the free generators~$x_n$,
$n\ge1$, under the surjective left $\R$\+contramodule morphism
$\R[[x_1,x_2,x_3,\dotsc]]\rarrow\B$ (in the notation of the proof
of Lemma~\ref{bass-contramodule-vanishes}).
 Then we have $z_n=a_nz_{n+1}\in\fH\tim\B$ for all $n\ge1$.
 Since $\fH\subset\R$ is a closed left ideal, $\fH\tim\B$ is
an $\R$\+subcontramodule in~$\B$.
 Since the left $\R$\+contramodule $\B$ is generated by its elements
$z_n$, $n\ge1$, it follows that $\fH\tim\B=\B$.
 Thus (by assumption) we have $\B=0$.
 Applying Lemma~\ref{bass-contramodule-vanishes}, we can conclude that
the sequence of elements $a_1\dotsm a_n\in\fH$ converges to zero
in $\R$ as $n\to\infty$.
\end{proof}

\begin{lem} \label{T-nilpotent-in-quotients}
 Let $H$ be a separated topological ring without unit, with a base of
neighborhoods of zero formed by open right ideals, and let $K\subset H$
be a closed two-sided ideal.
 Then $H$ is topologically left T\+nilpotent if and only if both
$K$ and $H/K$ are.
\end{lem}

\begin{proof}
 The ``only if'' assertion is obvious; let us prove the ``if''.
 Let $a_1$, $a_2$, $a_3$,~\dots\ be a sequence of elements in~$H$,
and let $I\subset H$ be an open right ideal.
 Denote by $\bar a_i$ the images of the elements $a_i$ in $H/K$.
 For any open right ideal $J\subset H$, we will denote by
$\overline J\subset H/K$ the image of the ideal~$J$.
 Then $\overline J$ is an open right ideal in~$H/K$.

 Since $H/K$ is topologically left T\+nilpotent, there exists
an integer $n_1\ge1$ such that the product $\bar a_1\dotsm\bar a_{n_1}$
belongs to~$\overline I$.
 Let $J_1\subset H$ be an open right ideal such that
$a_1\dotsm a_{n_1}J_1\subset I$.
 Then there exists an integer $n_2>n_1$ such that the product
$\bar a_{n_1+1}\dotsm\bar a_{n_2}$ belongs to~$\overline J_1$.
 Let $J_2\subset H$ be an open right ideal such that
$a_{n_1+1}\dotsm a_{n_2}J_2\subset J_1$, etc.
 Proceeding in this way, we construct a sequence of integers
$0=n_0<n_1<n_1<n_2<\dotsb$ and open right ideals $I=J_0$, $J_1$,
$J_2$,~\dots~$\subset H$ such that
$\bar a_{n_{m-1}+1}\dotsm\bar a_{n_m}\in\overline J_{m-1}$
and $a_{n_{m-1}+1}\dotsm a_{n_m}J_m\subset J_{m-1}$
for all $m\ge1$.

 For every $m\ge1$, we have $a_{n_{m-1}+1}\dotsm a_{n_m}\in J_{m-1}+K$.
 Choose $b_m\in J_{m-1}$ and $c_m\in K$ such that
$a_{n_{m-1}+1}\dotsm a_{n_m}=b_m+c_m$.
 Since $K$ is topologically left T\+nilpotent, there exists $m\ge1$
such that the product $c_1c_2\dotsm c_m$ belongs to $I\cap K$.
 Now we have
\begin{multline*}
 a_1\dotsm a_{n_m}=(b_1+c_1)\dotsm(b_m+c_m) = c_1c_2\dotsm c_m
 + b_1c_2c_3\dotsm c_m \\ + (b_1+c_1)b_2c_3c_4\dotsm c_m  + \dotsb
 + (b_1+c_1)(b_2+c_2)\dotsm (b_{m-2}+c_{m-2})b_{m-1}c_m \\
 + (b_1+c_1)(b_2+c_2)\dotsm(b_{m-1}+c_{m-1})b_m \\
 \in I\cap K + J_0 + (b_1+c_1)J_1 + \dotsb +
 (b_1+c_1)\dotsm (b_{m-1}+c_{m-1})J_m = I.
\end{multline*}
\end{proof}

\Section{Topological Jacobson Radical} \label{topological-jacobson-secn}

 Let $\R$ be a complete, separated topological associative ring with
a base of neighborhoods of zero formed by open right ideals.
 By an \emph{open maximal right ideal} $\fM\subset\R$ we mean
a maximal right ideal that is simultaneously an open right ideal.
 Notice that this is the same thing as a maximal element in the set
of all proper open right ideals, as any right ideal containing
an open right ideal is open.
 It follows that every proper open right ideal $\I\varsubsetneq\R$ is
contained in an open maximal right ideal $\I\subset\fM\subset\R$.

 We define the \emph{topological Jacobson radical} $\fH$ of
the topological ring $\R$ as the intersection of its open maximal
right ideals.
 This definition was previously discussed in
the paper~\cite[Section~3.B]{IMR}.

\begin{lem} \label{topological-jacobson-twosided}
 For any complete, separated topological associative ring\/ $\R$
with a base of neighborhoods of zero formed by open right ideals,
the topological Jacobson radical\/ $\fH\subset\R$ is a closed
two-sided ideal in\/~$\R$.
\end{lem}

\begin{proof}
 The assertion that $\fH$ is a closed right ideal in $\R$ is obvious
(as any open subgroup is closed, and any intersection of closed right
ideals is a closed right ideal).
 To prove that $\fH$ is a two-sided ideal, consider two elements
$h\in\fH$ and $r\in\R$.
 Given an open maximal right ideal $\fM\subset\R$, we have to show
that $rh\in\fM$.
 Let $\fN\subset\R$ be the set of all elements $q\in\R$ such that
$rq\in\fM$.
 Then $\fN$ is an open right ideal in~$\R$.
 The left multiplication with~$r$ is an injective right $\R$\+module
morphism $r*\:\R/\fN\rarrow\R/\fM$ (taking a coset $s+\fN$, \,$s\in\R$,
to the coset $rs+\fM$).
 Since $\R/\fM$ is a simple right $\R$\+module, it follows that either
$\fN=\R$, or $\R/\fN$ is a simple right $\R$\+module, too.
 In the former case, we have $r\R\subset\fM$, hence $rh\in\fM$.
 In the latter case, $\fN$ is an open maximal right ideal in $\R$,
hence $h\in\fN$ and $rh\in\fM$.

 Alternatively, it suffices to observe that $\fH$ is the intersection
of the annihilators of all the simple discrete right $\R$\+modules
(see Lemma~\ref{topological-jacobson-characterized}(ii) below).
 These are closed two-sided ideals in\/~$\R$.
\end{proof}

 Denote by $H\subset\R$ the Jacobson radical of the ring $\R$ viewed
as an abstract ring (without the topology).
 So $H$ is the intersection of all maximal right ideals in $\R$, and
as the set of all open maximal right ideals is a subset of the set
of all maximal right ideals, it follows that the topological Jacobson
radical of the ring $\R$ contains the nontopological one, that is
$H\subset\fH\subset\R$.

\begin{lem} \label{topological-jacobson-characterized}
 Let\/ $\R$ be a complete, separated topological associative ring with
a base of neighborhoods of zero formed by open right ideals, let\/
$\fH\subset\R$ be the topological Jacobson radical of\/ $\R$, and let
$h\in\R$ be an element.
 Then the following conditions are equivalent:
\begin{enumerate}
\renewcommand{\theenumi}{\roman{enumi}}
\item $h\in\fH$;
\item for every (semi)simple discrete right\/ $\R$\+module\/ $\N$
one has\/ $\N h=0$;
\item for every element $r\in\R$ and every open right ideal\/
$\I\subset\R$ one has $(1-hr)\R+\I=\R$;
\item for every pair of elements $r$, $s\in\R$ and every open right
ideal\/ $\I\subset\R$ one has $(1-shr)\R+\I=\R$.
\end{enumerate}
\end{lem}

\begin{proof}
 (i)$\,\Longleftrightarrow\,$(ii) holds, because the open maximal right
ideals in $\R$ are precisely the annihilators of nonzero elements in
simple discrete right $\R$\+modules (and the semisimple discrete right
$\R$\+modules are the direct sums of simple ones).

 (i)$\,\Longrightarrow\,$(iii) Observe that $(1-hr)\R+\I$ is
an open right ideal in~$\R$.
 If different from $\R$, it must be contained in an open maximal
right ideal~$\fM$.
 Now $1-hr\in\fM$ and $h\in\fM$ imply $1\in\fM$, a contradiction.

 (iii)$\,\Longrightarrow\,$(ii) Suppose $\N$ is a simple discrete
left $\R$\+module and $\N h\ne0$.
 Choose an element $b\in\N$ for which $bh\ne0$ in~$\N$.
 Since $\N$ is simple, one has $bh\R=\N$, so there exists an element
$r\in\R$ such that $bhr=b$.
 Thus we have $b(1-hr)=0$.
 Let $\I=\fM\subset\R$ be the annihilator of the element $b\in\N$.
 Then $\I$ is an open (maximal) right ideal in $\R$ and $(1-hr)\in\I$,
hence $(1-hr)\R+\I=\I\ne\R$.

 The implication (iv)$\,\Longrightarrow\,$(iii) is obvious, and to
prove (iii)$\,\Longrightarrow\,$(iv) it suffices to observe that
the set of all elements $h\in\R$ satisfying~(iii) is a two-sided ideal
(since (i)$\,\Longleftrightarrow\,$(iii) and by
Lemma~\ref{topological-jacobson-twosided}).
\end{proof}

\begin{lem} \label{coperfectness-reformulated}
 Let\/ $\R$ be a complete, separated topological ring with a base of
neighborhoods of zero formed by open right ideals.
 Then the following two conditions are equivalent:
\begin{enumerate}
\renewcommand{\theenumi}{\roman{enumi}}
\item for every sequence of elements $a_1$, $a_2$, $a_3$,~\dots~$\in\R$
and every open right ideal\/ $\I\subset\R$, the chain of right ideals\/
$\R\supset a_1\R+\I\supset a_1a_2\R+\I\supset a_1a_2a_3\R+\I\supset
\dotsb$ terminates;
\item any descending chain of cyclic discrete right\/ $\R$\+modules
terminates.
\end{enumerate}
\end{lem}

\begin{proof}
 One observes that the right $\R$\+module $\R/\I$ is cyclic and discrete
for any open right ideal $\I\subset\R$, and any cyclic discrete right
$\R$\+module has this form.
\end{proof}

\begin{lem} \label{coperfect-implies-nonzero-socle}
 Let\/ $\R$ be a complete, separated topological ring with a base of
neighborhoods of zero formed by open right ideals.
 Assume that every descending chain of cyclic discrete right\/
$\R$\+modules terminates.
 Then any nonzero discrete right\/ $\R$\+module has nonzero socle.
\end{lem}

\begin{proof}
 One easily shows that, in the assumptions of the lemma, any nonzero
discrete right $\R$\+module has a simple submodule.
\end{proof}

\begin{lem} \label{nonzero-socle-implies-jacobson-T-nilpotent}
 Let\/ $\R$ be a complete, separated topological ring with a base of
neighborhoods of zero formed by open right ideals.
 Assume that every nonzero discrete right\/ $\R$\+module has nonzero
socle.
 Then the topological Jacobson radical\/ $\fH\subset\R$ is topologically
left T\+nilpotent.
\end{lem}

\begin{proof}
 According to Lemma~\ref{discrete-module-Nakayama}, it suffices to show
that every nonzero discrete right $\R$\+module $\N$ contains a nonzero
element annihilated by~$\fH$.
 Indeed, let $\cL\subset\N$ be the socle of~$\N$.
 Then $\cL\ne0$ by assumption and $\cL\fH=0$ by
Lemma~\ref{topological-jacobson-characterized}(ii).
\end{proof}

\begin{lem} \label{T-nil-ideal-and-two-jacobson-radicals-lemma}
 Let\/ $\R$ be a complete, separated topological ring with a base of
neighborhoods of zero formed by open right ideals, let\/ $\fH\subset\R$
be the topological Jacobson radical of\/ $\R$, and let $H\subset\R$
be the Jacobson radical of the ring\/ $\R$ viewed as an abstract ring
(without the topology). \par
\textup{(a)} Let\/ $J\subset\R$ be a topologically nil left or
right ideal.
 Then\/ $J\subset H\subset\fH$. \par
\textup{(b)} In particular, if\/ $\fH$ is topologically nil,
then\/ $\fH=H$.
\end{lem}

\begin{proof}
 In part~(a), the inclusion $H\subset\fH$ was explained above in this
section, and the inclusion $J\subset H$ is true because
the element $1-x$ is invertible in $\R$ for every $x\in J$,
as the series $1+x+x^2+x^3+\dotsb$ converges in~$\R$.
 Part~(b) follows from part~(a).
\end{proof}

\begin{cor} \label{coperfect-implies-jacobson-T-nilpotent-cor}
 Let\/ $\R$ be a complete, separated topological ring with a base of
neighborhoods of zero formed by open right ideals.
 Assume that every descending chain of cyclic discrete right\/
$\R$\+modules terminates.
 Then the topological Jacobson radical\/ $\fH$ of the topological ring\/
$\R$ is topologically left T\+nilpotent, the Jacobson radical $H$ of
the ring\/ $\R$ viewed as an abstract ring is closed in\/ $\R$, and\/
$\fH=H$.
\end{cor}

\begin{proof}
 Follows from Lemmas~\ref{coperfect-implies-nonzero-socle},
\ref{nonzero-socle-implies-jacobson-T-nilpotent},
and~\ref{T-nil-ideal-and-two-jacobson-radicals-lemma}(b).
\end{proof}

\begin{lem} \label{topological-jacobson-for-quotient-ring}
 Let\/ $\R$ be a complete, separated topological ring with a base of
neighborhoods of zero formed by open right ideals, $\fK\subset\R$ be
a closed two-sided ideal such that the quotient ring\/ $\S=\R/\fK$ is
complete in its quotient topology, and $p\:\R\rarrow\S$ be the natural
surjective map.
 Let\/ $\fH\subset\R$ and $\J\subset\S$ be the topological Jacobson
radicals of the topological rings\/ $\R$ and\/~$\S$.
 Then \par
\textup{(a)} $p(\fH)\subset\J$; \par
\textup{(b)} if\/ $\fK\subset\fH$, then\/ $\fH=p^{-1}(\J)$.
\end{lem}

\begin{proof}
 If $\fN$ is an open maximal right ideal in $\S$, then $p^{-1}(\fN)$
is an open maximal right ideal in~$\R$ (since $p$~is a surjective
continuous ring homomorphism).
 The set of all right ideals $f^{-1}(\fN)\subset\R$ is contained in
the set of all open maximal right ideals $\fM\subset\R$, hence
$\fH=\bigcap_\fM\fM$ is contained in $p^{-1}(\J)=
\bigcap_\fN f^{-1}(\fN)$.
 Moreover, if $\fK\subset\fH$, then one has $\fK\subset\fM$ for
every~$\fM$.
 It follows that $\fN=p(\fM)$ is an open maximal right ideal in $\S$
(since $p$~is an open map) and $\fM=p^{-1}(\fN)$.
 So the sets of all right ideals $\fM$ and $p^{-1}(\fN)$ in $\R$
coincide in this case.
\end{proof}

 The following lemma is a part of~\cite[Theorem~3.8]{IMR}.

\begin{lem} \label{topological-jacobson-as-projective-limit}
 Let\/ $\R$ be a complete, separated topological ring with a base of
neighborhoods of zero formed by open two-sided ideals.
 For every discrete quotient ring $R=\R/\I$ of the topological ring\/
$\R$, consider the Jacobson radical $H(R)\subset R$ of the ring~$R$.
 Then the topological Jacobson radical\/ $\fH\subset\R$ is
the projective limit\/ $\fH=\varprojlim_{\I\subset\R}H(\R/\I)
\subset\varprojlim_{\I\subset\R}\R/\I=\R$ of the Jacobson radicals
$H(\R/\I)=H(R)$ taken over all the open two-sided ideals\/
$\I\subset\R$.
 Furthermore, the Jacobson radical $H$ of the ring\/ $\R$ viewed as
an abstract ring coincides with the topological Jacobson radical\/ $\fH$
(so, in particular, $H$ is a closed ideal in\/~$\R$).
\end{lem}

\begin{proof}
 The projective limit of Jacobson radicals of the rings $R$ is
well-defined, because for any surjective morphism of discrete
associative rings $f\:R'\rarrow R''$ one has $f(H(R'))\subset H(R'')$.
 To see that this projective limit coincides with the topological
Jacobson radical $\fH\subset\R$, one can use the characterization
provided by Lemma~\ref{topological-jacobson-characterized}(iii) or~(iv),
observing that it suffices to consider open two-sided ideals
$\I\subset\R$ in this characterization whenever such ideals form
a base of neighborhoods of zero in~$\R$.
 For the proof of the second assertion of the lemma we refer
to~\cite[Theorem~3.8\,(2\+-3)]{IMR}.
\end{proof}

\begin{rem}
 The reader should be warned that the assertion
of~\cite[Theorem~3.8\,(4)]{IMR} is erroneous.
 While it is correctly pointed out in~\cite[proof of
Theorem~3.8\,(4)]{IMR} that $\fH$ is the kernel of the natural ring
homomorphism
\begin{equation} \label{IMR-map}
 \R=\varprojlim\nolimits_{\I\subset\R}\R/\I\lrarrow
 \varprojlim\nolimits_{\I\subset\R}(\R/\I)/(H(\R/\I)),
\end{equation}
the map~\eqref{IMR-map} need not be surjective; not even for complete,
separated commutative rings $\R$ with a countable base of neighborhoods
of zero consisting of open ideals.
 Simply put, the problem is that the passage to projective limit
does not preserve surjectivity of homomorphisms of abelian groups.

 Here is a specific counterexample.
 Let $k$ be a field and $L$ be a commutative local $k$\+algebra with
a maximal ideal $m\subset L$ such that $L=k\oplus m$.
 Assume further that $L$ has no nonzero nilpotent elements.
 Notice that, for any commutative ring $K$ without nilpotent elements,
the Jacobson radical of the polynomial ring $K[x]$ vanishes.
 Consider the ring $E=L[x_1,x_2,x_3,\dotsc]$ of polynomials in
a countable set of variables $x_i$, \,$i\ge1$, with coefficients in
the ring~$L$.
 Denote by $E_n=L[x_n,x_{n+1},x_{n+2},\dotsc]=
E/(x_1,\dots,x_{n-1})$ the ring of polynomials in the variables
$x_{n+i}$, \,$i\ge0$, which we view as the quotient ring of the ring $E$
by the ideal generated by $x_1$,~\dots, $x_{n-1}$.

 Let $R$ be the following ring of sequences of polynomials
$(f_1,f_2,f_3,\dotsc)$.
 For every $n\ge1$, we have $f_n\in E_n$.
 The elements of the ring $R$ are all the \emph{eventually
$k$\+sequences} of polynomials $(f_1,f_2,f_3,\dots)$, i.~e., all
such sequences for which $f_n\in k$ for $n$~large enough.
 The addition and multiplication operations are performed on
the sequences of polynomials $(f_1,f_2,f_3,\dotsc)$ termwise; so $R$
is a subring in $\prod_{n\ge1}E_n$.
 Furthermore, let $R'\supset R$ be the ring of \emph{eventually
$L$\+sequences} of polynomials $(f_n\in E_n)_{n\ge1}$, i.~e., all such
sequences for which $f_n\in L$ for $n$~large enough.

 Let $I_j\subset R$ be the ideal consisting of all sequences of
polynomials $(f_1,f_2,f_3,\dotsc)$ such that, for every $n\ge1$,
the polynomial $f_n\in E_n$ belongs to the ideal generated by all
the elements~$x_i$ with $i\ge j$ (or equivalently,
$i\ge\operatorname{max}(n,j)$).
 In other words, one can say that $f_n$~belongs to the image of
the ideal $(x_j,x_{j+1},x_{j+2},\dotsc)\subset E$ under the natural
surjective ring homomorphism $E\rarrow E_n$.
 This condition implies that for every sequence $(f_1,f_2,f_3,\dotsc)
\in I_j$ one has $f_n=0$ for $n$~large enough.
 Obviously, we have $R'\supset R\supset I_1\supset I_2\supset I_3\supset
\dotsb$; and $I_j$ is an ideal in $R'$ as well as in~$R$.
 Denote by $\R$ and $\R'$, respectively, the completions of the rings
$R$ and $R'$ with respect to their topologies in which the ideals $I_j$,
\,$j\ge1$, form a base of neighborhoods of zero.
 Clearly, $\R$ is an open subring in~$\R'$.
 
 The quotient ring $R/I_j$ is the ring of all eventually $k$\+sequences
$(\bar f_1,\bar f_2,\bar f_3,\dotsc)$, where
$\bar f_n\in L[x_n,\dotsc,x_{j-1}]$ for $n<j$, \ $\bar f_n\in L$
for $n\ge j$, and $\bar f_n\in k$ for $n$~large enough.
 Therefore, the Jacobson radical $H(R/I_j)$ consists of all sequences
$(\bar h_1,\bar h_2,\bar h_3,\dotsc)$, where $\bar h_n=0$ for
$n<j$, \ $\bar h_n\in m$ for $n\ge j$, and $\bar h_n=0$ for $n$~large
enough.
 So one has $H(\R)=\fH(\R)=\varprojlim_{j\ge1}H(R/I_j)=0$.
 The key observation is that the derived projective limit
$\varprojlim^1_{j\ge1}H(R/I_j)=\varprojlim^1_{j\ge1}
\bigoplus_{n\ge j} m\cong\prod_{n\ge1}m\big/
\bigoplus_{n\ge1}m$ does not vanish.
 Furthermore, $R'/I_j$ is the ring of all sequences
$(\bar f_1,\bar f_2,\bar f_3,\dotsc)$, where $\bar f_n\in
L[x_n,\dotsc,x_{j-1}]$ for $n<j$ and $\bar f_n\in L$ for $n\ge j$.
 The Jacobson radical $H(R'/I_j)$ consists of all sequences
$(\bar h_1,\bar h_2,\bar h_3,\dotsc)$ with $\bar h_n=0$ for
$n<j$ and $\bar h_n\in m$ for $n\ge j$.
 So $H(\R')=\fH(\R')=\varprojlim_{j\ge1}H(R'/I_j)=0$ and
also $\varprojlim^1_{j\ge1}H(R'/I_j)=\varprojlim^1_{j\ge1}
\prod_{n\ge j}m=0$.
 Notice the isomorphism of $k$\+vector spaces $\varprojlim^1_{j\ge1}
H(R/I_j)\cong R'/R=\R'/\R$.
 
 It is straightforward to compute that the right-hand side
of~\eqref{IMR-map}, viewed as an abstract ring, is isomorphic
to~$\R'$.
 The map~\eqref{IMR-map} is the ring monomorphism $\R\rarrow\R'$,
and it is not an isomorphism.
 The projective limit topology on the right-hand side
of~\eqref{IMR-map}, however, is different from the topology on
the ring $\R'$ arising from the above construction (indeed,
the image of~\eqref{IMR-map} is dense in the projective limit
topology on the right-hand side).

 Now we return to the general case of a complete, separated
topological ring $\R$ with a base of neighborhoods of zero formed
by open two-sided ideals.
 The map~\eqref{IMR-map} has to be distinguished from the ring
homomorphism
\begin{equation} \label{quotient-by-Jacobson}
 \R=\varprojlim\nolimits_{\I\subset\R}\R/\I\lrarrow
 \varprojlim\nolimits_{\I\subset\R}\R/(\I+\fH).
\end{equation}
 In fact, there is a natural surjective homomorphism of discrete rings
$\R/(\I+\fH)\rarrow(\R/\I)/(H(\R/\I))$, but it is not always
an isomorphism, because the map of Jacobson radicals $\fH(R)\rarrow
H(R/\I)$ need not be surjective.
 So we have a commutative triangle diagram of continuous homomorphisms
of topological rings
$$
 \R\lrarrow\varprojlim\nolimits_{\I\subset\R}\R/(\I+\fH)
 \lrarrow\varprojlim\nolimits_{\I\subset\R}(\R/\I)/(H(\R/\I)).
$$

 The right-hand side of~\eqref{quotient-by-Jacobson} is the completion
of the quotient ring $\R/\fH$ in its quotient topology.
 The kernel of the map~\eqref{quotient-by-Jacobson} is the Jacobson
radical $\fH\subset\R$.
 The map~\eqref{quotient-by-Jacobson} is surjective when the topological
ring $\R$ has a countable base of neighborhoods of zero (see
the discussion in Sections~\ref{prelim-quotient-topologies}
and~\ref{prelim-strongly-closed-subgroups}\+-%
\ref{prelim-strongly-closed-ideals}).
 We do \emph{not} know whether the ring
homomorphism~\eqref{quotient-by-Jacobson} is surjective in general.
\end{rem}

\Section{Products of Topological Rings} \label{products-secn}

 Let $\Gamma$ be a set and $(A_\gamma)_{\gamma\in\Gamma}$ be a family
of topological abelian groups, each of them with a base of
neighborhoods of zero $B_\gamma$ consisting of open subgroups.
 The \emph{product topology} on the Cartesian product
$A=\prod_{\gamma\in\Gamma}A_\gamma$ has a base of neighborhoods of zero
formed by the subgroups $\prod_{\delta\in\Delta}U_\delta\times
\prod_{\gamma\in\Gamma\setminus\Delta}A_\gamma$, where
$\Delta\subset\Gamma$ are finite subsets and $U_\delta\in B_\delta$.
 The topological group $A=\prod_{\gamma\in\Gamma}A_\gamma$ does not
depend on the choice of bases of neighborhoods of zero $B_\gamma$
in topological groups~$A_\gamma$.
 When the topological groups $A_\gamma$ are separated, so is
the topological group $\prod_{\gamma\in\Gamma}A_\gamma$.

 If $A'_\gamma\subset A_\gamma$ are subgroups in topological abelian
groups $A_\gamma$ and $A'_\gamma$ are viewed as topological abelian
groups in the induced topology, then the product topology on
$\prod_{\gamma\in\Gamma}A'_\gamma$ coincides with the induced
topology on $\prod_{\gamma\in\Gamma}A'_\gamma\subset
\prod_{\gamma\in\Gamma}A_\gamma$.
 When the subgroups $A'_\gamma$ are closed in $A_\gamma$, so is
the subgroup $\prod_{\gamma\in\Gamma}A'_\gamma\subset
\prod_{\gamma\in\Gamma}A_\gamma$.
 If the quotient groups $A''_\gamma=A_\gamma/A'_\gamma$ are viewed as
topological abelian groups in the quotient topology, then
the product topology on $A''=\prod_{\gamma\in\Gamma}A''_\gamma$ coincides
with the quotient topology on $A''=\prod_{\gamma\in\Gamma}A_\gamma
\big/\prod_{\gamma\in\Gamma}A'_\gamma$.

 Let $(\A_\gamma)_{\gamma\in\Gamma}$ be a family of complete, separated
topological abelian groups.
 Then the group $\A=\prod_{\gamma\in\Gamma}\A_\gamma$ is complete and
separated in the product topology.
 Moreover, for any set $X$ there is a natural isomorphism of
abelian groups
$$
 \A[[X]]\cong\prod\nolimits_{\gamma\in\Gamma}\A_\gamma[[X]].
$$
 If $\fH_\gamma\subset\A_\gamma$ are strongly closed subgroups
then $\fH=\prod_{\gamma\in\Gamma}\fH_\gamma$ is a strongly closed
subgroup in $\A=\prod_{\gamma\in\Gamma}\A_\gamma$
(in the sense of Section~\ref{prelim-strongly-closed-subgroups}).

 If $(R_\gamma)_{\gamma\in\Gamma}$ is a family of topological rings
(with or without unit), then $R=\prod_{\gamma\in\Gamma}R_\gamma$ is
a topological ring (with or without unit, respectively) in
the product topology.
 If each of the topological rings $R_\gamma$ has a base of
neighborhoods of zero formed by open right (resp., two-sided)
ideals, then the ring $R$ also has a base of neighborhoods of
zero formed by open right (resp., two-sided) ideals.
 If $H_\gamma$ are topologically nil (resp., topologically left
T\+nilpotent) separated topological rings without unit, then their
product $H=\prod_{\gamma\in\Gamma}H_\gamma$ in its product topology is
also a topologically nil (resp., topologically left T\+nilpotent)
topological ring without unit.

 The following lemma is the main result of this section.
 Part~(a) is an easy version of part~(b), which is a generalization
of~\cite[Lemma~A.2.2]{Psemi} (see also~\cite[Theorem~4.5]{Sh}).

\begin{lem} \label{product-of-rings-lemma}
 Let $(\R_\gamma)_{\gamma\in\Gamma}$ be a family of complete, separated
topological rings, each of them having a base of neighborhoods of
zero formed by open right ideals; and let\/ $\R=\prod_{\gamma\in
\Gamma}\R_\gamma$ be their product.
 Then \par
\textup{(a)} the coproduct functor $(\N_\gamma)_{\gamma\in\Gamma}
\longmapsto\bigoplus_{\gamma\in\Gamma}\N_\gamma$ establishes
an equivalence between the Cartesian product of the abelian
categories of discrete right\/ $\R_\gamma$\+modules over all\/
$\gamma\in\Gamma$ and the abelian category of discrete right\/
$\R$\+modules; \par
\textup{(b)} the product functor $(\C_\gamma)_{\gamma\in\Gamma}
\longmapsto\prod_{\gamma\in\Gamma}\C_\gamma$ establishes an equivalence
between the Cartesian product of the abelian categories of left\/
$\R_\gamma$\+contramodules over all\/ $\gamma\in\Gamma$ and
the abelian category of left\/ $\R$\+contramodules.
\end{lem}

\begin{proof}
 Part~(a): for every $\gamma\in\Gamma$, denote by
$e_\gamma=(e_{\gamma,\gamma'})_{\gamma'\in\Gamma}\in\R$ the central
idempotent element whose $\gamma'$\+component $e_{\gamma,\gamma'}$
is equal to $0\in\R_{\gamma'}$ for all $\gamma'\in\Gamma$,
$\gamma'\ne\gamma$, and whose $\gamma$\+component $e_{\gamma,\gamma}$
is equal to $1\in\R_\gamma$.
 For any discrete right $\R$\+module $\N$, the subgroup
$\N e_\gamma\subset\N$ is the maximal $\R$\+submodule in $\N$
whose right $\R$\+module structure comes from a (discrete)
right $\R_\gamma$\+module structure via the natural continuous
ring homomorphism $p_\gamma\:\R\rarrow\R_\gamma$.
 So, in the notation of Sections~\ref{prelim-change-of-scalars}
and~\ref{prelim-strongly-closed-ideals}, we have
$p_\gamma^\diamond(\N)=\N e_\gamma$.
 We claim that the functor
$\N\longmapsto(\N_\gamma=\N e_\gamma)_{\gamma\in\Gamma}$ is quasi-inverse
to the functor $(\N_\gamma)_{\gamma\in\Gamma}\longmapsto
\N=\bigoplus_{\gamma\in\Gamma}p_\gamma{}_\diamond\N_\gamma$, where
$\N\in\discr\R$ and $\N_\gamma\in\discr\R_\gamma$.
 In other words, this simply means that any discrete right
$\R$\+module $\N$ is the direct sum of its submodules
$\N e_\gamma\subset\N$.

 Indeed, the idempotents $e_\gamma\in\R$, $\gamma\in\Gamma$ are
orthogonal to each other, which easily implies injectivity of
the map $\bigoplus_{\gamma\in\Gamma}\N e_\gamma\rarrow\N$.
 To prove surjectivity, consider an element $b\in\N$.
 Since $\N$ is a discrete right $\R$\+module by assumption,
there exists a neighborhood of zero $\U\subset\R$ such that $b\U=0$.
 By the definition of the product topology, there exists
a finite subset $\Delta\subset\Gamma$ such that
$\J=\prod_{\gamma\in\Gamma\setminus\Delta}\R_\gamma\subset\U\subset\R$.
 Consider the submodule $\N_\J\subset\N$ of all elements
annihilated by the closed two-sided ideal $\J\subset\R$;
then we have $b\in\N_\J$.
 Now we have $1-\sum_{\delta\in\Delta}e_\delta\in\J$, hence
$b=\sum_{\delta\in\Delta}be_\delta$ is a decomposition of
the element~$b$ into the sum of elements $be_\delta\in\N e_\delta$.

 Part~(b): we keep our notation for the central idempotent elements
$e_\gamma\in\R$.
 For any left $\R$\+contramodule $\C$, the map $e_\gamma\:\C\rarrow
e_\gamma\C$ represents $e_\gamma\C$ as a quotient group of~$\C$.
 This is the maximal quotient $\R$\+contramodule of $\C$ whose
left $\R$\+contramodule structure comes from a left
$\R_\gamma$\+contramodule structure via the homomorphism~$p_\gamma$.
 So, in the notation of Sections~\ref{prelim-change-of-scalars}
and~\ref{prelim-strongly-closed-ideals}, we have
$p_\gamma^\sharp(\C)=e_\gamma\C$.
 We claim that the functor
$\C\longmapsto(\C_\gamma=e_\gamma\C)_{\gamma\in\Gamma}$ is quasi-inverse
to the functor $(\C_\gamma)_{\gamma\in\Gamma}\longmapsto
\C=\prod_{\gamma\in\Gamma}p_\gamma{}_\sharp\C_\gamma$, where
$\C\in\R\contra$ and $\C_\gamma\in\R_\gamma\contra$.
 In other words, this simply means that the natural map
$$
 e=(e_\gamma)_{\gamma\in\Gamma}\:\C\lrarrow
 \prod\nolimits_{\gamma\in\Gamma}e_\gamma\C
$$
is an isomorphism for any left $\R$\+contramodule~$\C$.

 Indeed, let us construct an inverse map to~$e$.
 Given a family of elements $c_\gamma\in e_\gamma\C$, we consider
them as elements of $\C$ and assign to them the element
$$
 f((c_\gamma)_{\gamma\in\Gamma})=
 \pi_\C\left(\sum\nolimits_{\gamma\in\Gamma}e_\gamma c_\gamma\right).
$$
 Here it is important that the family of central idempotent elements
$e_\gamma\in\R$ converges to zero in the topology of $\R$, so
the expression $\sum_{\gamma\in\Gamma}e_\gamma c_\gamma$ defines
an element of the set $\R[[\C]]$ of all convergent infinite formal
linear combinations of elements of $\C$ with the coefficients
in $\R$ (to which the contraaction map $\pi_\C\:\R[[\C]]\rarrow\C$
can be applied).
 To check that $e\circ f=\id$, it suffices to compute, for any
family of elements $(c_\gamma\in\C)_{\gamma\in\Gamma}$ and any fixed
element $\gamma'\in\Gamma$,
$$
 e_{\gamma'}\pi_\C\left(\sum\nolimits_{\gamma\in\Gamma}e_\gamma c_\gamma
 \right) = \pi_\C\left(\sum\nolimits_{\gamma\in\Gamma}e_{\gamma'}e_\gamma
 c_\gamma\right) = e_{\gamma'}c_{\gamma'}
$$
using the contraassociativity equation.
 To check that $f\circ e=\id$, one computes, for any element
$c\in\C$,
$$
 \pi_\C\left(\sum\nolimits_{\gamma\in\Gamma}e_\gamma(e_\gamma c)\right)
 = \left(\sum\nolimits_{\gamma\in\Gamma}e_\gamma\right)c = c
$$
by the contraassociativity equation and because the infinite sum
$\sum_{\gamma\in\Gamma}e_\gamma$ converges to~$1$ in the topology
of~$\R$.
\end{proof}

\Section{Projectivity of Flat Contramodules}
\label{projectivity-of-flats-secn}

 In this section and in the next one, we consider
the following setting.
 Let $\R$ be a complete, separated topological associative ring
with a base of neighborhoods of zero formed by open right ideals.
 Let $\fH\subset\R$ be a strongly closed two-sided ideal in $\R$
(see Sections~\ref{prelim-strongly-closed-subgroups}\+-%
\ref{prelim-strongly-closed-ideals}).
 Assume that the quotient ring $\S=\R/\fH$ is isomorphic, as
a topological ring, to the product $\prod_{\gamma\in\Gamma} S_\gamma$
of a family of discrete rings $S_\gamma$ (viewed as a topological
ring in the product topology), and that every ring $S_\gamma$ is
a classically simple (i.~e., simple Artinian) ring.
 In other words, $S_\gamma$ is the matrix ring of some finite order
over a division ring (for every~$\gamma$).
 Finally, we will also assume that the ideal $\fH$ is topologically
left T\+nilpotent.

\begin{lem} \label{section-8-H-is-the-radical}
 In the above assumptions, $\fH$ is the topological Jacobson radical
of the topological ring\/ $\R$, and\/ $\fH$ coincides with
the Jacobson radical $H$ of the ring\/ $\R$ viewed as an abstract ring
(without the topology).
\end{lem}

\begin{proof}
 One observes that any nonzero element of $\S$ acts nontrivially in
a certain simple discrete right $\S$\+module,
so by Lemma~\ref{topological-jacobson-characterized}(ii)
the topological Jacobson radical of the ring $\S$ vanishes.
 It remains to use
Lemmas~\ref{T-nil-ideal-and-two-jacobson-radicals-lemma}(a)
and~\ref{topological-jacobson-for-quotient-ring}(b).
\end{proof}

 Denote the natural continuous ring homomorphisms by
$p\:\R\rarrow\S$, \ $q_\gamma\:\S\rarrow S_\gamma$, and
$p_\gamma=q_\gamma p\:\R\rarrow S_\gamma$.
 Set $\J_\gamma=\ker(p_\gamma)\subset\R$.
 Recall that, according to the discussion in
Sections~\ref{prelim-change-of-scalars}
and~\ref{prelim-strongly-closed-ideals}, the fully faithful
functor of contrarestriction of scalars $p_\sharp\:\S\contra\rarrow
\R\contra$ has a left adjoint functor of contraextension of
scalars $p^\sharp\:\R\contra\rarrow\S\contra$ computable as
$p^\sharp(\C)=\C/\fH\tim\C$.
 Similarly, the fully faithful functor $p_\gamma{}_\sharp\:S_\gamma\modl
=S_\gamma\contra\rarrow\R\contra$ has a left adjoint functor
$p_\gamma^\sharp\:\R\contra\rarrow S_\gamma\modl$ computable as
$p_\gamma^\sharp(\C)=\C/\J_\gamma\tim\C$.
 The fully faithful functor $q_\gamma{}_\sharp\:S_\gamma\modl\rarrow
\S\contra$ has a left adjoint functor $q_\gamma^\sharp\:\S\contra\rarrow
S_\gamma\modl$, which can be computed in the same fashion.

 Finally, according to Lemma~\ref{product-of-rings-lemma}(b), for
any left $\S$\+contramodule $\D$ we have a natural direct product
decomposition $\D\cong\prod_{\gamma\in\Gamma}
q_\gamma{}_\sharp q_{\gamma}^\sharp\D$.
 So, in particular, for any left $\R$\+contramodule $\C$ one has
$\C/\fH\tim\C\cong\prod_{\gamma\in\Gamma}\C/\J_\gamma\tim\C$.

 The analogous assertions hold for discrete right modules.
 The fully faithful functor of restriction of scalars
$p_\diamond\:\discr\S\rarrow\discr\R$ has a right adjoint functor of
coextension of scalars $p^\diamond\:\discr\R\rarrow\discr\S$ computable
as $p^\diamond(\N)=\N_\fH$.
 The fully faithful functor $p_\gamma{}_\diamond\:\modr S_\gamma
=\discr S_\gamma\rarrow\discr\R$ has a right adjoint functor
$p_\gamma^\diamond\:\discr\R\rarrow\modr S_\gamma$ computable as
$p_\gamma^\diamond(\N)=\N_{\J_\gamma}$.
 The fully faithful functor $q_\gamma{}_\diamond\:\modr S_\gamma\rarrow
\discr\S$ has a right adjoint functor $q_\gamma^\diamond\:\discr\S
\rarrow\modr S_\gamma$, which can be computed similarly.
 Finally, by Lemma~\ref{product-of-rings-lemma}(a), for any discrete
right $\S$\+module $\M$ we have a natural direct sum decomposition
$\M\cong\bigoplus_{\gamma\in\Gamma}
q_\gamma{}_\diamond q_\gamma^\diamond\M$; so, in particular,
for any discrete right $\R$\+module $\N$ one has
$\N_\fH\cong\bigoplus_{\gamma\in\Gamma}\N_{\J_\gamma}$.

\begin{lem} \label{flat-reduction-lemma}
 Let\/ $\R$ be a complete, separated topological associative ring
with a base of neighborhoods of zero formed by open right ideals,
and let\/ $\fH\subset\R$ be a topologically left T\+nilpotent strongly
closed two-sided ideal.
 Let $f\:\F'\rarrow\F''$ be a morphism of flat left\/
$\R$\+contramodules such that the induced morphism of left\/
$\S$\+contramodules $\F'/\fH\tim\nobreak\F'\rarrow\F''/\fH\tim\F''$
is an isomorphism.
 Then the morphism~$f$ is surjective and its kernel is contained in\/
$\bigcap_{\I\subset\R}\I\tim\F'\subset\F'$, where the intersection is
taken over all the open right ideals\/ $\I\subset\R$.
\end{lem}


\begin{proof}
 The conclusion that $f$~is surjective does not depend on
the flatness assumption on $\F'$ and~$\F''$, and only requires
surjectivity of the map $\F'/\fH\tim\nobreak\F'\allowbreak
\rarrow\F''/\fH\tim\F''$.
 It suffices to set $\C=\coker(f)$, observe that $\C/\fH\tim\C=0$,
and apply the contramodule Nakayama Lemma~\ref{contramodule-Nakayama}
in order to conclude that $\C=0$.

 In order to prove the assertion about $\ker(f)$, we will show that
the map of abelian groups $\N\ocn_\R f\:\N\ocn_\R\F'\rarrow\N\ocn_\R
\F''$ is an isomorphism for any discrete right $\R$\+module~$\N$.
 In particular, it will follow that the map $\F'/\I\tim\F'\rarrow
\F''/\I\tim\F''$ is an isomorphism for any open right ideal
$\I\subset\R$, hence $\ker(f)\subset\I\tim\F'\subset\F'$.

 Indeed, for any discrete right $\S$\+module $\M$, one has
$p_\diamond\M\ocn_\R\F'=\M\ocn_\S p^\sharp\F'=\M\ocn_\S p^\sharp\F''=
p_\diamond\M\ocn_\R\F''$ (see Section~\ref{prelim-change-of-scalars}),
so the map $\M\ocn_\R f$ is an isomorphism for any discrete right
$\R$\+module $\M$ annihilated by~$\fH$.
 Now, according to the discrete module Nakayama
Lemma~\ref{discrete-module-Nakayama}, any discrete
right $\R$\+module $\N$ has an increasing filtration
$0=F_0\N\subset F_1\N\subset F_2\N\subset\dotsb\subset F_\alpha\N=\N$,
indexed by some ordinal~$\alpha$, such that the quotient module
$F_{i+1}\N/F_i\N$ is annihilated by $\fH$ for all ordinals $i<\alpha$
and $F_j\N=\bigcup_{i<j}F_i\N$ for all limit ordinals $j\le\alpha$.
 Since the functors of contratensor product with $\F'$ and $\F''$
are exact on the abelian category $\discr\R$ by assumption, and
since they also preserve colimits, it follows by induction on~$i$
that $F_i\N\ocn_\R f$ is an isomorphism for all $0\le i\le\alpha$.
\end{proof}

\begin{thm} \label{projectivity-criterion}
 Let\/ $\R$ be a complete, separated topological associative ring
with a base of neighborhoods of zero formed by open right ideals,
let\/ $\fH\subset\R$ be a topologically left T\+nilpotent strongly
closed two-sided ideal, and let\/ $\S=\R/\fH$ be the quotient ring.
 Let\/ $\F$ be a flat left\/ $\R$\+contramodule.
 Then the left\/ $\R$\+contramodule\/ $\F$ is projective if and only
if the left $\S$\+contramodule\/ $\F/\fH\tim\F$ is projective.
\end{thm}

\begin{proof}
 The functor of contraextension of scalars~$f^\sharp$ with respect to
a continuous homomorphism of topological rings~$f$ always takes
projective contramodules to projective contramodules, since it is
left adjoint to an exact functor of contrarestriction of
scalars~$f_\sharp$ (cf.\ Sections~\ref{prelim-change-of-scalars}
and~\ref{prelim-strongly-closed-ideals}).
 So the ``only if'' assertion is obvious.

 To prove the ``if'', choose a set $X_0$ such that the projective
left $\S$\+contramodule $\Q=\F/\fH\tim\F$ is a direct summand of
the free left $\S$\+contramodule $\S[[X_0]]$.
 Setting $X=\boZ_{\ge0}\times X_0$, so that $\S[[X]]$ is the coproduct
of a countable family of copies of $\S[[X_0]]$ in $\S\contra$, and
using the cancellation trick, one can see that the left
$\S$\+contramodule $\Q\oplus\S[[X]]$ is isomorphic to $\S[[X]]$.

 Consider the left $\R$\+contramodule $\F''=\F\oplus\R[[X]]$ and put
$\Q''=\F''/\fH\tim\F''\cong\Q\oplus\S[[X]]$.
 Then $\Q''$ is a free left $\S$\+contramodule.
 Let us write $\Q''=\S[[Y]]$ (where $Y$ is a subset in $\Q''$
bijective to~$X$).
 Set $\F'=\R[[Y]]$ to be the free left $\R$\+contramodule with
$Y$ generators.
 Then we have natural surjective left $\R$\+contramodule morphisms
$\F''\rarrow p_\sharp\Q''=\S[[Y]]$ and $\F'\rarrow\S[[Y]]$.
 Since $\F'$ is a projective left $\R$\+contramodule, the latter
morphism lifts to a left $\R$\+contramodule morphism
$f\:\F'\rarrow\F''$ satisfying the assumption of
Lemma~\ref{flat-reduction-lemma}.
 (Notice that both the left $\R$\+contramodules $\F'$ and $\F''$
are flat.)
 Thus the morphism~$f$ is surjective with $\ker(f)\subset
\bigcap_{\I\subset\R}\I\tim\F'$.

 Since the left $\R$\+contramodule $\F'$ is projective (and even free)
by construction, the natural map $\F'\rarrow\varprojlim_{\I\subset\R}
\F'/\I\tim\F'$ is an isomorphism (see Section~\ref{prelim-reductions}).
 So one has $\bigcap_{\I\subset\R}\I\tim\F'=0$.
 Hence the morphism~$f$ is an isomorphism.
 We have shown that $\F''$ is a free left $\R$\+contramodule.
 Finally, we can conclude the left $\R$\+contramodule $\F$ is
projective as a direct summand of~$\F''$.
\end{proof}

 The following corollary is a generalization of~\cite[Lemma~A.3]{Psemi}.

\begin{cor} \label{flat-contramodules-are-projective-cor}
 In the assumptions formulated in the beginning of this section, all
flat left\/ $\R$\+contramodules are projective.
\end{cor}

\begin{proof}
 The abelian category $\S\contra\cong\prod_{\gamma\in\Gamma}
S_\gamma\modl$ (see Lemma~\ref{product-of-rings-lemma}(b))
is semisimple in these assumptions.
 So all left $\S$\+contramodules are projective, and the assertion
of the corollary follows from
Theorem~\ref{projectivity-criterion}.
\end{proof}

\Section{Existence of Projective Covers}
\label{existence-projcovers-secn}

 This section contains two proofs of its main result, which is
Theorem~\ref{projective-covers-exist-thm}.
 The first one is very short, consisting only of two references: one
of them to the main result of the previous section, and the other one
to a general theorem from category theory.
 The second proof is longer and more explicit.

\begin{thm} \label{rosicky-covers}
 Let\/ $\sB$ be a locally presentable abelian category with enough
projective objects.
 Assume that the class of all projective objects is closed under
direct limits in\/~$\sB$.
 Then every object of\/ $\sB$ has a projective cover.
\end{thm}

\begin{proof}
 This is a particular case of~\cite[Theorem~2.7, or Corollary~3.7,
or Corollary~4.17]{PR}.
 (This result goes back to~\cite[Theorems~2.1 and~3.1]{Eno}
and~\cite[Theorem~1.2]{Bash}.)
\end{proof}

\begin{cor} \label{flats-are-projective-implies-covers-existence}
 Let\/ $\R$ be a complete, separated topological associative ring with
a base of neighborhoods of zero consisting of open right ideals.
 Assume that all flat left\/ $\R$\+contramodules are projective.
 Then every left\/ $\R$\+contramodule has a projective cover.
\end{cor}

\begin{proof}
 The abelian category $\R\contra$ is
locally presentable~\cite[Section~5]{PR}.
 All the projective left $\R$\+contramodules are flat, and the class of
all flat left $\R$\+contramodules is closed under direct limits (see
Section~\ref{flat-contramodules-secn}).
 Thus the assertion of the corollary follows from
Theorem~\ref{rosicky-covers}.
\end{proof}

 This is essentially all we need for our first proof of
Theorem~\ref{projective-covers-exist-thm}.
 To prepare ground for the second one, we have to address the question
of lifting of idempotents.

 It is a classical fact in the associative ring theory that idempotents
can be lifted modulo any nil ideal.
 The following lemma provides a topological generalization.
 We refer to Section~\ref{t-nilpotent-secn} for the definition of
a \emph{topologically nil} topological ring without unit.

\begin{lem} \label{idempotent-lifting}
 Let\/ $\R$ be complete, separated topological associative ring with
a base of neighborhoods of zero formed by open right ideals.
 Let\/ $\fH\subset\R$ be a topologically nil closed two-sided ideal,
and let $S=\R/\fH$ be the quotient ring.
 Then any idempotent element in $S$ can be lifted to an idempotent
element in\/~$\R$.
\end{lem}

\begin{proof}
 We adopt the argument from~\cite[Tag~00J9]{SP} to the situation at
hand.
 Let $\bar e\in S$ be an idempotent element.
 Choose any preimage $f\in\R$ of the element $\bar e\in S$.
 Proceeding by induction, we construct a sequence of elements
$e_k\in\R$, \,$k\ge0$, starting from $e_0=f$ and passing
from $e_k$ to $e_{k+1}$ by the rule
$$
 e_{k+1}= e_k - (2e_k-1)(e_k^2-e_k)=3e_k^2-2e_k^3, \qquad k\ge0.
$$
 A straightforward computation yields
$$
 e_{k+1}^2-e_{k+1} = (4e_k^2-4e_k-3)(e_k^2-e_k)^2.
$$

 Now let us show that the sequence of elements $e_k\in\R$ converges
in the topology of $\R$ as $k\to\infty$, and that its limit~$e$ is
an idempotent element in $\R$ whose image in $S$ is equal to~$\bar e$.
 Indeed, set $h=f^2-f$; then we have $h\in\fH$, since $\bar e^2-\bar e
=0$ in~$S$.
 Notice that all the elements~$f$, $h$, and $e_k$ belong to the subring
generated by~$f$ in $\R$ over $\boZ$; so they commute with each other.
 It follows from the above formulas by a simple induction on~$k$
that $e_k^2-e_k\in h^{2^k}\R$ for all $k\ge0$.

 Let $\I\subset\R$ be an open right ideal.
 Since the ideal $\fH\subset\R$ is topologically nil, there exists
$n\ge1$ such that $h^n\in\I\cap\fH$.
 Choosing~$m$ such that $2^m\ge n$, we find that $e_{k+1}-e_k=
-(2e_k-1)(e_k^2-e_k)\in\I$ for all $k\ge m$.
 Thus the sequence of elements~$e_k$ converges in $\R$ as $k\to\infty$,
and we can consider its limit $e\in\R$.
 We also have $e_k^2-e_k\in\I$ for all $k\ge m$, hence
$e^2-e\in\I$, and, as this holds for all the open right ideals
$\I\subset\R$, it follows that $e^2-e=0$ in~$\R$.
 Finally, $e_{k+1}-e_k\in h\R\subset\fH$ for all $k\ge0$, hence
$e-f\in\fH$, and therefore the image of~$e$ in~$S$ is equal to~$\bar e$.
\end{proof}

\begin{thm} \label{projective-covers-exist-thm}
 In the assumptions formulated in the beginning of
Section~\ref{projectivity-of-flats-secn}, every
left\/ $\R$\+contramodule has a projective cover.
\end{thm}

\begin{proof}[First proof]
 The assertion follows from 
Corollaries~\ref{flat-contramodules-are-projective-cor}
and~\ref{flats-are-projective-implies-covers-existence}.
\end{proof}

\begin{proof}[Second proof]
 Let us first show that for every left $\S$\+contramodule $\D$ there
exists a projective left $\R$\+contramodule $\P$ such that
the left $\S$\+contramodule $\P/\fH\tim\P$ is isomorphic to~$\D$.
 Indeed, Lemma~\ref{product-of-rings-lemma}(b) applied to the ring
$\S=\prod_{\gamma\in\Gamma}S_\gamma$ implies that any left
$\S$\+contramodule, viewed as an object of $\S\contra$,
is a coproduct of irreducible left $\S$\+contramodules.
 The irreducible left $\S$\+contramodules are indexed by the elements
$\gamma\in\Gamma$ and have the form $q_\gamma{}_\sharp I_\gamma$, where
$I_\gamma$ is the (unique) irreducible left $S_\gamma$\+module.

 One easily finds a (noncentral) idempotent element $i_\gamma\in\S$
such that $q_\gamma{}_\sharp I_\gamma\cong\S i_\gamma$.
 Lifting~$i_\gamma$ to an idempotent element $\tilde i_\gamma\in\R$
using Lemma~\ref{idempotent-lifting}, one can produce a projective
left $\R$\+contramodule $\P_\gamma=\R\tilde i_\gamma$ such that
$\P_\gamma/\fH\tim\P_\gamma\cong\S i_\gamma$.
 Finally, coproducts of projective objects are projective, and
the reduction functor $\C\longmapsto p^\sharp(\C)=\C/\fH\tim\C$
preserves coproducts, which allows to construct a projective left
$\R$\+contramodule $\P$ such that $\P/\fH\tim\P\cong\D$.

 Now let $\C$ be a left $\R$\+contramodule.
 Consider the left $\S$\+contramodule $\D=\C/\fH\tim\C$ and find
a projective left $\R$\+contramodule $\P$ such that $\P/\fH\tim\P
\cong\D$.
 Then we have two surjective left $\R$\+contramodule morphisms
$\P\rarrow p_\sharp\D$ and $\C\rarrow p_\sharp\D$.
 Since $\P\in\R\contra$ is a projective object, we can lift
the former morphism to a left $\R$\+contramodule morphism
$f\:\P\rarrow\C$ such that the induced morphism $\P/\fH\tim\P
\rarrow\C/\fH\tim\C$ is an isomorphism.

 Arguing as in Lemma~\ref{flat-reduction-lemma} and using
the contramodule Nakayama Lemma~\ref{contramodule-Nakayama},
one shows that the map~$f$ is surjective.
 We claim that the morphism~$f$ is a projective cover of a left
$\R$\+contramodule~$\C$.
 Indeed, in view of Lemma~\ref{superfluous-kernel}, it suffices
to check that $\fK=\ker(f)$ is a superfluous $\R$\+subcontramodule
in~$\P$.

 Let $\G\subset\P$ be an $\R$\+subcontramodule such that
$\fK+\G=\P$.
 The morphism of left $\S$\+contramodules
$p^\sharp(\fK)\rarrow p^\sharp(\P/\G)$ is surjective, because
the morphism of left $\R$\+contramodules $\fK\rarrow\P/\G$ is.
 On the other hand, the morphism $p^\sharp(\fK)\rarrow p^\sharp(\P)$
is zero, since the composition $\fK\rarrow\P\rarrow\C$ vanishes and
the morphism $p^\sharp(\P)\rarrow p^\sharp(\C)$ is an isomorphism.
 Therefore, the composition $p^\sharp(\fK)\rarrow p^\sharp(\P)\rarrow
p^\sharp(\P/\G)$ also vanishes.
 It follows that $p^\sharp(\P/\G)=0$, that is $\P/\G=\fH\tim(\P/\G)$.
 Applying Lemma~\ref{contramodule-Nakayama} again, we conclude
that $\P/\G=0$, as desired.
\end{proof}

\Section{Proof of Main Theorem} \label{proof-of-main-theorem-secn}

 Let $\R$ be a complete, separated topological associative ring with
a base of neighborhoods of zero formed by open two-sided ideals.
 We will need to assume that one of the following three conditions
holds:
\begin{enumerate}
\renewcommand{\theenumi}{\alph{enumi}}
\item the ring $\R$ is commutative; or
\item $\R$ has a countable base of neighborhoods of zero; or
\item $\R$ has only a finite number of classically
semisimple (semisimple Artinian) discrete quotient rings.
\end{enumerate}

 The following theorem is the main result of this paper.

\begin{thm} \label{main-theorem}
 Let\/ $\R$ be a complete, separated topological associative ring with
a base of neighborhoods of zero formed by open two-sided ideals.
 Assume that one of the conditions (a), (b), or~(c) is satisfied.
 Then the following conditions are equivalent:
\begin{itemize}
\item[(i)] all flat left\/ $\R$\+contramodules have projective covers;
\item[(i$^\flat$)] all Bass flat left\/ $\R$\+contramodules have
projective covers;
\item[(ii)] all left $\R$\+contramodules have projective covers;
\item[(iii)] all flat left\/ $\R$\+contramodules are projective;
\item[(iii$^\flat$)] all Bass flat left\/ $\R$\+contramodules
are projective;
\item[(iv)] $\R$ has a topologically left T\+nilpotent strongly closed
two-sided ideal\/ $\fH$ such that the quotient ring\/ $\R/\fH$ is
topologically isomorphic to a product of simple Artinian discrete
rings endowed with the product topology;
\item[(v)] all descending chains of cyclic discrete right\/
$\R$\+modules terminate;
\item[(vi)] all the discrete quotient rings of\/ $\R$ are left perfect.
\end{itemize}
\end{thm}

\begin{proof}
 The implications
(ii)$\,\Longrightarrow\,$(i)$\,\Longrightarrow\,$(i$^\flat$)
and (iii)$\,\Longrightarrow\,$(iii$^\flat$) are obvious.
 So are the implications (iii)$\,\Longrightarrow\,$(i) and
(iii$^\flat$)$\,\Longrightarrow\,$(i$^\flat$).

 For any complete, separated topological ring $\R$ with a countable
base of neighborhoods of zero formed by open right ideals, any
left $\R$\+contramodule has a flat cover~\cite[Corollary~7.9]{PR}.
 Hence the condition~(iii) implies~(ii) under the assumption of~(b).
 Moreover, Corollary~\ref{flats-are-projective-implies-covers-existence}
provides the implication (iii)$\,\Longrightarrow\,$(ii) for any
complete, separated topological ring $\R$ with a base of neighborhoods
of zero formed by open right ideals.
 (But we do not need to use either of these arguments.)

 The condition~(iv) was already formulated in the beginning of
Section~\ref{projectivity-of-flats-secn}.
 In particular, by Lemma~\ref{section-8-H-is-the-radical},
for any complete, separated topological ring $\R$ with a base of
neighborhoods of zero formed by open right ideals the condition~(iv)
implies that $\fH$ is the topological Jacobson radical of the ring $\R$,
and that $\fH$ also coincides with the Jacobson radical of the ring $\R$
viewed as an abstract ring.

 The implications (iv)$\,\Longrightarrow\,$(iii) and
(iv)$\,\Longrightarrow\,$(ii) are provided by
Corollary~\ref{flat-contramodules-are-projective-cor} and
Theorem~\ref{projective-covers-exist-thm}, respectively, and
hold for any complete, separated topological ring $\R$ with
a base of neighborhoods of zero formed by open right ideals.

 The implication (iii$^\flat$)$\,\Longrightarrow\,$(vi) is provided
by Corollary~\ref{discrete-quotients-left-perfect-if-bass-projective},
and the implication (i$^\flat$)$\,\Longrightarrow\,$(vi) by
Corollary~\ref{discrete-quotients-left-perfect-if-bass-covers}.
 Using the assumption of open two-sided ideals forming a base of
neighborhoods of zero in $\R$, one can obtain the implication
(i$^\flat$)$\,\Longrightarrow\,$(iii$^\flat$) from
Corollary~\ref{countable-colimits-of-projective-contramodules}
and Proposition~\ref{projective-covers-of-flat-contramodules-prop}.

 The implication (iii$^\flat$)$\,\Longrightarrow\,$(v) is provided by
Proposition~\ref{bass-projective-implies-coperfect} with
Lemma~\ref{coperfectness-reformulated}.
 The implication (v)$\,\Longrightarrow\,$(vi) follows straightforwardly
from the characterization of left perfect rings in terms of
the descending chain condition on principal right ideals (as
in~\cite[Theorem~P(6)]{Bas} and~\cite[Theorem~28.4(e)]{AF}), as any
right module over a discrete quotient ring of $\R$ is a discrete
right module over~$\R$.
 The converse implication (vi)$\,\Longrightarrow\,$(v) holds under
the assumption of two-sided ideals forming a base of neighborhoods of
zero in $\R$, as any finitely generated discrete right $\R$\+module
is a module over a discrete quotient ring of $\R$ in this case.

 It is the implication (vi)$\,\Longrightarrow\,$(iv) that needs
both the assumption that $\R$ has a base of neighborhoods of zero
formed by open two-sided ideals and one of the conditions~(a),
(b), or~(c).
 Assuming~(vi) and denoting by $H(R)$ the Jacobson radical
(\,$=$~nilradical) of a left perfect discrete ring $R$, we set
$$
 \fH=\varprojlim\nolimits_{\I\subset\R}H(\R/\I)\subset\R,
$$
where the projective limit is taken over all the open two-sided ideals
$\I$ in~$\R$, to be the (topological) Jacobson radical of the ring $\R$
(see Lemma~\ref{topological-jacobson-as-projective-limit}).
 Notice that for any surjective morphism of left perfect rings
$f\:R'\rarrow R''$ one has $f(H(R'))=H(R'')$.
 In order to finish the proof of the theorem, it remains to
apply the next proposition.
\end{proof}

\begin{prop} \label{properfect-prop}
 Let\/ $\R$ be a complete, separated topological ring with a base
of neighborhoods of zero formed by open two-sided ideals such that
all the discrete quotient rings of\/ $\R$ are left perfect.
 Assume that one of the conditions (a), (b), or~(c) is satisfied.
 Then\/ $\fH=\varprojlim_{\I\subset\R}H(\R/\I)$ is a topologically left
T\+nilpotent strongly closed two-sided ideal in\/ $\R$, and
the quotient ring\/ $\S=\R/\fH$ is topologically isomorphic to
a product of simple Artinian discrete rings endowed with
the product topology.
\end{prop}

\begin{proof}
 The two-sided ideal $\fH\subset\R$ is closed by construction and,
viewed as a topological ring without unit, it is topologically
left T\+nilpotent as the projective limit of T\+nilpotent
discrete rings without unit.
 Alternatively, one can use the discussion in
Section~\ref{topological-jacobson-secn}, and deduce the topological
left T\+nilpotency of the topological Jacobson radical $\fH$ from
the condition~(v) and
Corollary~\ref{coperfect-implies-jacobson-T-nilpotent-cor}.

 In order to prove the remaining assertions, let us consider
the three cases separately.

 (b) First of all, any closed subgroup in a topological
abelian group with a countable base of neighborhoods of zero is
strongly closed (see Lemma~\ref{countable-base-strongly-closed}).

 Furthermore, for any discrete quotient ring $R=\R/\I$ of
the topological ring $\R$, we have a short exact sequence
$$
 0\lrarrow H(R)\lrarrow R\lrarrow R/H(R)\lrarrow0.
$$
 The transition maps in the projective system $(H(\R/\I))_{\I\subset\R}$
are surjective, so passing to the (countable filtered) projective limit
we get a short exact sequence
$$
0\lrarrow\fH\lrarrow\R\lrarrow\S=
 \varprojlim\nolimits_{\I\subset\R}R/H(R)\lrarrow0.
$$
 This proves that the topological ring $\S$ is the topological
projective limit of the countable filtered projective system of
semisimple Artinian discrete rings $R/H(R)$ and surjective
morphisms between them.
 All such ring homomorphisms are projections onto direct factors,
and it follows that $\S$ is a topological product of simple
Artinian discrete rings.

 (c) Let $\J_1$ and $\J_2\subset\R$ be two open two-sided
ideals such that the quotient rings $\R/\J_1$ and $\R/\J_2$ are
semisimple Artinian.
 Since $R=\R/(\J_1\cap\J_2)$ is a left perfect discrete ring by
assumption, we have $H(R)\subset \J_1/(\J_1\cap\J_2)$ and
$H(R)\subset\J_2/(\J_1\cap\J_2)$, so $H(R)=0$ and $R$ is
a semisimple Artinian ring, too.
 Since $\R$ only has a finite number of semisimple Artinian
discrete quotient rings, it follows that there exists
a unique minimal open two-sided ideal $\J\subset\R$ such that
$\R/\J$ is semisimple Artinian.

 Now if $\I\subset\J\subset\R$ is an open two-sided ideal, then
$H(\R/\I)=\J/\I$.
 Thus we have $\fH=\J$, so $\fH$ is an open (hence strongly closed)
two-sided ideal in $\R$ and the quotient ring $\R/\fH$ is a finite
product of simple Artinian rings.

 (a) Any perfect commutative ring uniquely decomposes as a finite
product of perfect commutative local rings, while any semisimple
commutative ring uniquely decomposes as a finite product of fields.

 Let $\Gamma$ be set of all open ideals $\G\subset\R$ such that
the discrete quotient ring $\R/\G$ is a field.
 Furthermore, consider the set $\Xi$ of all open ideals $\J\subset\R$
such that the discrete quotient ring $\R/\J$ is local.
 For any such ideal $\J$, there exists a unique ideal $\G\in\Gamma$
for which $\J\subset\G$.
 Denote by $\Xi_\G$ the disjoint union of the set $\{\J\in\Xi\mid
\J\subset\G\}$ with the one-point set $\{\R\}$ consisting of
the unit ideal of~$\R$.
 We will say that the element $\R$ is the \emph{marked point} of
the set $\Xi_\G$.
 So the set $\Xi$ is the disjoint union of the sets $\Xi_\G\setminus
\{\R\}$ over all open maximal ideals $\G\in\Gamma$.

 Now for any open ideal $\I\subset\R$ the subset $\Delta_\I\subset
\Gamma$ of all $\G\in\Gamma$ such that $\I\subset\G$ is finite
(and bijective to the spectrum of $\R/\I$).
 Furthermore, there exists a unique collection of open ideals
$\I_\G\in\Xi_\G\setminus\{\R\}$, \,$\G\in\Delta_\I$, such that
$\I\subset\I_\G$ and the natural ring homomorphism
$$
 \R/\I\lrarrow\prod\nolimits_{\G\in\Delta_\I}\R/\I_\G
$$
is an isomorphism.
 Conversely, for any finite subset $\Delta\subset\Gamma$ and any
collection of open ideals $\I_\G\in\Xi_\G\setminus\{\R\}$,
\,$\G\in\Delta$, the intersection $\I=\bigcap_{\G\in\Delta}\I_\G$
is an open ideal in $\R$ and the map
$\R/\I\rarrow\prod_{\G\in\Delta}\R/\I_\G$ is an isomorphism.

 Given an open ideal $\I\subset\R$ and an element $\G\in\Gamma\setminus
\Delta_\I$, we put $\I_\G=\R\in\Xi_\G$.
 We have constructed a bijective correspondence between the set of all
open ideals $\I\subset\R$ and the set of all \emph{finitely supported}
elements of the product $\prod_{\G\in\Gamma}\Xi_\G$.
 Here an element $(\I_\G)_{\G\in\Gamma}\in\prod_{\G\in\Gamma}\Xi_\G$
is said to be finitely supported if one has $\I_\G=\R$ for all but
a finite subset of indices $\G\in\Gamma$.

 For any two open ideals $\J'$, $\J''\in\Xi_\G$, the intersection
$\J'\cap\J''$ belongs to $\Xi_\G$ again.
 Hence $\Xi_\G$ is a directed poset with respect to inverse inclusion,
with a minimal element which is the marked point $\R\in\Xi_\G$.
 The local rings $\R/\J$, \,$\J\in\Xi_\G$, form a projective system
indexed by the directed poset $\Xi_\G$, and we can form
their projective limit 
$$
 \R_\G=\varprojlim\nolimits_{\J\in\Xi_\G}\R/\J,
$$
endowing it with the projective limit topology.
 There is a natural ring homomorphism $\R\rarrow\R_\G$ whose
compositions with the projections $\R_\G\rarrow\R/\J$ are surjective,
so these projections are surjective, too.
 In particular, there is a natural surjective ring homomorphism
$\R_\G\rarrow\R/\G$, whose open kernel we denote by
$\fH_\G\subset\R_\G$.

 The poset of all open ideals $\I\subset\R$ is isomorphic to the subposet
in the product $\prod_{\G\in\Gamma}\Xi_\G$ consisting of all the finitely
supported elements.
 Moreover, the projective system of quotient rings $\R/\I$ can be
recovered as the finitely supported product of the projective systems
$\R/\J$, \,$\J\in\Xi_\G$.
 The latter words mean that we have a discrete ring isomorphism
$\R/\I\cong\prod_{\G\in\Gamma}\R/\I_\G$ for all open ideals
$\I\subset\R$, where $\I_\G=\R$ and so $\R/\I_\G=0$ for all
$\G$ outside of the finite subset $\Delta_\G\subset\Gamma$; and such
isomorphisms agree with the transition maps in the projective systems.

 It follows from these considerations that the topological ring $\R$
decomposes as the product of topological rings $\R_\G$,
$$
 \R\cong\prod\nolimits_{\G\in\Gamma}\R_\G,
$$
and the topology on $\R$ coincides with the product topology.
 Furthermore, under this isomorphism one has
$$
 \fH=\prod\nolimits_{\G\in\Gamma}\fH_\G.
$$
 Now the ideal $\fH\subset\R$ is strongly closed as a product of
open ideals $\fH_\G\subset\R_\G$ (cf.\ the discussion in
the beginning of Section~\ref{products-secn}), and the quotient ring
$$
 \S=\R/\fH\cong\prod\nolimits_{\G\in\Gamma}\R_\G/\fH_\G=
 \prod\nolimits_{\G\in\Gamma}\R/\G
$$
is the topological product of discrete fields.
\end{proof}

 We will say that a topological ring $\R$ is \emph{left pro-perfect}
if it is separated and complete, has a base of neighborhoods of zero
consisting of open two-sided ideals, and all the discrete quotient
rings of $\R$ are left perfect.
 According to Theorem~\ref{main-theorem}, over a left pro-perfect
topological ring satisfying one of the conditions (a), (b), or~(c)
all left contramodules have projective covers and all flat left
contramodules are projective.
 Conversely, any complete, separated topological ring with a base
of neighborhoods of zero consisting of open two-sided ideals over
which all Bass flat left contramodules have projective covers is
pro-perfect.

\Section{Examples} \label{examples-secn}

 Three (classes of) examples of pro-perfect topological rings are
discussed below.
 The former two are commutative topological rings, while the third
one is not.

\begin{ex}
 Let $\R$ be a complete Noetherian commutative local ring with
the maximal ideal $\m\subset\R$.
 We view $\R$ as a topological ring in the $\m$\+adic topology.
 Then $\R=\varprojlim_{n\ge1}\R/\m^n$ is a separated and complete
topological ring with a base of neighborhoods of zero formed
by the ideals $\m^n\subset\R$.
 Furthermore, $\R$ is pro-perfect, as all of its discrete quotient
rings are Artinian and consequently perfect.
 The maximal ideal $\m\subset\R$ is strongly closed and topologically
T\+nilpotent.

 By~\cite[Theorem~B.1.1]{Pweak} or~\cite[Example~2.2(4)]{Pper},
the forgetful functor $\R\contra\rarrow\R\modl$ is fully faithful,
so the abelian category of $\R$\+contramodules is a full subcategory
in the category of arbitrary $\R$\+modules.
 This full subcategory consists of all the so-called
\emph{$\m$\+contramodule $\R$\+modules}, which means
the $\R$\+modules $C$ such that $\Ext^i_\R(\R[s^{-1}],C)=0$ for
all $i=0$, $1$ and all $s\in\m$.
 It suffices to check this condition for any chosen set of generators
$s_1$,~\dots, $s_m\in\m$ of the ideal~$\m$ (or of any ideal in $\R$
whose radical is equal to~$\m$) by~\cite[Theorem~5.1]{Pcta}.

 According to Theorem~\ref{main-theorem}, all $\R$\+contramodules have
projective covers and all flat left $\R$\+contramodules are projective.
 Let us explain how to obtain these results from the previously
existing literature.
 An $\R$\+contramodule is flat if and only if it is flat as
an $\R$\+module~\cite[Lemma~B.9.2]{Pweak},
\cite[Corollary~10.3(a)]{Pcta}.
 All flat $\R$\+contramodules are projective
by~\cite[Corollary~B.8.2]{Pweak} or~\cite[Theorem~10.5]{Pcta}.
 Moreover, the projective objects of the category $\R\contra$ are
precisely the free $\R$\+contramodules $\R[[X]]=
\varprojlim_{n\ge1}(\R/\m^n)[[X]]$ (see~\cite[Lemma~1.3.2]{Pweak}
or~\cite[Corollary~10.7]{Pcta}).

 Concerning the projective covers, one observes that all
$\R$\+contramodules are Enochs cotorsion
$\R$\+modules~\cite[Proposition~B.10.1]{Pweak},
\cite[Theorem~9.3]{Pcta}.
 Let $\C$ be an $\R$\+contramodule, and let $f\:F\rarrow\C$ be
a flat cover of the $\R$\+module~$\C$.
 Let $p\:\P\rarrow\C$ be a surjective morphism onto $\C$ from
a projective $\R$\+contramodule $\P$ with the kernel~$\fK$.
 Then $\P$ is also a flat $\R$\+module, while $\fK$ is
a cotorsion $\R$\+module; so $p$~is a special flat precover of
the $\R$\+module~$\C$.
 It follows that the $\R$\+module $F$ is a direct summand of~$\P$;
hence $F$ is also an $\R$\+contramodule.
 Thus the morphism~$f$ is a projective cover of $\C$ in
the category $\R\contra$.
\end{ex}

\begin{ex} \label{S-completion-example}
 Let $R$ be a commutative ring and $S\subset R$ be a multiplicative
subset.
 The \emph{$S$\+topology} on an $R$\+module $M$ has a base of
neighborhoods of zero formed by the $R$\+submodules $sM\subset M$,
where $s\in S$.
 In particular, the ring $R$ itself is a topological ring in
the $S$\+topology.
 Let $\R=\varprojlim_{s\in S}R/sR$ be its completion, endowed
with the projective limit topology~\cite[Section~2]{PMat}.
 Then $\R$ is a complete, separated topological commutative ring 
with a base of neighborhoods of zero formed by open ideals.

 Assume that the quotient ring $R/sR$ is perfect for all $s\in S$.
 Then $\R$ is a pro-perfect commutative topological ring, so
Theorem~\ref{main-theorem} tells that all $\R$\+contramodules have
projective covers and all flat $\R$\+contramodules are projective.
 These results do not seem to follow easily from the previously
existing literature.

 Let us discuss the category of $\R$\+contramodules $\R\contra$ in
some more detail.
 Following the proof of Proposition~\ref{properfect-prop}(a),
the topological ring $\R$ decomposes as the topological product
$\R\cong\prod_{\G\in\R}\R_\G$ over the open ideals $\G\subset\R$
such that the quotient ring $\R/\G$ is a field.
 (The same argument allows to obtain such a decomposition in
the slightly more general case of an \emph{$S$\+h-nil} ring~$R$
\cite[Section~6]{BP}.)
 Such open ideals $\G\subset\R$ correspond bijectively to the maximal
ideals $\m\subset R$ for which the intersection $\m\cap S$ is
nonempty, and the topological ring $\R_\G$ can be described as
the $S$\+completion of the localization $R_\m$ of the ring~$R$.
 By Lemma~\ref{product-of-rings-lemma}(b), the abelian category
$\R\contra$ decomposes as the Cartesian product of the abelian
categories $\R_\G\contra$.

 By~\cite[Theorem~6.13]{BP}, the localization $R_S$ of the ring $R$
at the multiplicative subset $S$, viewed as an $R$\+module, has
projective dimension at most~$1$.
 Thus the full subcategory of \emph{$S$\+contramodule $R$\+modules}
$R\modl_{S\ctra}\subset R\modl$, consisting of all the $R$\+modules $C$
such that $\Ext_R^i(R_S,C)=0$ for $i=0$ and~$1$, is an abelian
category, and the identity embedding $R\modl_{S\ctra}\rarrow R\modl$
is an exact functor~\cite[Theorem~3.4(a)]{PMat}.
 The forgetful functor $\R\contra\rarrow R\modl$ factorizes as
$\R\contra\rarrow R\modl_{S\ctra}\rarrow R\modl$
\cite[Example~2.4(2)]{Pper}.
 We studied the category $R\modl_{S\ctra}$
in~\cite[Sections~4 and~6]{BP}.

 Still, the functor $\R\contra\rarrow R\modl_{S\ctra}$ is \emph{not}
an equivalence of categories, generally
speaking~\cite[Example~1.3(6)]{Pper}.
 A sufficient condition for it to be an equivalence is that
the $S$\+torsion of $R$ be bounded \cite[Example~2.4(3)]{Pper}.
 More generally, using the decomposition of the category
$\R\contra$ into the Cartesian product over the maximal ideals~$\m$
of the ring $R$ with $\m\cap S\ne\varnothing$ and the similar
decomposition of the category $R\modl_{S\ctra}$
\cite[Corollary~6.15]{BP}, one shows that the functor
$\R\contra\rarrow R\modl_{S\ctra}$ is an equivalence whenever
the $S$\+torsion in $R_\m$ is bounded for every~$\m$.
 When $S$ is countable, the functor $\R\contra\rarrow R\modl_{S\ctra}$
is an equivalence of categories if and only if the $S$\+torsion in
$\R=\varprojlim_{s\in S}R/sR$ is bounded~\cite[Example~5.4(2)]{Pper}.

 Furthermore, \cite[Example~3.7(1)]{Pper} lists two conditions
which, taken together, are sufficient for the functor
$\R\contra\rarrow R\modl_{S\ctra}$ to be fully faithful.
 By~\cite[Proposition~2.1(2)]{BP}, every $S$\+divisible $R$\+module
is $S$\+h-divisible, so one of the two conditions always holds in
our case.
 Hence the functor $\R\contra\rarrow R\modl_{S\ctra}$ is fully faithful
whenever the other condition holds, that is, whenever for every set
$X$ the free $\R$\+contramodule $\R[[X]]=\varprojlim_{s\in S}(R/sR)[X]$
is complete in its $S$\+topology (or in other words, its $S$\+topology
coincides with its projective limit topology~\cite[Theorem~2.3]{PMat}).

 In particular, the functor $\R\contra\rarrow R\modl_{S\ctra}$ is
fully faithful whenever $S$ is countable~\cite[Example~3.7(2)]{Pper}.
 In this case, every object of $R\modl_{S\ctra}$ is an extension of
two objects from $\R\contra$ \cite[Example~5.4(2)]{Pper}.
\end{ex}

\begin{ex}
 Let $\cC$ be a coassociative, counital coalgebra over a field~$k$.
 Then the dual vector space $\cC^*=\Hom_k(\cC,k)$ to the coalgebra $\cC$
has a natural structure of complete, separated topological ring
(in fact, topological algebra over~$k$) with a base of neighborhoods
of zero formed by open two-sided ideals.
 Moreover, any coassociative coalgebra over a field is the union of
its finite-dimensional subcoalgebras; so $\cC^*$ is the projective
limit of a directed diagram of finite-dimensional algebras (and
surjective morphisms between them).
 The discrete quotient rings of $\cC^*$ are the finite-dimensional
algebras dual to the finite-dimensional subcoalgebras of~$\cC$
\,\cite[Sections~1.1 and~2.2]{Swe}, \cite[Section~1.3]{Prev}.

 Since all finite-dimensional algebras are perfect, the topological
ring $\cC^*$ is pro-perfect.
 All closed $k$\+vector subspaces in the pro-finite-dimensional
vector space $\cC^*$ are strongly closed (see
Example~\ref{pro-finite-dimensional}).
 Moreover, denoting by $\cD\subset\cC$ the maximal cosemisimple
subcoalgebra in $\cC$, one observes that the coalgebra $\cD$ is
a direct sum of cosimple coalgebras, while the quotient coalgebra
without counit $\cC/\cD$ is conilpotent.
 In other words, this means that the closed two-sided ideal
$\fH=(\cC/\cD)^*\subset\cC^*$ is not only topologically left
T\+nilpotent, but even \emph{topologically nilpotent}, that is,
for any  neighborhood of zero $\U\subset\cC^*$ there exists
an integer $n\ge1$ such that $\fH^n\subset\U$.
 The quotient ring $\cC^*/\fH\cong\cD^*$ is a product of
simple finite-dimensional $k$\+algebras with the product
topology~\cite[Sections~8.0 and 9.0\+-9.1]{Swe},
\cite[Section~A.2]{Psemi}.

 Thus the condition~(iv) of Theorem~\ref{main-theorem} is satisfied
for~$\cC^*$.
 As we will see below in Theorem~\ref{generalized-main-theorem}, it
follows that the topological ring $\cC^*$ satisfies the condition~(d)
(though it does not need to satisfy~(a), (b), or~(c)).

 A discrete right module over $\cC^*$ is the same thing as
a right $\cC$\+comodule~\cite[Section~2.1]{Swe},
\cite[Section~1.4]{Prev}.
 All $\cC$\+comodules are unions of their finite-dimensional
subcomodules.
 Left contramodules over the topological ring $\cC^*$ are otherwise
known as left contramodules over the coalgebra~$\cC$
\,\cite[Section~2.3]{Prev}.
 According to the proof of Theorem~\ref{main-theorem}, it follows that
all flat $\cC$\+contramodules are projective (cf.~\cite[Sections~0.2.9
and~A.3]{Psemi}) and all $\cC$\+contramodules have projective covers.
\end{ex}

\Section{Generalization of Main Theorem}
\label{gener-main-theorem-secn}

 The aim of this section is to generalize the result of
Theorem~\ref{main-theorem} so that the class of topological rings
covered by its equivalent conditions includes all the rings
satisfying the assumptions formulated in the beginning of
Section~\ref{projectivity-of-flats-secn}.

 Let $\R$ be a complete, separated topological ring $\R$ with
a base of neighborhoods of zero formed by open right ideals.
 We are interested in the following condition on the topological
ring~$\R$, generalizing the conditions~(a\+c) of
Section~\ref{proof-of-main-theorem-secn}:
\begin{enumerate}
\renewcommand{\theenumi}{\alph{enumi}}
\setcounter{enumi}{3}
\item there is a topologically left T\+nilpotent strongly closed
two-sided ideal $\fK\subset\R$ such that the quotient ring $\R/\fK$ is
isomorphic, as a topological ring, to the product
$\prod_{\delta\in\Delta}\T_\delta$ of a family of topological rings
$\T_\delta$, each of which has a base of neighborhoods of zero
consisting of open two-sided ideals and satisfies one of
the conditions~(a), (b), or~(c) of
Section~\ref{proof-of-main-theorem-secn}.
\end{enumerate}
 Here the quotient ring $\R/\fK$ is endowed with the quotient topology
and the product of topological rings $\prod_\delta\T_\delta$ is endowed
with the product topology.
 The following example shows that a topological ring satisfying~(d)
does not need to have a base of neighborhoods of zero consisting of
open two-sided ideals.

\begin{ex} \label{adeles-example}
 Let $\R=\Hom_\boZ(\boQ\oplus\boQ/\boZ,\>\boQ\oplus\boQ/\boZ)^\rop$ be
the opposite ring to the ring of endomorphisms of the abelian group
$\boQ\oplus\boQ/\boZ$, endowed with the topology defined in
Section~\ref{prelim-endomorphism-ring}.
 This topological ring, occurring in tilting theory, was described
in~\cite[Example~8.4]{PS} as the matrix ring
$$
 \R=
 \begin{pmatrix}
 \boQ & \mathbb A^f_\boQ \\
 0 & \widehat\boZ
 \end{pmatrix},
$$
where $\widehat\boZ = \prod_p\boZ_p$ is the product over the prime
numbers~$p$ of the topological rings of $p$\+adic integers $\boZ_p$
endowed with the $p$\+adic topology, $\mathbb A^f_\boQ=\boQ\ot_\boZ
\widehat\boZ$ is the ring of finite adeles of the field of rational
numbers $\boQ$ endowed with the adelic topology, and the field
$\boQ$ itself is endowed with the discrete topology.

 Consider the two-sided ideal $\fK=\mathbb A^f_\boQ\subset\R$.
 Then one has $\fK^2=0$, so this ideal is even (finitely) nilpotent.
 It is also clearly strongly closed in~$\R$.
 The quotient ring $\R/\fK$ is commutative, so it satisfies
the condition~(a).
 The topological ring $\R$ has a base of neighborhoods of zero
formed by the open right ideals
$$
 \begin{pmatrix}
 0 & r\widehat\boZ \\
 0 & n\widehat\boZ
 \end{pmatrix}
 \subset\R,
$$
where $r\in\boQ_{>0}$ and $n\in\boZ_{>0}$ are an arbitrary positive
rational number and a positive integer.
 But every open two-sided (or even left) ideal in $\R$ contains $\fK$,
so such ideals do not form a base of neighborhoods of zero.
\end{ex}

\begin{ex} \label{backwards-morphisms-example}
 Let $\alpha$~be an ordinal, and let $(M_i)$ be an $\alpha$\+indexed
sequence of left modules over an associative ring~$R$.
 Assume that all the morphisms between the $R$\+modules $M_i$ go
backwards, that is, $\Hom_R(M_i,M_j)=0$ for all $0\le i<j<\alpha$.
 Let $\T_i=\Hom_R(M_i,M_i)^\rop$ be topological rings opposite to
the endomorphism rings of the $R$\+modules $M_i$, and let
$\S=\Hom_R(M,M)^\rop$ be the topological ring opposite to the ring
of endomorphisms of the $R$\+module $M=\bigoplus_{i<\alpha}M_i$.

 Then there is a natural surjective morphism of topological rings
$p\:\S\rarrow\prod_{i\in\alpha}\T_i$.
 Set $\fK=\ker(p)\subset\S$; then $\fK$ is a strongly closed two-sided
ideal in $\S$ and the topological quotient ring $\S/\fK$ is isomorphic
to the topological product $\prod_{i\in\alpha}\T_i$.
 Moreover, the ideal $\fK$ is topologically left T\+nilpotent, because
for every element $b\in M$ and any sequence of endomorphisms $a_1$,
$a_2$, $a_3$,~\dots~$\in\fK$ one has $ba_1a_2\dotsm a_n=0$ in $M$
for $n$ large enough (as one easily shows using K\H onig's lemma).

 Thus if for every $i<\alpha$ the topological ring $\T_i$ has a base of
neighborhoods of zero consisting of open two-sided ideals and satisfies
one of the conditions~(a), (b), or~(c) of
Section~\ref{proof-of-main-theorem-secn},
then the topological ring $\S$ satisfies the condition~(d).
 Moreover, if for every index~$i$ the topological ring $\T_i$
satisfies the condition~(d), then so does the topological ring $\S$
(as we will see below in Lemma~\ref{class-d-closure-properties}(b)).
\end{ex}

\begin{lem} \label{discrete-quotients-left-perfect-lemma}
 \textup{(a)} Let $(R_\gamma)_{\gamma\in\Gamma}$ be a family of topological
rings.  Then all the discrete quotient rings of the topological ring
$R=\prod_{\gamma\in\Gamma}R_\gamma$ are left perfect if and only if all
the discrete quotient rings of the topological rings $R_\gamma$,
$\gamma\in\Gamma$, are left perfect. \par
\textup{(b)} Let $R$ be a separated topological ring, and let
$K\subset R$ be a topologically left T\+nilpotent closed two-sided ideal.
 Then all the discrete quotient rings of the ring $R$ are left
perfect if and only if all the discrete quotient rings of the topological
ring $R/K$ are left perfect.
\end{lem}

\begin{proof}
 Part~(a) holds, because the discrete quotient rings of $R$ are
the finite products of discrete quotient rings of $R_\gamma$, taken
over the finite subsets of the set~$\Gamma$.
 Furthermore, a finite product of left perfect rings is left
perfect.
 Part~(b) follows from its discrete version: if $\overline K$ is a left
T\+nilpotent two-sided ideal in an associative ring $\overline R$ and
the quotient ring $\overline R/\overline K$ is left perfect, then
the ring $R$ is left perfect.
 The latter is obtainable from the characterization of left perfect rings
in~\cite[Theorem~P\,(1)]{Bas} and the discrete version of
Lemma~\ref{T-nilpotent-in-quotients}.
\end{proof}

 The following theorem is our generalization of
Theorem~\ref{main-theorem}.

\begin{thm} \label{generalized-main-theorem}
 Let\/ $\R$ be a complete, separated topological associative ring with
a base of neighborhoods of zero formed by open right ideals.
 Assume that the condition~(d) is satisfied.
 Then the following conditions are equivalent:
\begin{itemize}
\item[(i)] all flat left\/ $\R$\+contramodules have projective covers;
\item[(i$^\flat$)] all Bass flat left\/ $\R$\+contramodules have
projective covers;
\item[(ii)] all left $\R$\+contramodules have projective covers;
\item[(iii)] all flat left\/ $\R$\+contramodules are projective;
\item[(iii$^\flat$)] all Bass flat left\/ $\R$\+contramodules
are projective;
\item[(iv)] $\R$ has a topologically left T\+nilpotent strongly closed
two-sided ideal\/ $\fH$ such that the quotient ring\/ $\R/\fH$ is
topologically isomorphic to a product of simple Artinian discrete
rings endowed with the product topology;
\item[(v)] all descending chains of cyclic discrete right\/
$\R$\+modules terminate; 
\item[(vi)] all discrete quotient rings of\/ $\R$ are left perfect;
\item[(vi$'$)] all discrete quotient rings of\/ $\R/\fK$ are left
perfect.
\end{itemize}
 Conversely, if a complete, separated topological associative ring
with a base of neighborhoods of zero formed by open right ideals
satisfies~(iv), then it also satisfies~(d).
\end{thm}

\begin{proof}
 The first four paragraphs of the proof of Theorem~\ref{main-theorem}
apply in our present context as well.
 Furthermore, as in Theorem~\ref{main-theorem},
the implication (iii$^\flat$)$\,\Longrightarrow\,$(vi) is provided
by Corollary~\ref{discrete-quotients-left-perfect-if-bass-projective},
the implication (i$^\flat$)$\,\Longrightarrow\,$(vi) by
Corollary~\ref{discrete-quotients-left-perfect-if-bass-covers},
the implication (iii$^\flat$)$\,\Longrightarrow\,$(v)
by Proposition~\ref{bass-projective-implies-coperfect} with
Lemma~\ref{coperfectness-reformulated}, and
the implication (v)$\,\Longrightarrow\,$(vi) is easy.
 The implication (vi)\,$\Longrightarrow$\,(vi$'$) is obvious,
and the inverse implication (vi$'$)\,$\Longrightarrow$\,(vi)
is provided by Lemma~\ref{discrete-quotients-left-perfect-lemma}(b).

 The final implication (vi$'$)\,$\Longrightarrow$\,(iv) holds in
the assumption of the condition~(d).
 This one, as well as the converse implication
(iv)\,$\Longrightarrow$\,(d), are provided by the following proposition.
 In other words, the proposition below shows that (iv)~is equivalent
to the combination of~(vi$'$) and~(d). 
\end{proof}

\begin{prop} \label{d-implies-section-8-prop}
 Let\/ $\R$ be a complete, separated topological ring with a base of
neighborhoods of zero formed by open right ideals.
 Suppose that there exists an ideal\/ $\fK\subset\R$ such that
the condition~(d) is satisfied and all the discrete quotient rings of
the ring\/ $\R/\fK$ are left perfect.
 Then there exists an ideal\/ $\fH\subset\R$ satisfying~(iv).
 Conversely, if an ideal\/ $\fH\subset\R$ satisfies~(iv), then
the same ideal\/ $\fK=\fH$ also satisfies~(d).
\end{prop}

\begin{proof}
 The converse assertion is obvious: any simple Artinian ring endowed
with the discrete topology satisfies both~(b) and~(c), so a product of
such rings is a product of topological rings satisfying~(b) and~(c).
 To prove the direct implication, suppose that $\fK\subset\R$ is
an ideal satisfying~(d) such that all the discrete quotient rings of
$\R/\fK$ are left perfect.
 Then we have $\R/\fK\cong\prod_{\delta\in\Delta}\T_\delta$, so any
discrete quotient ring of $\T_\delta$ is at the same time a discrete
quotient ring of~$\R/\fK$.

 Applying Proposition~\ref{properfect-prop} to the topological ring
$\T_\delta$, we conclude that there exists a topologically left
T\+nilpotent strongly closed two-sided ideal $\J_\delta\subset\T_\delta$
such that the quotient ring $\S_\delta=\T_\delta/\J_\delta$ is
topologically isomorphic to a product of discrete simple Artinian rings,
$\S_\delta\cong\prod_{\gamma\in\Gamma_\delta}S_{\gamma}$.
 According to the discussion in the beginning of
Section~\ref{products-secn}, it follows that $\J=\prod_\delta\J_\delta$
is a topologically left T\+nilpotent strongly closed two-sided ideal
in $\T=\prod_\delta\T_\delta$.
 Furthermore, the topological quotient ring $\T/\J\cong
\prod_{\delta\in\Delta}\S_\delta$ is isomorphic to the topological 
product $\prod_{\gamma\in\Gamma}S_\gamma$ of the discrete simple
Artinian rings $S_\gamma$ over the disjoint union
$\Gamma=\coprod_{\delta\in\Delta}\Gamma_\delta$ of the sets of
indices~$\Gamma_\delta$.

 Now we have a surjective continuous ring homomorphism $\R\rarrow
\R/\fK\cong\T$.
 Let $\fH\subset\R$ be the full preimage of the closed ideal
$\J\subset\T$ under this homomorphism.
 Then the ideal $\fH$ is strongly closed in $\R$ by
Lemma~\ref{strongly-closed-in-quotients}(b), $\fH$~is topologically
left T\+nilpotent by Lemma~\ref{T-nilpotent-in-quotients},
and the topological ring $\R/\fH\cong\T/\J$ is the topological
product of discrete simple Artinian rings~$S_\gamma$.
\end{proof}

 The following lemma shows the class of all topological rings
satisfying~(d) is closed under the operations that were used
to define it.

\begin{lem} \label{class-d-closure-properties}
 \textup{(a)} Let $(\R_\gamma)_{\gamma\in\Gamma}$ be a family of topological
rings satisfying the condition~(d).
 Then the topological ring\/ $\R=\prod_{\gamma\in\Gamma}\R_\gamma$ also
satisfies~(d). \par
\textup{(b)} Let\/ $\R$ be a complete, separated topological ring with
a base of neighborhoods of zero formed by open right ideals, and let\/
$\J\subset\R$ be a topologically left T\+nilpotent strongly closed
two-sided ideal.
 Assume that the topological quotient ring\/ $\R/\J$ satisfies~(d).
 Then the topological ring\/ $\R$ satisfies~(d).
\end{lem}

\begin{proof}
 Similar to the proof of Proposition~\ref{d-implies-section-8-prop}.
 The proof of part~(a) is based on the discussion in the beginning of
Section~\ref{products-secn}, while the proof of part~(b) uses
Lemmas~\ref{strongly-closed-in-quotients}(b)
and~\ref{T-nilpotent-in-quotients}.
\end{proof}

\begin{rem}
 More generally, for any complete, separated topological ring $\R$
with a base of neighborhoods of zero formed by open right ideals,
the four conditions (i), (ii), (iii), and~(iv) are equivalent to
each other.
 This is a result of the paper~\cite[Theorem~14.1]{PS3}.
 Furthermore, in the same (full) generality, the two conditions
(i$^\flat$) and~(iii$^\flat$) are equivalent to each other
by~\cite[Corollary~3.10]{BPS}.

 A conjecture claiming all the conditions (i), (i$^\flat$), (ii),
(iii), (iii$^\flat$), (iv), and~(v) to be equivalent to each other is
formulated in the paper~\cite{PS3} as~\cite[Conjecture~14.3]{PS3}.
 This conjecture is true for all complete, separated topological rings
$\R$ with a \emph{countable} base of neighborhoods of zero consisting
of open right ideals~\cite[Theorem~14.8]{PS3}.
\end{rem}

\end{document}